\documentclass[11pt,a4paper,oneside]{amsart}
\usepackage[latin9]{inputenc}
\usepackage{mathtools}
\usepackage{amstext}
\usepackage{amsthm}
\usepackage{amssymb}
\usepackage{esint}
\usepackage[unicode=true,pdfusetitle,
 bookmarks=true,bookmarksnumbered=false,bookmarksopen=false,
 breaklinks=false,pdfborder={0 0 1},backref=false,colorlinks=false]
 {hyperref}

\makeatletter

\pdfpageheight\paperheight
\pdfpagewidth\paperwidth

\numberwithin{equation}{section}
\numberwithin{figure}{section}
\theoremstyle{plain}
\newtheorem{thm}{\protect\theoremname}[section]
  \theoremstyle{remark}
  \newtheorem{rem}[thm]{\protect\remarkname}
  \theoremstyle{plain}
  \newtheorem{prop}[thm]{\protect\propositionname}
  \theoremstyle{plain}
  \newtheorem{lem}[thm]{\protect\lemmaname}
  \theoremstyle{plain}
  \newtheorem{cor}[thm]{\protect\corollaryname}

\renewcommand{\hat}[1]{\widehat{#1}}
\renewcommand{\tilde}[1]{\widetilde{#1}}

\makeatother

  \providecommand{\corollaryname}{Corollary}
  \providecommand{\lemmaname}{Lemma}
  \providecommand{\propositionname}{Proposition}
  \providecommand{\remarkname}{Remark}
\providecommand{\theoremname}{Theorem}

\begin{document}

\title[Scattering for Defocusing gBO]{Scattering for Defocusing generalized Benjamin-Ono Equation in the
Energy Space $H^{\frac{1}{2}}(\mathbb{R})$}

\author{\noindent Kihyun Kim}

\email{khyun1215@kaist.ac.kr}

\address{Department of Mathematical Sciences, Korea Advanced Institute of
Science and Technology, 291 Daehak-ro, Yuseong-gu, Daejeon 34141,
Korea}

\author{\noindent Soonsik Kwon}

\email{soonsikk@kaist.edu}

\address{Department of Mathematical Sciences, Korea Advanced Institute of
Science and Technology, 291 Daehak-ro, Yuseong-gu, Daejeon 34141,
Korea}


\begin{abstract}
We prove the scattering for the defocusing generalized Benjamin-Ono
equation in the energy space $H^{\frac{1}{2}}(\mathbb{R})$. We first
establish the monotonicity formula that describes the unidirectional
propagation. More precisely, it says that the center of energy moves
faster than the center of mass. This type of monotonicity was first
observed by Tao \cite{Tao2007DiscCont} in the defocusing gKdV equations. 

We use the monotonicity in the setting of compactness-contradiction
argument to prove the large data scattering in the energy space $H^{\frac{1}{2}}(\mathbb{R})$.
On the way, we extend critical local theory of Vento \cite{Vento2010}
to the subcritical regime. Indeed, we obtain subcritical local theory
and global well-posedness in the energy space.
\end{abstract}

\maketitle
\tableofcontents{}

\section{Introduction}

We consider the \emph{defocusing generalized Benjamin-Ono equation}
\eqref{eq:gBO}: 
\begin{align}
\begin{cases}
\partial_{t}u+\mathcal{H}\partial_{xx}u+\partial_{x}(u^{k+1})=0, & (t,x)\in\mathbb{R}\times\mathbb{R},\\
u(0,x)=u_{0}(x),
\end{cases} & \tag{gBO}\label{eq:gBO}
\end{align}
with $u:\mathbb{R}\times\mathbb{R}\to\mathbb{R}$ a real-valued function,
$k$ an even natural number $\geq4$, $u_{0}$ an initial data in
homogeneous or inhomogeneous Sobolev space. We denote by $\mathcal{H}$
the \emph{Hilbert transform} on $\mathbb{R}$, which acts on Schwartz
functions $f$ by 
\[
\mathcal{H}f=\Big(\frac{1}{\pi}\mathrm{p.v.}\frac{1}{x}\Big)\ast f.
\]
Equivalently, $\mathcal{H}$ is a Fourier multiplier operator with
multiplier $m(\xi)=-i\mathrm{sgn}(\xi)$:
\[
\widehat{\mathcal{H}f}(\xi)=-i\mathrm{sgn}(\xi)\hat{f}(\xi).
\]
We let the nonlinearity $F(u)$ by $-\partial_{x}(u^{k+1})$, and
rewrite \eqref{eq:gBO} in the form
\[
\partial_{t}u+\mathcal{H}\partial_{xx}u=F(u).
\]
The linear part of \eqref{eq:gBO} is called as the \emph{linear Benjamin-Ono
equation}. If the nonlinearity has the opposite sign, we call the
equation \emph{focusing}. The equation \eqref{eq:gBO} is in fact
a generalization of the Benjamin-Ono equation
\begin{equation}
\partial_{t}u+\mathcal{H}\partial_{xx}u+\partial_{x}(u^{2})=0\tag{BO}\label{eq:BO}
\end{equation}
in its power of nonlinearity. The \eqref{eq:gBO} with $k=1$ is the
above \eqref{eq:BO}, and the case when $k=2$ is called as defocusing
modified Benjamin-Ono equation (mBO). It is known that \eqref{eq:BO}
is completely integrable \cite{FokasAblowitz1983}.

The equation \eqref{eq:gBO} has scaling invariance. For any $\lambda>0$
and a solution $u(t,x)$ to \eqref{eq:gBO} with initial data $u_{0}(x)$,
$u_{\lambda}(t,x)\coloneqq\lambda^{\frac{1}{k}}u(\lambda^{2}t,\lambda x)$
is also a solution with initial data $\lambda^{\frac{1}{k}}u_{0}(\lambda x)$.
Thus, \eqref{eq:gBO} is $\dot{H}^{s_{k}}$-critical with 
\[
s_{k}=\frac{1}{2}-\frac{1}{k},
\]
in the sense that scalings preserve $\dot{H}^{s_{k}}$ norm of the
initial data. 

Moreover, we have mass and energy conservation laws.\footnote{In fact, there is another (formal) conservation law: 
\[
\int_{\mathbb{R}}u(t,x)dx=\int_{\mathbb{R}}u(0,x)dx
\]
for a solution $u$ to \eqref{eq:gBO}, but this will not be used
throughout the paper. Note that it is not positive-definite and may
not even defined for $H^{\frac{1}{2}}$ solutions.} The $L^{2}$\emph{-mass} (in short, \emph{mass}) and \emph{energy}
of a function $u(t,x)$ is defined by
\begin{align*}
M(u)(t) & \coloneqq\int_{\mathbb{R}}u^{2}(t,x)dx,\\
E(u)(t) & \coloneqq\int_{\mathbb{R}}[\frac{1}{2}u\mathcal{H}u_{x}+\frac{1}{k+2}u^{k+2}](t,x)dx.
\end{align*}
Whenever $u$ is a smooth solution, mass and energy are conserved:
\[
M(u)(t)=M(u)(0)\qquad\text{and}\qquad E(u)(t)=E(u)(0).
\]
Henceforth, we drop the time variable $t$ and write $M(u)$ and $E(u)$
to denote mass and energy of a solution $u$, respectively. If $u(t)\in H^{\frac{1}{2}}(\mathbb{R})$,
then $E(u)$ is well-defined and finite in view of Plancherel's theorem.
So we regard that $H^{\frac{1}{2}}(\mathbb{R})$ is the \emph{energy
space}.

\ 

The local theory of Benjamin-Ono type equations have been intensively
studied. For readers who are interested in local theory of \eqref{eq:BO}
and (mBO), we refer to a recent survey by Ponce \cite{Ponce2016survey}.
We now want to pick up some works for \eqref{eq:gBO} with $k\geq4$.
At first, by the energy method, I\'orio \cite{Iorio1986CPDE} showed
local well-posedness (LWP) in $H^{s}$ with $s>\frac{3}{2}$. Using
the method of gauge transform, Molinet and Ribaud \cite{MolinetRibaud2004}
showed LWP in $H^{\frac{1}{2}}$ for $k\geq5$ and in $H^{\frac{1}{2}+}$
for $k=4$. Moreover, they showed small data global well-posedness
in the critical space $\dot{H}^{s_{k}}$. For $k=4$, the first large
data critical LWP was obtained by Burq and Planchon \cite{BurqPlanchon2006}.
Finally, Vento \cite{Vento2010} obtained the critical LWP for all
$k\geq4$ by modifying the linear propagator of \eqref{eq:gBO}.

Throughout the paper, we rely on the critical LWP obtained by Vento,
so we want to state the result in more detail. Denote the linear propagator
of the linear Benjamin-Ono equation by $V(t)$. A function $u:I\times\mathbb{R}\to\mathbb{R}$
is called as a \emph{$\dot{H}^{s_{k}}$ solution} to \eqref{eq:gBO}
if $0\in I$ is a time interval, $u\in C_{J}\dot{H}^{s_{k}}\cap L_{x}^{k}L_{J}^{\infty}\cap\dot{X}_{J}^{s_{k}}$
for any compact subinterval $J\subseteq I$, (see Notations and Section
\ref{subsec:review of Vento's work} for definitions of these function
spaces) and $u$ satisfies the Duhamel formula
\[
u(t)=V(t)u_{0}+\int_{0}^{t}V(t-s)F(u(s))ds
\]
for all $t\in I$. Replacing $\dot{H}^{s_{k}}$ and $\dot{X}_{I}^{s_{k}}$
by $H^{s}$ and $X_{I}^{s}$ respectively, we similarly define $H^{s}$
\emph{solutions} to \eqref{eq:gBO}. A solution is called \emph{global}
if $I=\mathbb{R}$. We now state the critical LWP of \eqref{eq:gBO}.
\begin{thm}[Critical local theory, \cite{Vento2010}]
\label{thm:LWP}Let $k\geq4$.\\
1. For any $u_{\ast}$ in $\dot{H}^{s_{k}}$, there exists $T=T(u_{\ast})>0$
and $r=r(u_{\ast})>0$ such that for any initial data $u_{0}\in B(u_{\ast};r)$,
there exists a unique $\dot{H}^{s_{k}}$ solution $u$ in $L_{T}^{\infty}\dot{H}_{x}^{s_{k}}\cap L_{x}^{k}L_{T}^{\infty}\cap\dot{X}_{T}^{s_{k}}$.
Moreover, the solution $u$ indeed lies in $\dot{Z}_{T}^{s_{k}}$
and the solution map is locally Lipschitz.\\
2. For any $s\geq s_{k}$, the $H^{s}$-version of the result holds.
That is, we can replace $L_{T}^{\infty}\dot{H}_{x}^{s_{k}}\cap L_{x}^{k}L_{T}^{\infty}\cap\dot{X}_{T}^{s_{k}}$,
$\dot{H}^{s_{k}}$, and $\dot{Z}_{T}^{s_{k}}$ by $L_{T}^{\infty}H_{x}^{s}\cap L_{x}^{k}L_{T}^{\infty}\cap X_{T}^{s}$,
$H^{s}$, and $Z_{T}^{s}$.
\end{thm}
In Theorem \ref{thm:LWP}, the lifespan of a solution depends on the
profile of its initial data even for $s>s_{k}$. However, a slight
modification of the proof of Theorem \ref{thm:LWP} yields that for
given $u_{0}\in H^{s}$ with $s>s_{k}$, we can construct a $\dot{H}^{s_{k}}$
solution $u\in C_{T}\dot{H}^{s_{k}}$ whose lifespan $T$ depends
only on $\|u_{0}\|_{H^{s}}$. Moreover, we use persistence of regularity
(Proposition \ref{prop:persistence of regularity}) to guarantee that
this $\dot{H}^{s_{k}}$ solution is indeed a $H^{s}$ solution. In
particular, we obtain subcritical local well-posedness. When $s\geq\frac{1}{2}$,
by mass/energy conservation, we obtain the global well-posedness of
\eqref{eq:gBO} in the energy space $H^{\frac{1}{2}}$. This is our
first result.
\begin{thm}[Subcritical local theory and global well-posedness]
\label{thm:subcritical LWP, GWP}\ Let $k\geq4$.\\
1. (Subcritical LWP) Assume that $s>s_{k}$. For any $R>0$, there
exists $T=T(R)>0$ such that if the initial data $u_{0}$ satisfies
$\|u_{0}\|_{H^{s}}<R$, then there exists a unique $H^{s}$ solution
$u$ in $L_{T}^{\infty}H_{x}^{s}\cap L_{x}^{k}L_{T}^{\infty}\cap X_{T}^{s}$.
Moreover, the solution $u$ indeed lies in $Z_{T}^{s}$ and the solution
map is continuous.\\
2. (GWP and conservation laws) Assume that $s\geq\frac{1}{2}$. For
any $T>0$ and initial data $u_{0}\in H^{s}$, there exists a unique
$H^{s}$ solution $u$ in $L_{T}^{\infty}H_{x}^{s}\cap L_{x}^{k}L_{T}^{\infty}\cap X_{T}^{s}$.
Moreover, the solution $u$ indeed lies in $Z_{T}^{s}$ and the solution
map is continuous. Furthermore, we have both mass and energy conservation.
\end{thm}
Beyond the well-posedness theory, it is of great interest to study
long-time dynamics of the solutions. It is widely believed that for
the defocusing equations, the scattering holds.\footnote{For defocusing Benjamin-Ono type equations, by heuristic observations
of time decay, it is believed that the linear scattering holds for
$k>2$.} (however, see also \cite{Tao2016AnalPDE,Tao2018AnalPDE} for supercritical
equations.) The defocusing nature in general forces solutions to disperse
in the physical space. As the nonlinearity contains the power of $u$,
its effect becomes much weaker than the linear evolution. As a result,
the linear evolution dominates the dynamics of $u$ and $u$ resembles
some linear solution asymptotically. In mathematical terms, we say
that a solution $u$ \emph{scatters forward} (resp., \emph{backward})
\emph{in time} if there exists a \emph{scattering state} $u_{\pm}$
satisfying 
\[
\lim_{t\to\pm\infty}\|V(t)u_{\pm}-u(t)\|_{\dot{H}^{s_{k}}}=0.
\]
It is well-known that finiteness of the solution norm implies the
scattering. We now present our main result, that is, any $H^{\frac{1}{2}}$
solutions to \eqref{eq:gBO} scatter.
\begin{thm}[Scattering for defocusing gBO]
\label{thm:scattering}For $k>4$,\footnote{Recall that $k$ is an even number, so $k=6,8,\dots$.}
any  $H^{\frac{1}{2}}$-solution to \eqref{eq:gBO} scatters both
forward and backward in time. Moreover, there exists a function $L:[0,\infty)\to[0,\infty)$
such that 
\[
\|u\|_{\dot{X}_{t}^{s_{k}}}+\|u\|_{L_{x}^{k}L_{t}^{\infty}}\leq L(M(u)+E(u)).
\]
\end{thm}
There have been a number of results addressing scattering for semilinear
dispersive equations. In this paper, we will use concentration compactness
argument. It originates from Lions \cite{Lions1984AIHP1,Lions1984AIHP2},
and was first used in the dispersive equation by Bahouri and G\'erard
\cite{BahouriGerard1999AJM}. For the scattering problem, Kenig and
Merle \cite{Kenig-Merle2006Invent,KenigMerle2008Acta} used it for
the focusing energy-critical nonlinear Schr\"odinger and wave equations.\footnote{As Kenig and Merle dealt with focusing equations, they showed scattering
when solutions have energy less than that of the ground state.}In fact, their argument have a great generality and applied to many
other equations. There are too extensive research on this subject
to list them here. We refer, for example, \cite{Dodson2012JAMS,Dodson2015AdvMath,Killip-Tao-Visan2009JEMS,Killip-Visan2013Clay,HolmerRoudenko2008}
for the nonlinear Schr\"odinger equations (NLS). For the mass-critical
defocusing generalized Korteweg-de Vries equation (gKdV), Dodson \cite{Dodson2017AnnPDE}
proved the scattering. Recently, his argument is extended to supercritical
defocusing (gKdV) \cite{FarahLinaresPastorVisciglia2017}.

The our case \eqref{eq:gBO} shares a similar nature with the defocusing
(gKdV) 
\[
\partial_{t}u+\partial_{xxx}u=\partial_{x}(|u|^{p-1}u)
\]
in their nonlinearities and unidirectional propagation.\footnote{Unidirectional propagation means that radiation waves propagate in
one direction. See Section \ref{sec:Monotonicity-Formula} for details.} It is expected that qualitative asymptotic behaviors of solutions
are similar. In particular, we show that Tao's \emph{monotonicity
formula} \cite{Tao2007DiscCont} still holds in \eqref{eq:gBO}:
\[
\partial_{t}(\langle x\rangle_{E}-\langle x\rangle_{M})>0.
\]
Here, $\langle x\rangle_{M}$ and $\langle x\rangle_{E}$ denote mass
and energy center, respectively. For details, see Section \ref{sec:Monotonicity-Formula}.
Outline of proof of Theorem \ref{thm:scattering} closely follows
that of (gKdV) \cite{Dodson2017AnnPDE}. However, we encounter several
difficulties in \eqref{eq:gBO}. They are due to weaker dispersion,
technical issues from the Hilbert transform, and their consequences,
for example, trickier local theory. We will explain these issues in
detail.

\subsection*{Outline of the Proof and Ideas}

In this subsection, we explain our scheme of the proof and what difficulties
arise in the setting of \eqref{eq:gBO}. As we already mentioned,
the subcritical well-posedness (Theorem \ref{thm:subcritical LWP, GWP})
follows by extending Vento's argument \cite{Vento2010}; see Section
\ref{sec:WP theory} for details. Henceforth, we focus on the scattering
(Theorem \ref{thm:scattering}).

We now explain how compactness-contradiction argument goes. Suppose
that Theorem \ref{thm:scattering} fails. In the first step, we show
the existence of the critical element, which does not scatter (both
forward and backward in time). In this step, we start with the linear
profile decomposition and then obtain nonlinear profiles. The first
goal is to show that sum of nonlinear profiles becomes an approximate
solution. This requires a long-time perturbation theory. As a result,
the extremizing sequence converges to a critical element. Moreover,
this critical element stays in a compact set modulo symmetries. In
the next step, we use (a truncated version of) monotonicity formula
to show that such a solution cannot exist. We now explain details
and difficulties step by step.

\ 

The linear profile decomposition is used to obtain compactness property.
To illustrate this in our case, we consider the local smoothing estimate
\[
\|V(t)u_{0}\|_{L_{x}^{k}L_{t}^{\infty}}\lesssim\|u_{0}\|_{H^{\frac{1}{2}}}.
\]
This estimate has two non-compact symmetries: spatial translation
and time translations. The linear profile decomposition says that
lack of compactness of the estimate essentially comes from these symmetries.

Its rigorous statement and proof are fairly standard, we include the
proof in Appendix \ref{sec:proof of linear profile decomposition}.
However, the usual method cannot take care of the case when $k=4$.
This is because $L_{x}^{4}L_{t}^{\infty}$ is indeed the endpoint
exponents in the local smoothing estimates, so it cannot be obtained
by interpolating other estimates. So we can only prove the case when
$k>4$, where $L_{x}^{k}L_{t}^{\infty}$ norm can be interpolated
with $L_{x}^{4}L_{t}^{\infty}$ and $L_{t,x}^{\infty}$. This is the
only point where we should assume $k>4$ in Theorem \ref{thm:scattering}.

\ 

Each profile from the linear profile decomposition gives rise global
nonlinear solutions, so called nonlinear profiles. To guarantee that
the sum of nonlinear profiles approximates a nonlinear solution to
\eqref{eq:gBO}, we need a long-time perturbation theory. In \eqref{eq:gBO},
we do not obtain a standard and strong long-time perturbation as in
other contexts. This is principally due to a \emph{delicate} local
theory \cite{Vento2010}. Its difficulty is well-explained in Vento
\cite{Vento2010}, but we take it for beginning. At first, one may
try a naive estimate
\[
\|D_{x}^{s_{k}-\frac{1}{2}}\partial_{x}(u^{k+1})\|_{L_{x}^{1}L_{T}^{2}}\lesssim\|D_{x}^{s_{k}+\frac{1}{2}}u\|_{L_{x}^{\infty}L_{T}^{2}}\|u\|_{L_{x}^{k}L_{T}^{\infty}}^{k}.
\]
Here, we do not know whether we can use Leibniz rule at endpoint.
Moreover, we cannot guarantee that smallness of $T$ implies smallness
of $L_{x}^{k}L_{T}^{\infty}$ norm because of $L^{\infty}$ factor.
Vento resolves this difficulty by extracting out a dangerous high-low
interaction in the nonlinearity, approximating it by suitable linear
term $\pi(u_{0},u)$, and rewrite \eqref{eq:gBO} in distorted form:
\[
(\partial_{t}+\mathcal{H}\partial_{xx})u+\pi(u_{0},u)=[\pi(u_{0},u)-\pi(u,u)]+g(u).
\]

In other contexts, the long-time perturbation can be obtained in a
general form by concatenating local theory on short time intervals.
As an explicit example, in the (gKdV) setting \cite{KillipVisanKwonShao2012},
one uses $L_{x}^{5}L_{t}^{10}$ norm as a solution norm. If an interval
$I$ (possibly unbounded) is a disjoint union of $I_{j}$'s, then
by Minkowski's inequality,
\[
\|u\|_{L_{x}^{5}L_{I}^{10}}=\|u\|_{L_{x}^{5}\ell_{j}^{10}L_{I_{j}}^{10}}\geq\|u\|_{\ell_{j}^{10}L_{x}^{5}L_{I_{j}}^{10}}=\Big(\sum_{j}\|u\|_{L_{x}^{5}L_{I_{j}}^{10}}^{10}\Big)^{\frac{1}{10}}.
\]
This says that one can subdivide $I$ into finitely many subintervals
$I_{j}$ such that each $L_{x}^{5}L_{I_{j}}^{10}$ norm of $u$ is
sufficiently small. In our case, however, we use solution norms $L_{x}^{k}L_{t}^{\infty}$
and $\dot{X}_{t}^{s_{k}}$, where both contains $L^{\infty}$ part.
Thus, due to nature of $L^{\infty}$, the usual subdivision argument
does not guarantee smallness of our solution norms.

Recalling the reason why we need long-time perturbation in the proof,
we do not have to obtain its full power of perturbation. It only suffices
to somehow obtain nonlinear version of the profile decomposition.
For this purpose, we prove perturbation for a restricted class of
approximate solutions, namely those explicitly constructed from the
profiles obtained in the linear profile decomposition.

To obtain nonlinear profile decomposition, we use three kinds of perturbation
theory, the small data global theory, local theory at time zero, and
local theory at time $\pm\infty$. Note that the small data global
theory and local theory at time $\pm\infty$ do not require modifying
the linear propagator. This helps us to obtain perturbation lemmas
at time $\pm\infty$ and safely ignore small data profiles constructed
in profile decomposition. Thus, it only suffices to obtain perturbation
on the remaining compact interval, on which we can concatenate local
theory.

\ 

Since the orbit $\{u(t)\}_{t\in\mathbb{R}}$ is precompact in $H^{\frac{1}{2}}$
modulo spatial translation, we have only soliton-like enemies. See
Remark \ref{rem:4.23}. The monotonicity formula is the main tool
to exclude soliton-like solutions. In (gKdV), Tao \cite{Tao2007DiscCont}
established the monotonicity formula:
\[
\partial_{t}(\langle x\rangle_{E}-\langle x\rangle_{M})>0,
\]
where $\langle x\rangle_{M}$ and $\langle x\rangle_{E}$ denote mass
and energy center of a solution, respectively. One can rewrite the
monotonicity as an interaction form:
\[
\partial_{t}\int_{\mathbb{R}\times\mathbb{R}}(y-x)\rho(t,x)e(t,y)dxdy>0,
\]
where $\rho$ and $e$ are mass density and energy density, respectively.
See Section \ref{sec:Monotonicity-Formula}. Tao's monotonicity formula
is based on the following basic observations. Note that the group
velocity $\frac{d}{d\xi}\omega(\xi)=-3\xi^{2}$ is sign definite and
the higher frequency pieces travel faster. Moreover, the energy density
is more weighted on higher frequency. Thus, the monotonicity formula
quantitatively gives a clue that solutions disperse in the physical
space. \eqref{eq:gBO} also has the same property as the group velocity
is $2|\xi|$.

Tao's proof is nontrivial but surprisingly elementary. His argument
seems to be applicable in \eqref{eq:gBO}. We can closely follow his
argument, but there are technical issues arising from the Hilbert
transform. More precisely, we have to show 
\[
s\coloneqq\int_{\mathbb{R}}u^{k+1}\mathcal{H}u_{x}\geq0.
\]
In (gKdV), the corresponding statement is just $-\int_{\mathbb{R}}u^{k+1}u_{xx}\geq0$,
which is obvious. To show positivity of $s$ in \eqref{eq:gBO}, we
use finite-dimensional approximation and the Lagrange multiplier method.
We reduce it to the case of the circle $\mathbb{T}$.\footnote{As $u$ and $\mathcal{H}u$ cannot be localized simultaneously in
the physical space, we work instead on the frequency space, where
both $u$ and $\mathcal{H}u$ can be localized.} By a density argument in the frequency space, we reduce it to the
finite-dimensional case $\mathbb{C}^{2N}$. Furthermore, by homogeneity
of the functional $s$, we reduce it to the sphere case $S^{2N-1}$,
which is compact. Then, in light of Lagrange multiplier, we can obtain
a useful relation what an extremizer should satisfy. We then use Pohozaev
type argument to show that $s$ is nonnegative.

Finally, to use monotonicity in practice, because we assume that $u$
is a $H^{\frac{1}{2}}$ solution, we consider the localized interaction
functional
\[
M(t)=\int_{\mathbb{R}\times\mathbb{R}}\Phi(y-x)\rho(t,x)e(t,y)dxdy,
\]
where $\Phi(x)=\Phi_{R}(x)$ approximates $x$ but remains bounded
by $R$ for large $R>0$. The truncation creates a number of errors.
In (gKdV), Dodson \cite{Dodson2017AnnPDE} utilized it in this concrete
form (see also \cite{OgawaTsutsumi1991JDE}). In our case, we want
to point out two technical difficulties. At first, as opposed to other
defocusing equations, the energy density 
\[
e=\frac{1}{2}u\mathcal{H}u_{x}+\frac{1}{k+2}u^{k+2}
\]
is not pointwisely nonnegative. Due to this, one should interpret
the integral involving $e$ in view of the Parseval identity to estimate
various terms by the $H^{\frac{1}{2}}$ norm. Moreover, in error estimates,
we run into terms containing derivatives of $\Phi$ and $\mathcal{H}$
that should be small as $C(\|u\|_{H^{\frac{1}{2}}})\cdot o_{R}(1)$.
Secondly, as the Hilbert transform does not satisfy a simple Leibniz
rule, there are many commutator terms in computation such as $[\mathcal{H},Q]$
for some smooth function $Q$. The calculations and error estimates
are involved.

\subsection*{\label{subsec:Notations}Notations}

We shall use the notation $A\lesssim B$ frequently. We say $A\lesssim B$
when there is some implicit constant $C$ that does not depend on
$A$ and $B$ satisfying $A\leq CB$. For some parameter $r$, we
write $A\lesssim_{r}B$ if the implicit constant $C$ depends on $r$.
In this paper, $k$-dependence will be ignored; we shall abbreviate
$A\lesssim_{k}B$ by $A\lesssim B$.

The \emph{Fourier transform} of a function $f(x)$ is defined by
\[
\hat{f}(\xi)\coloneqq\int_{\mathbb{R}}f(x)e^{-i\xi x}dx.
\]
We define the \emph{Sobolev norms} for $s\in\mathbb{R}$ as
\[
\|f\|_{H^{s}}\coloneqq\|\langle D_{x}\rangle^{s}f\|_{L^{2}}\quad\text{and}\quad\|f\|_{\dot{H}^{s}}\coloneqq\|D_{x}^{s}f\|_{L^{2}},
\]
where $D_{x}^{s}$ and $\langle D_{x}\rangle^{s}$ are the Fourier
multiplier operators with multiplier $|\xi|^{s}$ and $\langle\xi\rangle^{s}\coloneqq(1+|\xi|^{2})^{\frac{s}{2}}$,
respectively.

We let $Q_{j}$ be the \emph{Littlewood-Paley projection} to the frequency
$\sim2^{j}$ in $x$-variable. Given a function $u$, we write
\begin{align*}
u_{\ll j} & =Q_{\ll j}u=Q_{<j-J}u,\\
u_{\sim j} & =Q_{\sim j}u=\sum_{|r-j|\leq J}Q_{r}u
\end{align*}
for some $J$ large. We remark that one should choose $J$ depending
on $\|u_{0}\|_{\dot{H}^{s_{k}}}$ to obtain linear estimates for the
distorted equation. In later sections, we still use same $J$ for
global nonlinear solutions. Although $\dot{H}^{s_{k}}$ norm may change,
but it is bounded by $H^{\frac{1}{2}}$. Due to conservation laws,
our global solution has uniform $H^{\frac{1}{2}}$ bound. See \cite[Lemma 3.5 and Proposition 3.2]{Vento2010}
or Section \ref{subsec:review of Vento's work}.

We use various \emph{mixed Lebesgue norms}. For $1\leq q,r\leq\infty$,
an interval $I\subseteq\mathbb{R}$, and a Banach space $X$, we write
\begin{align*}
\|f(t)\|_{L_{I}^{q}X} & \coloneqq\Big\|\|f(t)\|_{X}\Big\|_{L_{t}^{q}(I)},\\
\|G(t,x)\|_{L_{x}^{q}L_{I}^{r}} & \coloneqq\Big\|\|G\|_{L_{t}^{r}(I)}\Big\|_{L_{x}^{q}(\mathbb{R})}.
\end{align*}
In the situation of $q=r$, by Fubini's theorem, we abbreviate $L_{x}^{q}L_{t}^{q}$
and $L_{t}^{q}L_{x}^{q}$ by $L_{t,x}^{q}$. If $t\in I\mapsto G(t)\in X$
is continuous and bounded, we write $G\in C_{I}X$. If $I=\mathbb{R}$,
$I=[-T,T]$, $I=[T,+\infty)$, or $I=(-\infty,-T]$, we replace subscript
$I$ by $t$, $T$, $T+$, or $T-$, respectively.

We now define \emph{Besov-type spaces} by following \cite{Vento2010}.
For $s\in\mathbb{R}$, $p,q,r\in[1,\infty]$, and an interval $I\subseteq\mathbb{R}$,
we define
\[
\|f\|_{\dot{\mathcal{B}}_{p}^{s,r}(L_{I}^{q})}\coloneqq\Big(\sum_{j\in\mathbb{Z}}[2^{js}\|Q_{j}f\|_{L_{x}^{p}L_{I}^{q}}]^{r}\Big)^{\frac{1}{r}},
\]
where $Q_{j}$ is the Littlewood-Paley projection to the frequency
$\sim2^{j}$ in $x$-variable. In other words, $\dot{\mathcal{B}}_{p}^{s,r}(L_{I}^{q})$
sums up each $\|Q_{j}f\|_{L_{x}^{p}L_{I}^{q}}$ in $\ell_{j}^{r}$-sense. 

\subsection*{Organization of the Paper}

In Section \ref{sec:Monotonicity-Formula}, we establish the monotonicity
formula for \eqref{eq:gBO}. In Section \ref{sec:WP theory}, we prove
subcritical well-posedness in $H^{s}$ with $s>s_{k}$ and global
well-posedness in $H^{s}$ with $s\geq\frac{1}{2}$. In Section \ref{sec:existence of critical elt},
we derive the scattering criterion, linear profile decomposition,
and nonlinear profile decomposition. We then prove existence of the
critical element if Theorem \ref{thm:scattering} fails. In Section
\ref{sec:precllusion of a.p. sol}, we use the monotonicity formula
to show that such critical element cannot exist, concluding Theorem
\ref{thm:scattering}.

\subsection*{Acknowledgements}

We appreciate Terry Tao for helpful discussion while his visiting
at KAIST. The authors are partially supported by Samsung Science \&
Technology Foundation BA1701-01.

\section{\label{sec:Monotonicity-Formula}Monotonicity Formula}

One of the crucial steps toward asymptotic control of the global solutions
is a decay estimate, or monotonicity formula. Though most canonical
nonlinear dispersive equations are Hamiltonian systems, it is well
known that certain monotonicity phenomena occur and are formulated
as for example, Morawetz inequality, Virial inequality, and so on.

In this section, we derive a monotonicity formula for the defocusing
\eqref{eq:gBO} as well as linear Benjamin-Ono flow. It is similar
to the monotonicity formula in the defocusing generalized Korteweg-de
Vries equation \eqref{eq:gKdV}. Tao constructed a monotonicity formula
\cite{Tao2007DiscCont} for \eqref{eq:gKdV}. Tao's argument is ingenious
but elementary, and we observe that his argument similarly works for
\eqref{eq:gBO}. However, we will meet a nontrivial technical difficulty,
which does not appear in \eqref{eq:gKdV}.

Firstly, we recall Tao's monotonicity formula in \eqref{eq:gKdV}.
Consider the \emph{defocusing generalized Korteweg-de Vries} \emph{equation}
\eqref{eq:gKdV}
\begin{equation}
\partial_{t}u+\partial_{xxx}u=\partial_{x}(|u|^{p-1}u)\tag{gKdV}\label{eq:gKdV}
\end{equation}
where $p$ is an integer $\geq2$ and $u:\mathbb{R}\times\mathbb{R}\to\mathbb{R}$.
If a wave packet with frequency $\xi$ solves the \emph{Airy equation},
i.e. the linear part of the \eqref{eq:gKdV}, then it propagates with
group velocity $-3\xi^{2}$. This says that any solutions to the Airy
equation essentially propagates to the left and the higher frequency
piece travels faster. The defocusing \eqref{eq:gKdV} preserves this
phenomenon, as we shall see in the next paragraph.

Not only qualitatively, we can quantitatively capture the unidirectional
propagation of \eqref{eq:gKdV}. Let us define mass density $\rho_{\mathrm{gKdV}}$
and energy density $e_{\mathrm{gKdV}}$ by 
\[
\rho_{\mathrm{gKdV}}\coloneqq u^{2}\quad\text{and}\quad e_{\mathrm{gKdV}}\coloneqq\frac{1}{2}u_{x}^{2}+\frac{1}{p+1}|u|^{p+1}.
\]
Define the mass current $j_{\mathrm{gKdV}}$ and energy current $k_{\mathrm{gKdV}}$
by
\[
j_{\mathrm{gKdV}}\coloneqq2u_{x}^{2}+\frac{2p}{p+1}|u|^{p+1}\quad\text{and}\quad k_{\mathrm{gKdV}}\coloneqq\frac{3}{2}u_{xx}^{2}+2p|u|^{p-1}u_{x}^{2}+\frac{1}{2}|u|^{2p}.
\]
They satisfy the local conservation laws
\begin{align}
\partial_{t}\rho_{\mathrm{gKdV}}+\partial_{xxx}\rho_{\mathrm{gKdV}} & =\partial_{x}j_{\mathrm{gKdV}},\label{eq:loc cons law 1}\\
\partial_{t}e_{\mathrm{gKdV}}+\partial_{xxx}e_{\mathrm{gKdV}} & =\partial_{x}k_{\mathrm{gKdV}},\label{eq:loc cons law 2}
\end{align}
and we have mass and energy conservation laws,
\[
M_{\mathrm{gKdV}}(u)\coloneqq\int\rho_{\mathrm{gKdV}}\quad\text{and}\quad E_{\mathrm{gKdV}}(u)\coloneqq\int e_{\mathrm{gKdV}}.
\]
We consider the center of mass and energy $\langle x\rangle_{M}$
and $\langle x\rangle_{E}$:
\[
\langle x\rangle_{M}\coloneqq\frac{1}{M_{\mathrm{gKdV}}(u)}\int x\rho_{\mathrm{gKdV}}\quad\text{and}\quad\langle x\rangle_{E}\coloneqq\frac{1}{E_{\mathrm{gKdV}}(u)}\int xe_{\mathrm{gKdV}},
\]
whenever the integrals are defined. We then have monotonicity formulae\footnote{In fact, the lower bound of \eqref{eq:gKdV monotonicity not interaction 2}
depends on $\|u\|_{H^{2}}$, but it is not conserved under the flow.}
\begin{align}
-\partial_{t}\langle x\rangle_{M} & =\frac{1}{M_{\mathrm{gKdV}}(u)}\int j_{\mathrm{gKdV}}\gtrsim_{E_{\mathrm{gKdV}}(u),M_{\mathrm{gKdV}}(u)}1,\label{eq:gKdV monotonicity not interaction}\\
-\partial_{t}\langle x\rangle_{E} & =\frac{1}{E_{\textrm{gKdV}}(u)}\int k_{\mathrm{gKdV}}>0,\label{eq:gKdV monotonicity not interaction 2}
\end{align}
The formulae \eqref{eq:gKdV monotonicity not interaction} and \eqref{eq:gKdV monotonicity not interaction 2}
tell that wave packets move to the left. From sign definiteness of
the group velocity $-3\xi^{2}$, we expect a bit more. Indeed, $\langle x\rangle_{E}$
moves faster than $\langle x\rangle_{M}$ because $\langle x\rangle_{E}$
is more weighted on higher frequencies.

Tao \cite{Tao2007DiscCont} obtained the following refined monotonicity
(where we drop the subscript $\mathrm{gKdV}$) 
\begin{equation}
\partial_{t}(\langle x\rangle_{M}-\langle x\rangle_{E})\gtrsim\frac{M(u)}{E(u)}\Big(\int|u|^{p+1}\Big)^{2}.\label{eq:gKdV monotonicity}
\end{equation}
Equivalently, this phenomenon of separation of mass center and energy
center can be rewritten in an interaction form:
\begin{align}
 & \partial_{t}\int_{\mathbb{R}\times\mathbb{R}}(x-y)\rho(t,x)e(t,y)dxdy\nonumber \\
 & =\int_{\mathbb{R}}e(t,y)dy\Big(\partial_{t}\int_{\mathbb{R}}x\rho(t,x)dx\Big)-\int_{\mathbb{R}}\rho(t,x)dx\Big(\partial_{t}\int_{\mathbb{R}}ye(t,y)dy\Big)\nonumber \\
 & =M(u)E(u)\partial_{t}(\langle x\rangle_{M}-\langle x\rangle_{E})\nonumber \\
 & \gtrsim M(u)^{2}\Big(\int|u|^{p+1}\Big)^{2}.\label{eq:interaction}
\end{align}
Here, we assume good spatial decay of $u$ to guarantee that integrals
are finite. In \cite{Dodson2017AnnPDE}, Dodson makes use of \eqref{eq:interaction}
to show the scattering of the defocusing mass-critical \eqref{eq:gKdV}
flow.\footnote{Dodson uses it to preclude soliton-like enemies.}
In practice, since one cannot assume that $u$ has a good spatial
decay, one uses a localized version of \eqref{eq:interaction} replacing
$x-y$ by $\Phi(x-y)$ where $\Phi$ is a truncated version of $x-y$.

We remark that the monotonicity formulae \eqref{eq:gKdV monotonicity not interaction}
and \eqref{eq:gKdV monotonicity not interaction 2} are not sufficient
to prove scattering. In order to show scattering, we should somehow
preclude soliton-like solutions, which preserve their profile but
move sufficiently fast to the left.

\ 

In the study of \eqref{eq:gBO}, we expect the similar unidirectional
propagation as like \eqref{eq:gKdV}. In fact, a wave packet with
frequency $\xi$ propagates with group velocity $2|\xi|$ under the
linear Benjamin-Ono flow. It is natural to expect that the center
of energy moves to the right faster than the center of mass. It turns
out that Tao's monotonicity holds in \eqref{eq:gBO}. Henceforth,
we focus on obtaining analogous monotonicity formula of interaction
form.

We define \emph{mass density} $\rho$ and \emph{energy density} $e$
of $u$ by 
\begin{align*}
\rho[u] & \coloneqq u^{2},\\
e[u] & \coloneqq\frac{1}{2}u\mathcal{H}u_{x}+\frac{1}{k+2}u^{k+2}.
\end{align*}
We have mass and energy conservation laws
\[
M(u)\coloneqq\int\rho\quad\text{and}\quad E(u)\coloneqq\int e.
\]
We define \emph{mass current} $j$ and \emph{energy current} $k$
of $u$ by 
\begin{align*}
j[u] & \coloneqq2u\mathcal{H}u_{x}+\frac{2(k+1)}{k+2}u^{k+2},\\
k[u] & \coloneqq u_{x}^{2}+\frac{3}{2}u^{k+1}\mathcal{H}u_{x}+\frac{1}{2}u^{2k+2}.
\end{align*}
They satisfy 
\begin{equation}
\partial_{t}\int x\rho=\int j\quad\text{and}\quad\partial_{t}\int xe=\int k.\label{eq:mass energy center}
\end{equation}
These namings come from analogy with the ones of \eqref{eq:gKdV}.
\begin{rem}
It is worth noticing that we do not have pointwise nonnegativity of
$e$, $j$, and $k$ now. This fact will cause problems when we estimate
errors arising from localizing monotonicity formula. For instance,
we \emph{cannot} say that $\|e[u]\|_{L^{1}}$ is equal to $E(u)$,
or even estimated in terms of $\|u\|_{H^{\frac{1}{2}}}$. This is
in contrast to the case of \eqref{eq:gKdV}, where we have $\|e_{\mathrm{gKdV}}[u]\|_{L^{1}}=E_{\mathrm{gKdV}}(u)$.
We will come back to this issue in Section \ref{sec:precllusion of a.p. sol}.
\end{rem}
\begin{rem}
We are not sure whether local conservations laws such as \eqref{eq:loc cons law 1}
and \eqref{eq:loc cons law 2} hold for \eqref{eq:gBO}. In the derivation
what follows, we only use \eqref{eq:mass energy center}. Note that
when we prove mass/energy conservation, it suffices to use integration
by parts and properties of the Hilbert transform.
\end{rem}
Let us now state and prove the monotonicity formula for \eqref{eq:gBO}.
\begin{prop}[Monotonicity formula]
\label{prop:monotonicity formula}Let $p\geq\sqrt{2}$ and $u\in H^{1}$.
Then, we have
\[
\int\rho\int k-\int j\int e\geq\frac{k^{2}}{2(k+2)^{2}}M(u)^{2}\Big(\int u^{k+2}\Big)^{2}.
\]
If $u(t,x)$ is a classical solution to \eqref{eq:gBO} satisfying,
for example, $\langle x\rangle u\in C_{t,loc}H^{1}$,\footnote{In Section \ref{sec:precllusion of a.p. sol}, we use a truncated
version of the monotonicity formula instead of \eqref{eq:interaction gBO}.
Thus we do not need additional assumption on spatial decay of $u$
for our later analysis.}then we have
\begin{equation}
\partial_{t}\int_{\mathbb{R}\times\mathbb{R}}(y-x)\rho(x)e(y)dxdy\geq\frac{k^{2}}{2(k+2)^{2}}M(u)^{2}\Big(\int u^{k+2}\Big)^{2}.\label{eq:interaction gBO}
\end{equation}
In particular,
\[
\partial_{t}(\langle x\rangle_{E}-\langle x\rangle_{M})>0.
\]
\end{prop}
\begin{proof}
We closely follow Tao's argument, but we encounter a technical difficulty.
We will explain the difference on the way.

Note that the second and third assertion follow from the first assertion
using mass and energy conservation and \eqref{eq:mass energy center}.
The first assertion is elaborated as
\begin{multline*}
\Big(\int u^{2}\Big)\Big(\int u_{x}^{2}+\frac{3}{2}\int u^{k+1}\mathcal{H}u_{x}+\frac{1}{2}\int u^{2k+2}\Big)\\
-\Big(2\int u\mathcal{H}u_{x}+\frac{2(k+1)}{k+2}\int u^{k+2}\Big)\Big(\frac{1}{2}\int u\mathcal{H}u_{x}+\frac{1}{k+2}\int u^{k+2}\Big)\\
\geq\frac{k^{2}}{2(k+2)^{2}}\Big(\int u^{k+2}\int u^{2}\Big)^{2}.
\end{multline*}
Set real numbers $a,b,q,r,s$ such that 
\begin{align*}
a^{2}M(u) & =\int u_{x}^{2},\qquad aqM(u)=\int u\mathcal{H}u_{x},\\
b^{2}M(u) & =\int u^{2p},\qquad brM(u)=\int u^{k+2},\qquad absM(u)=\int u^{k+1}\mathcal{H}u_{x}.
\end{align*}
It then suffices to show that 
\[
a^{2}(1-q^{2})+ab\Big(\frac{3}{2}s-\frac{k+3}{k+2}qr\Big)+\frac{b^{2}}{2}\Big(1-\frac{4(k+1)}{(k+2)^{2}}r^{2}\Big)\geq\frac{k^{2}}{2(k+2)^{2}}b^{2}r^{2},
\]
or equivalently, 
\[
a^{2}(1-q^{2})+ab\Big(\frac{3}{2}s-\frac{k+3}{k+2}qr\Big)+\frac{b^{2}}{2}(1-r^{2})\geq0.
\]
It is obvious that $q$ and $r$ are positive. However, it is not
trivial whether $s$ is positive or not. For a moment, we assume $s>0$
and proceed to complete the proof. Then, we will provide a proof of
$s>0$ in Lemma \ref{lem:positivity of s}.
\begin{rem}
In case of \eqref{eq:gKdV}, $s$ corresponds to $\int p|u|^{p-1}u_{x}^{2}$.
So the positivity of $s$ is obvious.
\end{rem}
\begin{lem}
The real symmetric matrix
\[
\begin{pmatrix}1 & q & r\\
q & 1 & s\\
r & s & 1
\end{pmatrix}
\]
is positive semi-definite.
\end{lem}
\begin{proof}
For any real numbers $\alpha,\beta,\gamma$, a computation shows that
\begin{align*}
 & \begin{pmatrix}\gamma & \alpha & \beta\end{pmatrix}\begin{pmatrix}1 & q & r\\
q & 1 & s\\
r & s & 1
\end{pmatrix}\begin{pmatrix}\gamma\\
\alpha\\
\beta
\end{pmatrix}\\
 & =\gamma^{2}+\alpha^{2}+\beta^{2}+2\gamma\alpha q+2\gamma\beta r+2\alpha\beta s\\
 & =\frac{1}{M(u)}\int(\gamma u+\frac{\alpha}{a}\mathcal{H}u_{x}+\frac{\beta}{b}u^{k+1})^{2}
\end{align*}
is always nonnegative.
\end{proof}
Taking determinants and minors, we have
\[
0<q,r,s\leq1\qquad\text{and}\qquad1-q^{2}-r^{2}-s^{2}+2qrs\geq0.
\]
Using discriminants, it suffices to show that 
\[
\frac{k+3}{k+2}qr-\frac{3}{2}s\leq\sqrt{2(1-q^{2})(1-r^{2})}.
\]
Because $k\geq4$, we have $\frac{k+3}{k+2}\leq\sqrt{2}$. As $qr$
is positive, it reduces to 
\[
s\geq\frac{2\sqrt{2}}{3}[qr-\sqrt{(1-q^{2})(1-r^{2})}].
\]
On the other hand, we know $(s-qr)^{2}\leq(1-q^{2})(1-r^{2})$ by
positive-definiteness of the matrix. This yields 
\[
s\geq qr-\sqrt{(1-q^{2})(1-r^{2})}.
\]
Hence, assuming that $s$ is positive, this completes the proof.
\end{proof}
The rest of this section is to show that $s$ is positive. In other
words, it suffices to show that 
\begin{equation}
\int_{\mathbb{R}}u^{k+1}\mathcal{H}u_{x}>0.\label{eq:positivity of s}
\end{equation}
It seems not easy to prove \eqref{eq:positivity of s} directly. Indeed,
we were not able to prove \eqref{eq:positivity of s} using Fourier
expressions or integration by parts.

The key observation is as follows. Using the transference principle
between the Fourier series and transform, we change the problem defined
on the torus $\mathbb{T}$. We then use density argument to further
reduce to the finite-dimensional setting. More precisely, we can treat
\eqref{eq:positivity of s} as a functional defined on some finite-dimensional
Hilbert space. Moreover, as \eqref{eq:positivity of s} is homogeneous
in $u$, it suffices to restrict ourselves on the finite-dimensional
sphere, which is compact. We then use the Lagrange multiplier method
and the relations satisfied by a minimizer. It turns out that \eqref{eq:positivity of s}
for a minimizer is positive.

For the later use in the truncated monotonicity formula, we need more
general positivity lemma. \eqref{eq:positivity of s} a direct consequence
of substituting $\chi=1$ into the following lemma (technically, one
should mimic the proof).
\begin{lem}[Positivity of $s$]
\label{lem:positivity of s}For $\chi\in C_{c}^{\infty}$ and $u\in H^{1}$,
we have 
\[
\int_{\mathbb{R}}\chi^{2}u^{k+1}\mathcal{H}u_{x}>-c\|\partial_{xx}(\chi^{2})\|_{L^{\infty}}\|u\|_{\dot{H}^{\alpha}}^{k+2}
\]
where $\alpha=\frac{1}{2}-\frac{2}{k+2}$ and $c$ is some implicit
constant.
\end{lem}
\begin{proof}
We observe that both $u$ and $\mathcal{H}u$ cannot be localized
simultaneously in the physical space. So we are not able to directly
use the density argument to reduce for $u$ and $\mathcal{H}u\in C_{c}^{\infty}$.
Thus we localize them in the Fourier space instead. Lastly, we use
the transference principle to finish the proof on the real line.

We first show our assertion on the torus $\mathbb{T}$. More precisely,
we show that
\[
\int_{\mathbb{T}}\chi^{2}u^{k+1}D_{x}u>-c\|\partial_{xx}(\chi^{2})\|_{L^{\infty}(\mathbb{T})}\|u\|_{\dot{H}^{\alpha}(\mathbb{T})}^{k+2}
\]
for all $u\in C^{\infty}(\mathbb{T})$ with $\hat{u}(0)=0$. By density,
we may assume that $\hat{u}$ is compactly supported. We consider
a finite-dimensional Hilbert space $\mathcal{H}_{N}$ and a sphere
$S$ in $\mathcal{H}_{N}$ as follows. 
\begin{align*}
\mathcal{H}_{N} & \coloneqq\{u\in L^{2}(\mathbb{T}):|\hat{u}(\xi)|=0\text{ if }\xi=0\text{ or }|\xi|>N\},\\
S & \coloneqq\{u\in\mathcal{H}_{N}:\|u\|_{\dot{H}^{\alpha}(\mathbb{T})}=1\}.
\end{align*}
Compactness of $S$ in $\mathcal{H}_{N}$ will play a crucial role
in what follows. Let us define a function $f:\mathcal{H}_{N}\to\mathbb{R}$
by 
\[
f(u)\coloneqq\int_{\mathbb{T}}\chi^{2}u^{k+1}D_{x}u.
\]
Then $f:\mathcal{H}_{N}\to\mathbb{R}$ is smooth and 
\[
\nabla f(u)=Q_{\leq N}\big[(k+1)\chi^{2}u^{k}D_{x}u+D_{x}(\chi^{2}u^{k+1})\big]
\]
with respect to the usual $L^{2}(\mathbb{T})$ inner product. On the
other hand, the constraint function $g(u)\coloneqq\|u\|_{\dot{H}^{\alpha}(\mathbb{T})}^{2}$
satisfies 
\[
\nabla g(u)=2D_{x}^{2\alpha}u.
\]
Let $u_{0}$ be a point in $S$ that attains the minimum of $f$.
By the Lagrange multiplier theorem, we have
\[
\nabla f(u_{0})=\lambda D_{x}^{2\alpha}u_{0}
\]
for some $\lambda\in\mathbb{R}$. In a spirit of Pohozaev identities,
we compute
\[
\lambda=\langle\lambda D_{x}^{2\alpha}u_{0},u_{0}\rangle=\langle\nabla f(u),u_{0}\rangle=(k+2)f(u_{0})
\]
and 
\begin{align*}
\lambda\langle D_{x}^{2\alpha}u_{0},D_{x}u_{0}\rangle & =\langle\nabla f(u_{0}),D_{x}u_{0}\rangle\\
 & =(k+1)\int\chi^{2}u_{0}^{k}[(D_{x}u_{0})^{2}+(\partial_{x}u_{0})^{2}]-\frac{1}{k+2}\int\partial_{xx}(\chi^{2})u_{0}^{k+2}.
\end{align*}
Therefore, 
\[
\lambda>-\frac{1}{(k+2)^{2}}\|\partial_{xx}(\chi^{2})\|_{L^{\infty}(\mathbb{T})}\frac{\|u_{0}\|_{L^{k+2}(\mathbb{T})}^{k+2}}{\|u_{0}\|_{\dot{H}^{\alpha+\frac{1}{2}}(\mathbb{T})}^{2}}.
\]
Because $\|u_{0}\|_{\dot{H}^{\alpha}}=1$ and $\|u\|_{L^{k+2}}^{k+2}\lesssim\|u\|_{\dot{H}^{\alpha+\frac{1}{2}}}^{2}\|u\|_{\dot{H}^{\alpha}}^{k}$
for any $u\in C^{\infty}(\mathbb{T})$ having mean zero (see \cite{BenyiOh2013}),
we have
\[
f(u_{0})>-c\|\partial_{xx}(\chi^{2})\|_{L^{\infty}(\mathbb{T})}
\]
for some $c>0$. Therefore, we have
\[
\int_{\mathbb{T}}\chi^{2}u^{k+1}D_{x}u>-c\|\partial_{xx}(\chi^{2})\|_{L^{\infty}(\mathbb{T})}\|u\|_{\dot{H}^{\alpha}(\mathbb{T})}^{k+2}
\]
proving the claim on the torus $\mathbb{T}$.

We now show how to transfer the result on the torus $\mathbb{T}$
to that on the real line $\mathbb{R}$. Let $u$ be a function in
$H^{1}(\mathbb{R})$. By density and scaling, we may assume that $\hat{\chi}$
and $\hat{u}$ are compactly supported on $[0,1]$ and continuous.
Observe by Parseval's identity that 
\[
\int_{\mathbb{R}}\chi^{2}u^{k+1}D_{x}u=\int_{\eta_{1},\eta_{2},\xi_{1},\dots,\xi_{k+1}\in[0,1]}\hat{\chi}(\eta_{1})\hat{\chi}(\eta_{2})\hat{u}(\xi_{1})\cdots\hat{u}(\xi_{k+1})|\xi|\hat{u}(\xi)
\]
where $\xi\coloneqq\eta_{1}+\eta_{2}+\xi_{1}+\cdots+\xi_{k+1}$. In
light of Riemann sum, the right hand side in the above is expressed
as
\[
\lim_{N\to\infty}\frac{1}{N^{k+4}}\sum_{\eta_{1},\eta_{2},\xi_{1},\dots,\xi_{k+1}\in A_{N}}\hat{\chi}(\frac{\eta_{1}}{N})\hat{\chi}(\frac{\eta_{2}}{N})\hat{u}(\frac{\xi_{1}}{N})\cdots\hat{u}(\frac{\xi_{k+1}}{N})|\xi|\hat{u}(\frac{\xi}{N})
\]
where $\xi=\eta_{1}+\eta_{2}+\xi_{1}+\cdots+\xi_{k+1}$ and $A_{N}=\{-N,-N+1,\dots,-1,1,\dots,N-1,N\}$.
Notice that $A_{N}$ does not contain $0$. The above display now
equals
\[
\lim_{N\to\infty}\frac{1}{N^{k+4}}\int_{\mathbb{T}}\chi_{N}^{2}u_{N}^{k+1}D_{x}u_{N}
\]
where $\chi_{N}$ and $u_{N}$ are functions defined on the torus
$\mathbb{T}$ by
\[
\chi_{N}(x)\coloneqq\sum_{\eta\in A_{N}}\hat{\chi}(\frac{\eta}{N})e^{2\pi i\eta x}\quad\text{and}\quad u_{N}(x)\coloneqq\sum_{\xi\in A_{N}}\hat{u}(\frac{\xi}{N})e^{2\pi i\xi x}.
\]
Applying the result on the torus case, we have 
\[
\int_{\mathbb{R}}\chi^{2}u^{k+1}\mathcal{H}u_{x}>-c\limsup_{N\to\infty}\frac{1}{N^{k+4}}\|\partial_{xx}(\chi_{N}^{2})\|_{L^{\infty}(\mathbb{T})}\|u_{N}\|_{\dot{H}^{\alpha}(\mathbb{T})}^{k+2}.
\]
In view of 
\begin{align*}
\lim_{N\to\infty}\frac{1}{N^{4}}\|\partial_{xx}(\chi_{N}^{2})\|_{L^{\infty}(\mathbb{T})} & =\|\partial_{xx}(\chi^{2})\|_{L^{\infty}(\mathbb{R})},\\
\lim_{N\to\infty}\frac{1}{N^{k}}\|u_{N}\|_{\dot{H}^{\alpha}(\mathbb{T})}^{k+2} & =\|u\|_{\dot{H}^{\alpha}(\mathbb{R})}^{k+2},
\end{align*}
we conclude 
\[
\int_{\mathbb{R}}\chi^{2}u^{k+1}\mathcal{H}u_{x}>-c\|\partial_{xx}(\chi^{2})\|_{L^{\infty}(\mathbb{R})}\|u\|_{\dot{H}^{\alpha}(\mathbb{R})}^{k+2}.
\]
\end{proof}

\section{\label{sec:WP theory}Well-posedness Theory}

One of the main ingredient of compactness-contradiction argument is
the perturbation theory. This is basically inherited from local well-posedness
of the Cauchy problem. As like other canonical nonlinear dispersive
equations, \eqref{eq:gBO}'s local theory is based on local smoothing
estimates. A satisfactory critical well-posedness was obtained by
Vento \cite{Vento2010}. Due to a full derivative in the nonlinearity,
the local theory is more delicate than other equations such as nonlinear
Schr\"odinger equations and generalized Korteweg-de Vries equation.
In this section, we begin reviewing Vento's proof of the critical
local well-posedness. We then prove subcritical local theory and global
well-posedness.

Denote by $V(t)$ the linear propagator associated to the linear Benjamin-Ono
flow. By Duhamel's formula, a nonlinear solution $u$ to initial data
$u_{0}$ satisfies 
\[
u(t)=V(t)u_{0}+\int_{0}^{t}V(t-s)F(u(s))ds
\]
where $F(u)=-\partial_{x}(u^{k+1})$ is the nonlinearity of \eqref{eq:gBO}.
Using the local smoothing estimates, one has 
\[
\|V(t)u_{0}\|_{\mathcal{S}}\lesssim\|u_{0}\|_{\dot{H}^{s}}
\]
where $\mathcal{S}$ is some space where the linear evolution lies
and $\dot{H}^{s}$ is a Sobolev space where initial data lies. Usually,
combined with Christ-Kiselev lemma, the nonlinear evolution part can
be estimated by
\[
\Big\|\int_{0}^{t}V(t-s)F(s)ds\Big\|_{\mathcal{S}}\lesssim\|F\|_{\mathcal{N}}
\]
where $\mathcal{N}$ is the dual space of $\mathcal{S}$. The main
goal of the local theory is to find suitable function spaces $\mathcal{S}$
and $\mathcal{N}$ such that a nonlinear estimate 
\[
\|F(u)\|_{\mathcal{N}}\lesssim\|u\|_{\mathcal{S}}^{k+1}
\]
holds and $\mathcal{S}$ is embedded into $C_{T}\dot{H}^{s}$.

In Section \ref{subsec:review of Vento's work}, we review the critical
local theory of \eqref{eq:gBO} by Vento \cite{Vento2010}. In Section
\ref{subsec:GWP}, we prove subcritical local well-posedness in $H^{s}$
($s>s_{k}$), and deduce global well-posedness in $H^{s}$ ($s\geq\frac{1}{2}$)
by using  conservation laws.

\subsection{\label{subsec:review of Vento's work}Review of Vento's argument}

In this subsection, we review Vento's approach \cite{Vento2010} in
proving critical local well-posedness of \eqref{eq:gBO}. Because
our nonlinearity $F(u)=-\partial_{x}(u^{k+1})$ has one derivative,
we should somehow recover one derivative. More precisely, in view
of Duhamel's formula, we want to estimate the integral part as
\[
\Big\|\int_{0}^{t}V(t-s)F(u(s))ds\Big\|_{\mathcal{S}}\lesssim\|\partial_{x}(u^{k+1})\|_{\mathcal{N}}\lesssim\|u\|_{\mathcal{S}}^{k+1}.
\]
In the linear Benjamin-Ono equation, it is well-known that local smoothing
estimates can recover at most half derivative:
\[
\|D_{x}^{s_{k}+\frac{1}{2}}V(t)u_{0}\|_{L_{x}^{\infty}L_{t}^{2}}\lesssim\|u_{0}\|_{\dot{H}^{s_{k}}}.
\]
In a spirit of $TT^{\ast}$ formulation, one can recover the full
derivative only when we use $L_{x}^{\infty}L_{t}^{2}$ norm for $\mathcal{S}$
and $L_{x}^{1}L_{t}^{2}$ norm for $\mathcal{N}$. However, in order
to bound $L_{x}^{1}L_{t}^{2}$ norm by $L_{x}^{\infty}L_{t}^{2}$
(and other norms), we are forced to use maximal-type norms, $L_{x}^{p}L_{t}^{\infty}$.
A naive choice would be as follows
\begin{equation}
\|D_{x}^{s_{k}-\frac{1}{2}}\partial_{x}(u^{k+1})\|_{L_{x}^{1}L_{T}^{2}}\lesssim\|D_{x}^{s_{k}+\frac{1}{2}}u\|_{L_{x}^{\infty}L_{T}^{2}}\|u\|_{L_{x}^{k}L_{T}^{\infty}}^{k}.\label{eq:double local smoothing}
\end{equation}
However, for large data, since one takes $L_{T}^{\infty}$, it is
not possible to make small $L_{x}^{k}L_{T}^{\infty}$ by shrinking
$T$ small. This prohibits us from running contraction argument.

Vento \cite{Vento2010} overcomes this difficulty by distorting the
linear propagator. Consider a general form of a nonlinear evolution
equation
\[
\partial_{t}u+\mathcal{L}u=F(u)
\]
for some linear operator $\mathcal{L}$. In a spirit of paraproduct
decomposition, we extract the strong interaction term $\mathcal{N}_{1}(u)$
of $F(u)$, for which one cannot establish a desired nonlinear estimate.
and rewrite the equation by
\[
\partial_{t}u+\mathcal{L}u-\mathcal{N}_{1}(u)=\mathcal{N}_{2}(u),
\]
where $\mathcal{N}_{2}(u)\coloneqq F(u)-\mathcal{N}_{1}(u)$. The
strong interaction term $\mathcal{N}_{1}(u)$ is nonlinear in $u$,
we approximate it by some linear operator $u\mapsto\mathcal{N}_{1}(u_{0};u)$
using the initial data $u_{0}$. We then have the distorted equation
\[
\partial_{t}u+\tilde{\mathcal{L}}u=[\mathcal{N}_{1}(u)-\mathcal{N}_{1}(u_{0};u)]+\mathcal{N}_{2}(u),
\]
where $\tilde{\mathcal{L}}\coloneqq\mathcal{L}-\mathcal{N}_{1}(u_{0};\cdot)$.
It turns out that $\mathcal{N}_{1}(u_{0};u)$ contains a strong low-high
interaction, $u_{\mathrm{low}}^{k}\partial_{x}u_{\mathrm{high}}$.
This phenomenon is universal in other nonlinear dispersive equations
containing derivative nonlinearities such as KP-I, Benjamin-Ono, higher
order KdV, and so on.

Vento's observations are as follows. At first, the linear part to
the above distorted equation still admits analogous linear estimates
for the linear Benjamin-Ono equation, at least for short times. Secondly,
even though we cannot avoid usage of $L_{x}^{k}L_{T}^{\infty}$, we
have small $\|u-u_{0}\|_{L_{x}^{k}L_{T}^{\infty}}$ at least for short
times. Indeed, this holds when $u$ is a linear solution to either
Benjamin-Ono equation or distorted linear solution. This fact is well-exploited
in $\mathcal{N}_{1}(u)-\mathcal{N}_{1}(u_{0};u)$ estimate. Finally
the good term $\mathcal{N}_{2}(u)$, as its naming suggests, can be
estimated small by shrinking $T$.

\ 

We first start with linear estimates for the linear Benjamin-Ono flow.
A triplet $(\alpha,p,q)\in\mathbb{R}^{3}$ is said to be \emph{admissible}
if $(\alpha,p,q)=(\frac{1}{2},\infty,2)$ or 
\[
4\leq p<\infty,\quad2<q\leq\infty,\quad\frac{2}{p}+\frac{1}{q}\leq\frac{1}{2},\quad\alpha=\frac{1}{p}+\frac{2}{q}-\frac{1}{2}.
\]
We then have linear estimates for these admissible triples as follows.
\begin{lem}
\label{lem:local smoothing estimates}Let $(\alpha,p,q)$ and $(\tilde{\alpha},\tilde{p},\tilde{q})$
be admissible triplets. We then have
\begin{align*}
\|D_{x}^{\alpha}V(t)\varphi\|_{L_{x}^{p}L_{t}^{q}} & \lesssim\|\varphi\|_{L^{2}},\\
\Big\| D_{x}^{\alpha}\int V(-s)G\Big\|_{L^{2}} & \lesssim\|G\|_{L_{x}^{\tilde{p}'}L_{t}^{\tilde{q}'}},\\
\Big\| D_{x}^{\alpha+\tilde{\alpha}}\int_{0}^{t}V(t-s)G(s)ds\Big\|_{L_{x}^{p}L_{t}^{q}} & \lesssim\|G\|_{L_{x}^{\tilde{p}'}L_{t}^{\tilde{q}'}},
\end{align*}
where $\tilde{p}'$ and $\tilde{q}'$ are conjugate Lebesgue exponents
of $\tilde{p}$ and $\tilde{q}$, respectively.
\end{lem}
\begin{proof}
From \cite[Theorem 2.5 and 4.1]{KPV1991Indiana}, one has
\[
\|D_{x}^{\frac{1}{2}}V(t)\varphi\|_{L_{x}^{\infty}L_{t}^{2}}\sim\|\varphi\|_{L^{2}}\quad\text{and}\quad\|V(t)\varphi\|_{L_{x}^{4}L_{t}^{\infty}}\lesssim\|D_{x}^{\frac{1}{4}}\varphi\|_{L^{2}}.
\]
We then apply Stein-Weiss interpolation to obtain the first estimate.
The second estimate easily follows by duality. For the last one, we
have by first two estimates that 
\[
\Big\| D_{x}^{\alpha+\tilde{\alpha}}\int_{\mathbb{R}}V(t-s)G(s)ds\Big\|_{L_{x}^{p}L_{t}^{q}}\lesssim\|G\|_{L_{x}^{\tilde{p}'}L_{t}^{\tilde{q}'}}.
\]
By the Christ-Kiselev lemma for reversed norm \cite[Theorem B]{BurqPlanchon2006},
we have 
\[
\Big\| D_{x}^{\alpha+\tilde{\alpha}}\int_{0}^{t}V(t-s)G(s)ds\Big\|_{L_{x}^{p}L_{t}^{q}}\lesssim\|G\|_{L_{x}^{\tilde{p}'}L_{t}^{\tilde{q}'}}
\]
except the case where $(\alpha,p,q)=(\tilde{\alpha},\tilde{p},\tilde{q})=(\frac{1}{2},\infty,2)$.
In that case, we refer to \cite[Theorem 2.1]{KPV1994TransAMS}.
\end{proof}
We use Besov-type function spaces for $s\in\mathbb{R}$
\[
\dot{\mathcal{S}}_{I}^{s,\theta}=\dot{\mathcal{B}}_{\frac{4}{1-\theta}}^{s+\frac{3\theta-1}{4},2}\big(L_{I}^{\frac{2}{\theta}}\big),\qquad\dot{\mathcal{N}}_{I}^{s}=\dot{\mathcal{B}}_{1}^{s-\frac{1}{2},2}(L_{I}^{2}),
\]
where $\theta\in[0,1]$. One can think that the space $\dot{\mathcal{S}}_{I}^{s,\theta}$
contains solutions to \eqref{eq:gBO} and $\dot{\mathcal{N}}_{I}^{s}$
contains the nonlinearity of \eqref{eq:gBO}. We do not parametrize
the space for nonlinearity because we will only consider $\dot{\mathcal{N}}_{I}^{s}$.

We define the space $\dot{X}_{I}^{s}$ by
\[
\dot{X}_{I}^{s}\coloneqq\dot{\mathcal{S}}_{I}^{s,\epsilon}\cap\dot{\mathcal{S}}_{I}^{s,1}.
\]
The parameter $0<\epsilon\ll1$ only depends on $k$ and is chosen
small. For the choice of $\epsilon$, see \cite{Vento2010}. As $\dot{X}_{T}^{s}$-norm
does not contain $L_{T}^{\infty}$-norm, it becomes smaller as $T$
becomes smaller. This property is crucial in our argument. Of course,
solutions to \eqref{eq:gBO} should lie in $C_{T}\dot{H}_{x}^{s}$.
For this purpose, we finally work on the Banach space 
\[
\dot{Z}_{I}^{s}\coloneqq\{u\in C_{I}\dot{H}_{x}^{s}\cap L_{x}^{k}L_{I}^{\infty}\cap\dot{\mathcal{S}}_{I}^{s,0}\cap\dot{\mathcal{S}}_{I}^{s,1}:u\text{ satisfies }\eqref{eq:continuity assumption 1 for (k,infty)}\text{ and }\eqref{eq:continuity assumption 2 for (k,infty)}\},
\]
where ($J\subseteq I$ is a subinterval of $I$)
\begin{gather}
\text{for any }t_{0}\in I,\ \|u-u(t_{0})\|_{L_{x}^{k}L_{J}^{\infty}}\to0\text{ whenever }t_{0}\in J\text{ and }|J|\to0,\label{eq:continuity assumption 1 for (k,infty)}\\
\lim_{T\to+\infty}\|u\|_{L_{x}^{k}L_{[T,+\infty)}^{\infty}}=0=\lim_{T\to-\infty}\|u\|_{L_{x}^{k}L_{(-\infty,T]}^{\infty}}.\label{eq:continuity assumption 2 for (k,infty)}
\end{gather}
The inhomogeneous form of the spaces $X$ and $Z$ for subcritical
well-posedness theory are defined as follows. For $s>0$, we define
\[
X_{I}^{s}\coloneqq\dot{X}_{I}^{0}\cap\dot{X}_{I}^{s}\quad\text{and}\quad Z_{I}^{s}\coloneqq\dot{Z}_{I}^{0}\cap\dot{Z}_{I}^{s}.
\]

\begin{prop}[Linear estimates]
Let $k\geq4$ and $s\in\mathbb{R}$. We then have
\begin{align*}
\|V(t)\varphi\|_{\dot{\mathcal{S}}_{I}^{s,0}\cap\dot{\mathcal{S}}_{I}^{s,1}\cap L_{t}^{\infty}\dot{H}_{x}^{s}} & \lesssim\|\varphi\|_{\dot{H}^{s}},\\
\|V(t)\varphi\|_{L_{x}^{k}L_{t}^{\infty}} & \lesssim\|\varphi\|_{\dot{H}^{s_{k}}}.
\end{align*}
For retarded estimates, we have
\begin{align*}
\Big\|\int_{0}^{t}V(t-s)G(s)ds\Big\|_{\dot{X}_{t}^{s}\cap L_{t}^{\infty}\dot{H}_{x}^{s}} & \lesssim\|G\|_{\dot{\mathcal{N}}_{t}^{s}},\\
\Big\|\int_{0}^{t}V(t-s)G(s)ds\Big\|_{L_{x}^{k}L_{t}^{\infty}} & \lesssim\|G\|_{\dot{\mathcal{N}}_{t}^{s_{k}}}.
\end{align*}
\end{prop}
\begin{proof}
Note that $\dot{\mathcal{B}}_{2}^{s,2}$ is equal to $\dot{H}^{s}$
by definition. We use the linear estimates (Lemma \ref{lem:local smoothing estimates})
for each frequency piece $Q_{j}\varphi$ or $Q_{j}G$ and take $\ell_{j}^{2}$
summation. For instance, the last estimate follows from
\begin{align*}
\Big\| V(t)\int_{\mathbb{R}}V(-s)G(s)ds\Big\|_{L_{x}^{k}L_{t}^{\infty}} & \lesssim\Big\|\int_{\mathbb{R}}V(-s)G(s)ds\Big\|_{\dot{\mathcal{B}}_{2}^{s_{k},2}}\\
 & \lesssim\|G\|_{\dot{\mathcal{N}}_{t}^{s_{k}}}.
\end{align*}
We then use Christ-Kiselev lemma \cite[Theorem B]{BurqPlanchon2006}
to conclude.
\end{proof}
We now review the paraproduct decomposition of the nonlinearity $F(u)=-\partial_{x}(u^{k+1})$
following \cite{Vento2010}. For simplicity of notations, we shall
group similar frequency pieces when they satisfy the same estimates
and number of them is bounded by a universal constant, not depending
on $J$. For instance, we group $\sum_{\ell=0}^{k}u_{<r+1}^{k-\ell}u_{<r}^{\ell}$
by $u_{<r+1}^{k}$. We observe that
\begin{align*}
\partial_{x}Q_{j}(u^{k+1}) & =\partial_{x}Q_{j}\Big(\lim_{r\to\infty}u_{<r}^{k+1}\Big)\\
 & =\partial_{x}Q_{j}\Big(\sum_{r\in\mathbb{Z}}u_{<r+1}^{k+1}-u_{<r}^{k+1}\Big)\\
 & =\partial_{x}Q_{j}\Big(\sum_{r\in\mathbb{Z}}u_{r}u_{<r+1}^{k}\Big)\\
 & =\partial_{x}Q_{j}\Big(\sum_{|r-j|\leq C_{k}}u_{r}u_{\leq r-J}^{k}\Big)+\partial_{x}Q_{j}\Big(\sum_{|r-j|\leq C_{k}}u_{r}u_{\sim r}u_{<r+1}^{k-1}\Big)\\
 & \quad+\partial_{x}Q_{j}\Big(\sum_{r\geq j+C_{k}}u_{r}u_{\sim r}u_{<r+1}^{k-1}\Big)\\
 & =\pi_{j}(u,u)-g_{j}(u)
\end{align*}
where $C_{k}$ is some natural number only depending on $k$ and 
\begin{align*}
\pi_{j}(\psi,\phi) & \coloneqq\partial_{x}Q_{j}(\psi_{\ll j}^{k}\phi_{\sim j})\\
g_{j}(u) & \coloneqq-\partial_{x}Q_{j}\Big(\sum_{r-j\geq-C_{k}}u_{\sim r}^{2}u_{\lesssim r}^{k-1}\Big).
\end{align*}
Summing up in $j$, we have 
\[
F(u)=-\pi(u,u)+g(u),
\]
where 
\begin{align*}
\pi(\psi,\phi) & \coloneqq\sum_{j}\pi_{j}(\psi,\phi),\\
g(\varphi) & \coloneqq\sum_{j}g_{j}(\varphi).
\end{align*}

Using the linear estimates, we can estimate each components of the
nonlinearity.
\begin{lem}[{Nonlinear estimates, \cite[Proposition 4.1]{Vento2010}}]
\label{lem:nonlinearity estimate}We have 
\begin{align*}
\|\pi(\varphi,\psi)\|_{\dot{\mathcal{N}}_{I}^{s_{k}}} & \lesssim\|\varphi\|_{L_{x}^{k}L_{T}^{\infty}}^{k}\|\psi\|_{\dot{X}_{I}^{s_{k}}},\\
\|g(u)\|_{\dot{\mathcal{N}}_{I}^{s_{k}}} & \lesssim\|u\|_{L_{x}^{k}L_{I}^{\infty}}^{k-1}\|u\|_{\dot{X}_{I}^{s_{k}}}^{2},\\
\|\pi(\varphi_{1},\psi)-\pi(\varphi_{2},\psi)\|_{\dot{\mathcal{N}}_{I}^{s_{k}}} & \lesssim\|\varphi_{1}-\varphi_{2}\|_{L_{x}^{k}L_{I}^{\infty}}(\|\varphi_{1}\|_{L_{x}^{k}L_{I}^{\infty}}^{k-1}+\|\varphi_{2}\|_{L_{x}^{k}L_{I}^{\infty}}^{k-1})\|\psi\|_{\dot{X}_{I}^{s_{k}}},\\
\|g(u_{1})-g(u_{2})\|_{\dot{\mathcal{N}}_{I}^{s_{k}}} & \lesssim\|u_{1}-u_{2}\|_{L_{x}^{k}L_{I}^{\infty}}(\|u_{1}\|_{L_{x}^{k}L_{I}^{\infty}}^{k-2}+\|u_{2}\|_{L_{x}^{k}L_{I}^{\infty}}^{k-2})(\|u_{1}\|_{\dot{X}_{I}^{s_{k}}}+\|u_{2}\|_{\dot{X}_{I}^{s_{k}}})^{2}\\
 & \quad+\|u_{1}-u_{2}\|_{\dot{X}_{I}^{s_{k}}}(\|u_{1}\|_{L_{x}^{k}L_{I}^{\infty}}^{k-1}+\|u_{2}\|_{L_{x}^{k}L_{I}^{\infty}}^{k-1})(\|u_{1}\|_{\dot{X}_{I}^{s_{k}}}+\|u_{2}\|_{\dot{X}_{I}^{s_{k}}})
\end{align*}
\end{lem}
\begin{proof}
We only prove first two estimates following \cite[Proposition 4.1]{Vento2010}.
The remaining estimates can be shown in a similar manner. For the
first estimate, observe that 
\begin{align*}
\|\pi(\varphi,\psi)\|_{\dot{\mathcal{N}}_{I}^{s_{k}}} & \lesssim\Big\|2^{j(s_{k}+\frac{1}{2})}\|\varphi_{\ll j}^{k}\psi_{\sim j}\|_{L_{x}^{1}L_{I}^{2}}\Big\|_{\ell_{j}^{2}}\\
 & \lesssim\Big(\sup_{j\in\mathbb{Z}}\|\varphi_{\ll j}\|_{L_{x}^{k}L_{I}^{\infty}}^{k}\Big)\Big\|2^{j(s_{k}+\frac{1}{2})}\|\psi_{\sim j}\|_{L_{x}^{\infty}L_{I}^{2}}\Big\|_{\ell_{j}^{2}}\\
 & \lesssim\|\varphi\|_{L_{x}^{k}L_{I}^{\infty}}^{k}\|\psi\|_{\dot{X}_{I}^{s_{k}}}.
\end{align*}
For the second estimate, observe that
\begin{align*}
\|g(u)\|_{\dot{X}_{I}^{s_{k}}} & \lesssim\Big\|2^{j(s_{k}+\frac{1}{2})}\|\sum_{r\gtrsim j}u_{\sim r}^{2}u_{\lesssim r}^{k-1}\|_{L_{x}^{1}L_{I}^{2}}\Big\|_{\ell_{j}^{2}}\\
 & \lesssim\Big(\sup_{r\in\mathbb{Z}}\|u_{\lesssim r}\|_{L_{x}^{k}L_{I}^{\infty}}^{k-1}\Big)\Big\|\sum_{r\gtrsim j}2^{(j-r)(s_{k}+\frac{1}{2})}\|2^{r(s_{k}+\frac{1}{2})}u_{\sim r}^{2}\|_{L_{x}^{k}L_{I}^{2}}\Big\|_{\ell_{j}^{2}}\\
 & \lesssim\|u\|_{L_{x}^{k}L_{I}^{\infty}}^{k-1}\Big\|2^{j(s_{k}+\frac{1}{2})}\|u_{\sim j}^{2}\|_{L_{x}^{k}L_{I}^{2}}\Big\|_{\ell_{j}^{2}},
\end{align*}
where we used Young's inequality in $j$. We then observe that
\[
\Big\|2^{j(s_{k}+\frac{1}{2})}\|u_{\sim j}^{2}\|_{L_{x}^{k}L_{I}^{2}}\Big\|_{\ell_{j}^{2}}=\Big\|2^{j(s_{k}+\frac{1}{2})/2}\|u_{\sim j}\|_{L_{x}^{2k}L_{I}^{4}}\Big\|_{\ell_{j}^{4}}^{2}.
\]
Using the embedding $\ell_{j}^{2}\hookrightarrow\ell_{j}^{4}$ and
the definition of $\dot{X}_{I}^{s_{k}}$ norm, we have the assertion.
\end{proof}
Vento's proof of critical local well-posedness of \eqref{eq:gBO}
goes as follows. The paraproduct decomposition allows us to rewrite
\eqref{eq:gBO} as so called \emph{distorted equation} 
\begin{equation}
(\partial_{t}+\mathcal{H}\partial_{xx})u+\pi(u_{0},u)=[\pi(u_{0},u)-\pi(u,u)]+g(u).\label{eq:distorted eq}
\end{equation}
Let us denote by $U(t)$ the linear propagator associated to \eqref{eq:distorted eq}.
As alluded to above, the propagator $U(t)$ still obeys analogous
linear estimates what the propagator $V(t)$ satisfy, at least for
short times (see Corollary \ref{cor:distorted linear estimates} or
\cite[Proposition 3.2]{Vento2010} where $T$ depends on $u_{0}\in\dot{H}^{s_{k}}$
instead of $\|u_{0}\|_{H^{\frac{1}{2}}}$). Moreover, further shrinking
$T$ if necessary, we may assume that $\|U(t)u_{0}-u_{0}\|_{L_{x}^{k}L_{T}^{\infty}}$
and $\|u\|_{\dot{X}_{T}^{s_{k}}}$ are sufficiently small. Hence,
using Duhamel's formula to \eqref{eq:distorted eq} and Lemma \ref{lem:nonlinearity estimate}
to the nonlinearity of \eqref{eq:distorted eq}, one can iterate on
the space $B_{\dot{X}_{T}^{s_{k}}}(0;r)\cap B_{L_{x}^{k}L_{T}^{\infty}}(u_{0};r)$
for sufficiently small $r$.

\subsection{\label{subsec:GWP}Subcritical Local Theory and Global Well-posedness}

In this subsection, we prove subcritical well-posedness of \eqref{eq:gBO}
in $H^{s}$ with $s>s_{k}$ and global well-posedness of \eqref{eq:gBO}
in $H^{s}$ with $s\geq\frac{1}{2}$. The critical local theory by
Vento \cite{Vento2010} guarantees profile dependent lifespans for
$\dot{H}^{s_{k}}$ and $H^{s}$ $(s\geq s_{k})$ solutions, but does
not answer the question of subcritical well-posedness. Here, we slightly
modify Vento's argument to obtain $H^{s}$ $(s>s_{k})$ subcritical
well-posedness, whose lifespan depends on $H^{s}$ norm of the initial
data. Afterward, we obtain the persistence of regularity and combine
with conservation laws to get the global well-posedness in $H^{s}$
$(s\geq\frac{1}{2})$.

We first obtain a subcritical version (Corollary \ref{cor:distorted linear estimates})
of the distorted linear estimates \cite[Proposition 3.2]{Vento2010}.
To do this, it suffices to obtain smallness of 
\[
\|V(t)u_{0}-u_{0}\|_{L_{x}^{k}L_{T}^{\infty}}\quad\text{and}\quad\|D_{x}^{0+}V(t)u_{0}\|_{L_{x}^{k+}L_{T}^{\infty-}}
\]
by choosing $T$ only depending on $H^{s}$ norm of $u_{0}$.
\begin{lem}
\label{lem:subcritical-V(t)-comparison}Let $s>s_{k}$, $u_{0}\in H^{s}$,
and $\eta>0$. Then, there exists $T=T(\|u_{0}\|_{H^{s}},s,\eta)>0$
such that 
\[
\|V(t)u_{0}-u_{0}\|_{L_{x}^{k}L_{T}^{\infty}\cap L_{T}^{\infty}\dot{H}_{x}^{s_{k}}}+\|D_{x}^{0+}V(t)u_{0}\|_{L_{x}^{k+}L_{T}^{\infty-}}<\eta.
\]
\end{lem}
\begin{proof}
We proceed as in \cite{Vento2010}. Observe that $u\coloneqq V(t)u_{0}-u_{0}$
solves
\begin{align*}
(\partial_{t}+\mathcal{H}\partial_{xx})u & =-\mathcal{H}\partial_{xx}u_{0},\\
u(0) & =0.
\end{align*}
By Duhamel's formula, we have
\begin{align*}
V(t)Q_{<N}u_{0}-Q_{<N}u_{0} & =-\int_{0}^{t}V(t-s)\mathcal{H}\partial_{xx}Q_{<N}u_{0}ds\\
 & =-\int_{0}^{t}V(t)\mathcal{H}\partial_{xx}Q_{<N}u_{0}ds.
\end{align*}
Therefore, we have 
\[
\|V(t)Q_{<N}u_{0}-Q_{<N}u_{0}\|_{L_{x}^{k}L_{T}^{\infty}\cap L_{T}^{\infty}\dot{H}_{x}^{s_{k}}}\lesssim T2^{2N}\|u_{0}\|_{\dot{H}^{s_{k}}}.
\]
On the other hand, by the triangle inequality and linear estimates,
we have
\[
\|V(t)Q_{\geq N}u_{0}-Q_{\geq N}u_{0}\|_{L_{x}^{k}L_{T}^{\infty}\cap L_{T}^{\infty}\dot{H}_{x}^{s_{k}}}\lesssim N^{-(s-s_{k})}\|u_{0}\|_{H^{s}}.
\]
Only depending on $H^{s}$ norm of $u_{0}$, choose $N\gg1$ large
and then $T>0$ small to obtain the first estimate. For the second
estimate, use Sobolev embedding in $x$, H\"older's inequality in
$t$, and linear estimates to have
\[
\|D_{x}^{0+}V(t)u_{0}\|_{L_{x}^{k+}L_{T}^{\infty-}}\lesssim T^{0+}\|u_{0}\|_{\dot{H}^{s_{k}+}}\lesssim T^{0+}\|u_{0}\|_{H^{s}}
\]
We can choose $T>0$ small to conclude.
\end{proof}
\begin{cor}[Distorted linear estimates]
\label{cor:distorted linear estimates}Let $s>s_{k}$. For $u_{0}\in H^{s}$,
there exists $T=T(\|u_{0}\|_{H^{s}},s)>0$ and a nondecreasing polynomial
$p_{k}$ such that whenever $u$ is a solution to
\begin{align*}
(\partial_{t}+\mathcal{H}\partial_{xx})u+\pi(u_{0},u) & =\tilde{f}\\
u(0) & =\tilde{u}_{0}
\end{align*}
for $\tilde{u}_{0}\in\dot{H}^{s_{k}}$ and $\tilde{f}\in\dot{\mathcal{N}}_{T}^{s_{k}}$,
we have $u\in\dot{Z}_{T}^{s_{k}}$ with the bound
\[
\|u\|_{\dot{Z}_{T}^{s_{k}}}\leq p_{k}(\|u_{0}\|_{\dot{H}^{s_{k}}})\big(\|\tilde{u}_{0}\|_{\dot{H}^{s_{k}}}+\|\tilde{f}\|_{\dot{\mathcal{N}}_{T}^{s_{k}}}\big).
\]
If in addition $\tilde{u}_{0}\in H^{s}$ and $\tilde{f}\in\mathcal{N}_{T}^{s}$,
then we have $u\in Z_{T}^{s}$ with the bound
\[
\|u\|_{Z_{T}^{s}}\leq p_{k}(\|u_{0}\|_{\dot{H}^{s_{k}}})\big(\|\tilde{u}_{0}\|_{H^{s}}+\|\tilde{f}\|_{\mathcal{N}_{T}^{s}}\big).
\]
\end{cor}
\begin{proof}
We only prove the first $\dot{H}^{s_{k}}$-critical estimate. The
remaining inhomogeneous estimate easily follows by mimicking the proof.
We proceed as in \cite{Vento2010}. We start from the estimate presented
in the proof of \cite[Proposition 3.2]{Vento2010}. With notations
$u_{j}=Q_{j}u$ and $u_{L,\ll j}=V(t)u_{0,\ll j}$ for any fixed $j\in\mathbb{Z}$,
our starting point is
\[
\|u_{j}\|_{\dot{\mathcal{S}}_{I}^{s_{k},0}\cap\dot{\mathcal{S}}_{I}^{s_{k},1}}\leq p_{k}(\|u_{0}\|_{\dot{H}^{s_{k}}})\big\{\|\tilde{u}_{0,j}\|_{\dot{H}^{s_{k}}}+\|\tilde{f}_{j}\|_{\dot{\mathcal{N}}_{T}^{s_{k}}}+(A+B+C+D)\big\}
\]
where 
\begin{align*}
A & =\|\partial_{x}(u_{0,\ll j})^{k}u_{j}\|_{\dot{\mathcal{N}}_{T}^{s_{k}}},\\
B & =\|(u_{0,\ll j})^{2k}u_{j}\|_{\dot{\mathcal{N}}_{T}^{s_{k}}},\\
C & =\|\partial_{x}[Q_{j},u_{0,\ll j}^{k}-u_{L,\ll j}^{k}]u_{\sim j}\|_{\dot{\mathcal{N}}_{T}^{s_{k}}},\\
D & =\|\partial_{x}[Q_{j},u_{L,\ll j}^{k}]u_{\sim j}\|_{\dot{\mathcal{N}}_{T}^{s_{k}}}.
\end{align*}
We show that $A+B\ll\|u_{j}\|_{\dot{X}_{T}^{s_{k}}}$. By Bernstein
estimates, we have 
\begin{align*}
\|\partial_{x}(u_{0,\ll j})^{k}u_{j}\|_{\dot{\mathcal{N}}_{T}^{s_{k}}} & \lesssim2^{j(s_{k}-\frac{1}{2})}2^{j-J}\|(u_{0,\ll j})^{k}\|_{L_{x}^{1}}\|u_{j}\|_{L_{x}^{\infty}L_{T}^{2}}\\
 & \lesssim2^{-J}\|u_{0}\|_{L_{x}^{k}}^{k}\|u_{j}\|_{\dot{X}_{T}^{s_{k}}}
\end{align*}
and 
\begin{align*}
\|(u_{0,\ll j})^{2k}u_{j}\|_{\dot{\mathcal{N}}_{T}^{s_{k}}} & \lesssim2^{j(s_{k}-\frac{1}{2})}\|(u_{0,\ll j})^{k}\|_{L_{x}^{1}}\|(u_{0,\ll j})^{k}\|_{L_{x}^{\infty}}\|u_{j}\|_{L_{x}^{\infty}L_{T}^{2}}\\
 & \lesssim2^{-J}\|u_{0}\|_{L_{x}^{k}}^{2k}\|v_{j}\|_{\dot{X}_{T}^{s_{k}}}.
\end{align*}
We choose sufficiently large $J=J(\|u_{0}\|_{L^{k}})$ to obtain $A+B\ll\|u_{j}\|_{\dot{X}_{T}^{s_{k}}}$.
Therefore, we have
\[
\|u_{j}\|_{\dot{\mathcal{S}}_{I}^{s_{k},0}\cap\dot{\mathcal{S}}_{I}^{s_{k},1}}\leq p_{k}(\|u_{0}\|_{\dot{H}^{s_{k}}})\big\{\|\tilde{u}_{0,j}\|_{\dot{H}^{s_{k}}}+\|\tilde{f}_{j}\|_{\dot{\mathcal{N}}_{T}^{s_{k}}}+(C+D)\big\}.
\]
From now on, we fix $J$ determined as above and claim that $C+D\ll\|u_{\sim j}\|_{\dot{X}_{T}^{s_{k}}}$
by choosing $T=T(\|u_{0}\|_{H^{s}})>0$ sufficiently small. On the
one hand, observe that 
\begin{align*}
C & \lesssim2^{j(s_{k}+\frac{1}{2})}\|u_{0,\ll j}^{k}-u_{L,\ll j}^{k}\|_{L_{x}^{1}L_{T}^{\infty}}\|u_{\sim j}\|_{L_{x}^{\infty}L_{T}^{2}}\\
 & \lesssim\|u_{0}-u_{L}\|_{L_{x}^{k}L_{T}^{\infty}}\|u_{0,\ll j}\|_{L_{x}^{k}}^{k-1}\|u_{\sim j}\|_{\dot{X}_{T}^{s_{k}}}.
\end{align*}
On the other hand, by the commutator estimate \cite[Lemma 2.4]{BurqPlanchon2006},
we obtain
\begin{align*}
D & \lesssim2^{j(s_{k}-\frac{1}{2})}\|\partial_{x}(u_{L,\ll j})^{k}\|_{L_{x}^{1+}L_{T}^{\infty-}}\|u_{\sim j}\|_{L_{x}^{\infty-}L_{T}^{2+}}\\
 & \lesssim2^{j(s_{k}-\frac{1}{2})}\|\partial_{x}u_{L,\ll j}\|_{L_{x}^{k+}L_{T}^{\infty-}}\|u_{L,\ll j}\|_{L_{x}^{k}L_{T}^{\infty}}^{k-1}\|u_{\sim j}\|_{L_{x}^{\infty-}L_{T}^{2+}}\\
 & \lesssim\|D_{x}^{0+}u_{L,\ll j}\|_{L_{x}^{k+}L_{T}^{\infty-}}\|u_{L,\ll j}\|_{L_{x}^{k}L_{T}^{\infty}}^{k-1}\|u_{\sim j}\|_{\dot{\mathcal{S}}_{T}^{s_{k},1-}}.
\end{align*}
Applying Lemma \ref{lem:subcritical-V(t)-comparison}, there is $T=T(\|u_{0}\|_{H^{s}})>0$
such that $C+D\ll\|u_{\sim j}\|_{\dot{X}_{T}^{s_{k}}}$. This yields
\[
\|u_{j}\|_{\dot{\mathcal{S}}_{I}^{s_{k},0}\cap\dot{\mathcal{S}}_{I}^{s_{k},1}}\leq p_{k}(\|u_{0}\|_{\dot{H}^{s_{k}}})(\|\tilde{u}_{0,j}\|_{\dot{H}^{s_{k}}}+\|\tilde{f}_{j}\|_{\dot{\mathcal{N}}_{T}^{s_{k}}}).
\]
Taking $\ell_{j}^{2}$ summation proves the assertion for $\dot{\mathcal{S}}_{T}^{s_{k},0}\cap\dot{\mathcal{S}}_{T}^{s_{k},1}$
norm.

In order to estimate $\dot{Z}_{T}^{s_{k}}$ norm of $u$, observe
that $u$ satisfies
\[
(\partial_{t}+\mathcal{H}\partial_{xx})u=-\pi(u_{0},u)+\tilde{f}.
\]
From Duhamel's formula and Lemma \ref{lem:nonlinearity estimate},
we obtain
\begin{align*}
\|u\|_{L_{x}^{k}L_{T}^{\infty}\cap L_{T}^{\infty}\dot{H}_{x}^{s_{k}}} & \lesssim\|V(t)\tilde{u}_{0}\|_{L_{x}^{k}L_{T}^{\infty}\cap L_{t}^{\infty}\dot{H}_{x}^{s_{k}}}+\|u_{0}\|_{L_{x}^{k}}^{k}\|u\|_{\dot{X}_{T}^{s_{k}}}+\|\tilde{f}\|_{\dot{\mathcal{N}}_{T}^{s_{k}}}\\
 & \leq p_{k}(\|u_{0}\|_{\dot{H}^{s_{k}}})(\|\tilde{u}_{0}\|_{\dot{H}^{s_{k}}}+\|\tilde{f}\|_{\dot{\mathcal{N}}_{T}^{s_{k}}}).
\end{align*}
To check the continuity condition \eqref{eq:continuity assumption 1 for (k,infty)},
consider the estimate
\begin{align*}
 & \|u-u(t_{0})\|_{L_{x}^{k}L_{I}^{\infty}\cap L_{I}^{\infty}\dot{H}_{x}^{s_{k}}}\\
 & \lesssim\|V(t-t_{0})u(t_{0})-u(t_{0})\|_{L_{x}^{k}L_{I}^{\infty}\cap L_{I}^{\infty}\dot{H}_{x}^{s_{k}}}+\|\pi(u_{0},u)\|_{\dot{\mathcal{N}}_{I}^{s_{k}}}+\|\tilde{f}\|_{\dot{\mathcal{N}}_{I}^{s_{k}}}
\end{align*}
for any interval $I\subset[-T,T]$ with $t_{0}\in I$. We then use
Lemma \ref{lem:subcritical-V(t)-comparison} for the first term and
Lebesgue's DCT for the remaining terms.
\end{proof}
\begin{rem}
\label{rem:remark 3.6}If one merely assumes that $u_{0}\in\dot{H}^{s_{k}}$
in Lemma \ref{lem:subcritical-V(t)-comparison}, then $T$ depends
on the profile of $u_{0}$. Hence, the corresponding distorted estimate
is only valid on the time interval whose length depends on the profile
of $u_{0}$. More precisely, one has the following. Let $u_{0}\in\dot{H}^{s_{k}}$
and $\eta>0$ be fixed. Then, there exists $T=T(u_{0},\eta)>0$ and
$r=r(u_{0},\eta)>0$ such that 
\[
\|V(t)u_{0}^{\ast}-u_{0}^{\ast}\|_{L_{x}^{k}L_{T}^{\infty}\cap L_{T}^{\infty}\dot{H}_{x}^{s_{k}}}+\|D_{x}^{0+}V(t)u_{0}^{\ast}\|_{L_{x}^{k+}L_{T}^{\infty-}}<\eta
\]
holds for $u_{0}^{\ast}\in\dot{H}^{s_{k}}$ with $\|u_{0}^{\ast}-u_{0}\|_{\dot{H}^{s_{k}}}<r$.
As a direct consequence, the distorted estimate for $u_{0}^{\ast}\in\dot{H}^{s_{k}}$
holds with $T=T(u_{0})>0$. See \cite{Vento2010} for details.
\end{rem}
We now turn into obtaining $H^{s}$ ($s>s_{k})$ norm dependent lifespan
of a solution. Denote by $U(t)u_{0}$ the solution to the equation
\begin{align*}
\partial_{t}u+\mathcal{H}\partial_{xx}u+\pi(u_{0},u) & =0,\\
u(0) & =u_{0}.
\end{align*}
Following the argument in \cite[Section 4.2]{Vento2010}, we should
obtain smallness of 
\[
\|U(t)u_{0}\|_{\dot{X}_{T}^{s_{k}}}\quad\text{and}\quad\|U(t)u_{0}-u_{0}\|_{L_{x}^{k}L_{T}^{\infty}\cap L_{T}^{\infty}\dot{H}_{x}^{s_{k}}}
\]
by choosing $T=T(\|u_{0}\|_{H^{\frac{1}{2}}})>0$ small.
\begin{lem}
\label{lem:subcritical U(t) comparison}Let $s>s_{k}$, $u_{0}\in H^{s}$,
and $\eta>0$. Then, there exists $T=T(\|u_{0}\|_{H^{s}},s,\eta)>0$
such that 
\[
\|U(t)u_{0}\|_{\dot{X}_{T}^{s_{k}}}+\|U(t)u_{0}-u_{0}\|_{L_{x}^{k}L_{T}^{\infty}\cap L_{T}^{\infty}\dot{H}_{x}^{s_{k}}}<\eta.
\]
\end{lem}
\begin{proof}
Recall that $\dot{X}_{T}^{s_{k}}$ norm does not contain $L_{T}^{\infty}$
component. Hence, we can use Sobolev embedding in space and H\"older's
inequality in time to have 
\[
\|U(t)u_{0}\|_{\dot{X}_{T}^{s_{k}}}\lesssim T^{0+}\|U(t)u_{0}\|_{\mathcal{S}_{T}^{s_{k}+,0}\cap\mathcal{S}_{T}^{s_{k}+,1}}\leq T^{0+}p_{k}(\|u_{0}\|_{\dot{H}^{s_{k}}})\|u_{0}\|_{\dot{H}^{s_{k}+}}.
\]
Possibly taking $T$ much smaller (only depending on $\eta$ and $H^{s}$
norm of $u_{0}$), we have 
\[
\|U(t)u_{0}\|_{\dot{X}_{T}^{s_{k}}}<\eta,
\]
which is the first assertion. In order to show the second assertion,
we observe that 
\begin{align*}
 & \|U(t)u_{0}-u_{0}\|_{L_{x}^{k}L_{T}^{\infty}\cap L_{T}^{\infty}\dot{H}_{x}^{s_{k}}}\\
 & \lesssim\|V(t)u_{0}-u_{0}\|_{L_{x}^{k}L_{T}^{\infty}\cap L_{T}^{\infty}\dot{H}_{x}^{s_{k}}}+\|[U(t)-V(t)]u_{0}\|_{L_{x}^{k}L_{T}^{\infty}\cap L_{T}^{\infty}\dot{H}_{x}^{s_{k}}}\\
 & \lesssim\|V(t)u_{0}-u_{0}\|_{L_{x}^{k}L_{T}^{\infty}\cap L_{T}^{\infty}\dot{H}_{x}^{s_{k}}}+\|u_{0}\|_{L_{x}^{k}}^{k}\|U(t)u_{0}\|_{\dot{X}_{T}^{s_{k}}}.
\end{align*}
Use the first assertion and Lemma \ref{lem:subcritical-V(t)-comparison}.
\end{proof}
\begin{cor}
\label{cor:subcritical lifespan of s_k solution}Let $u_{0}$ be a
initial data in $H^{s}$. Then, there exists $T=T(\|u_{0}\|_{H^{s}},s)>0$
such that $u_{0}$ admits a unique  $\dot{H}^{s_{k}}$-solution $u$
defined on $[0,T]$.
\end{cor}
\begin{proof}
Consider the operator
\[
[\Phi u](t)=U(t)u_{0}+\int_{0}^{t}U(t-t')f(u(t'))dt',
\]
where $f(u)=\pi(u_{0},u)-\pi(u,u)+g(u)$. We will iterate in the space
\[
B_{\dot{X}_{T}^{s_{k}}}(0;\eta)\cap B_{L_{x}^{k}L_{T}^{\infty}}(u_{0};\eta)\cap B_{L_{T}^{\infty}\dot{H}_{x}^{s_{k}}}(u_{0};\eta)
\]
with $0<\eta\ll\min\{1,\|u_{0}\|_{\dot{H}^{s_{k}}}\}$ chosen later.
For convenience, denote 
\[
\|u\|_{Y}\coloneqq\|u\|_{\dot{X}_{T}^{s_{k}}}+\|u-u_{0}\|_{L_{x}^{k}L_{T}^{\infty}}+\|u-u_{0}\|_{L_{T}^{\infty}\dot{H}_{x}^{s_{k}}}.
\]
Taking $T=T(\|u_{0}\|_{H^{s}},s,\eta)>0$ sufficiently small, combining
with Corollary \ref{cor:distorted linear estimates} with Lemma \ref{lem:nonlinearity estimate},
there exists $C=C(\|u_{0}\|_{\dot{H}^{s_{k}}})$ such that 
\[
\|\Phi u\|_{Y}\leq\|U(t)u_{0}\|_{Y}+C\|u\|_{Y}^{2}.
\]
On the other hand, since
\begin{align*}
\|\Phi u_{1}-\Phi u_{2}\|_{\dot{Z}_{T}^{s_{k}}} & \leq C\|f(u_{1})-f(u_{2})\|_{\dot{\mathcal{N}}_{T}^{s_{k}}}
\end{align*}
with 
\begin{multline*}
f(u_{1})-f(u_{2})=[\pi(u_{0},u_{1}-u_{2})-\pi(u_{1},u_{1}-u_{2})]-[\pi(u_{1},u_{2})-\pi(u_{2},u_{2})]\\
+[g(u_{1})-g(u_{2})],
\end{multline*}
we have
\[
\|\Phi u_{1}-\Phi u_{2}\|_{\dot{Z}_{T}^{s_{k}}}\leq C\|u_{1}-u_{2}\|_{\dot{Z}_{T}^{s_{k}}}(\|u_{1}\|_{\dot{X}_{T}^{s_{k}}}+\|u_{2}\|_{\dot{X}_{T}^{s_{k}}}).
\]
Therefore, we choose $\eta>0$ small and then $T=T(\|u_{0}\|_{H^{s}},s)>0$
much smaller (to guarantee smallness of $\|U(t)u_{0}\|_{Y}$ by Lemma
\ref{lem:subcritical U(t) comparison}), we see that $\Phi$ becomes
a well-defined contraction on $B_{Y}(\eta)$. By a fixed-point procedure,
we can find a solution in $\dot{Z}_{T}^{s_{k}}$. The remaining part
of the proof is standard.
\end{proof}
So far, if the initial data $u_{0}$ lies in $H^{s}$ with $s>s_{k}$,
we have constructed $\dot{H}^{s_{k}}$ solution with the lifespan
only depending on $\|u_{0}\|_{H^{s}}$. It turns out that this $\dot{H}^{s_{k}}$
solution can be upgraded to the $H^{s}$ solution.
\begin{prop}[Persistence of regularity]
\label{prop:persistence of regularity}Let $s>s_{k}$. Suppose that
we have a $\dot{H}^{s_{k}}$ solution $u\in\dot{Z}_{I}^{s_{k}}$ for
some compact time interval $I$ with $u(t_{0})\in H^{s}$ for some
$t_{0}\in I$. Then, $u\in Z_{I}^{s}$.
\end{prop}
\begin{proof}
By time translation and reversing symmetries, we may assume that $I=[0,T]$
and $t_{0}=0$. Let $u_{0}\coloneqq u(t_{0})\in H^{s}$. Applying
the critical well-posedness (Theorem \ref{thm:LWP}), we can construct
$H^{s}$ solution $\tilde{u}$ with initial data $u_{0}\in H^{s}$.
Suppose that this $H^{s}$ solution $\tilde{u}$ is defined on $[0,T_{1})$
where $T_{1}>0$ is chosen to be maximal. If $T_{1}>T$, then $\tilde{u}=u$
on $[0,T]$ by uniqueness, so the conclusion follows easily.

We now consider the case when $T_{1}\leq T$. In this case, uniqueness
only tells us that $\tilde{u}=u$ on $[0,T_{1})$ and we do not know
whether $\tilde{u}$ can be defined at time $T_{1}$ or not. Due to
continuity of $t\mapsto u(t)\in\dot{H}^{s_{k}}$ at $t=T_{1}$ and
Remark \ref{rem:remark 3.6}, there exists $\delta>0$ such that for
any $t_{1}\in[T_{1}-\delta,T_{1}]$ the following distorted linear
estimate is valid on the interval $[T_{1}-\delta,T_{1}]$: if $u$
solves
\[
(\partial_{t}+\mathcal{H}\partial_{xx})u+\pi(u(t_{1}),u)=\tilde{f}
\]
then we have the bound
\[
\|u\|_{Z_{[t_{1},T_{1}-\eta]}^{s}}\lesssim\|u(t_{1})\|_{H^{s}}+\|\tilde{f}\|_{\mathcal{N}_{[t_{1},T_{1}-\eta]}^{s}}.
\]
for any $\eta\in(0,T_{1}-t_{1})$. We remark that the implicit constant
above does not depend on choice of $t_{1}$ and $\eta$. In our situation,
we substitute $\tilde{f}=\pi(u(t_{1}),u)-\pi(u,u)+g(u)$. If we mimic
the proof of Lemma \ref{lem:nonlinearity estimate}, then (where $\tilde{I}\coloneqq[t_{1},T_{1}-\eta]$)
\[
\|\tilde{f}\|_{\mathcal{N}_{\tilde{I}}^{s}}\lesssim\|u-u(t_{1})\|_{L_{x}^{k}L_{[t_{1},T_{1}]}^{\infty}}^{k}\|u\|_{X_{\tilde{I}}^{s}}+\|u\|_{L_{x}^{k}L_{[t_{1},T_{1}]}^{\infty}}^{k-1}\|u\|_{\dot{X}_{[t_{1},T_{1}]}^{s_{k}}}\|u\|_{X_{\tilde{I}}^{s}}.
\]
Because
\[
\lim_{t_{1}\uparrow T_{1}}\big(\|u-u(t_{1})\|_{L_{x}^{k}L_{[t_{1},T_{1}]}^{\infty}}+\|u\|_{\dot{X}_{[t_{1},T_{1}]}^{s_{k}}}\big)=0,
\]
we have $\|\tilde{f}\|_{\mathcal{N}_{[t_{1},T_{1}-\eta]}^{s}}\ll\|u\|_{X_{[t_{1},T_{1}-\eta]}^{s}}$
whenever $t_{1}$ is sufficiently close to $T_{1}$. Fixing such $t_{1}$,
we have 
\[
\|u\|_{Z_{[t_{1},T_{1}-\eta]}^{s}}\lesssim\|u(t_{1})\|_{H^{s}}.
\]
As $\eta\in(0,T_{1}-t_{1})$ is arbitrary, we obtain in particular
\[
\sup_{t\in[0,T_{1})}\|u(t)\|_{H^{s}}<\infty.
\]
Since $u(t)\to u(T_{1})$ in $\dot{H}^{s_{k}}$ as $t\uparrow T_{1}$,
we have
\[
\|u(T_{1})\|_{H^{s}}\leq\liminf_{t\uparrow T_{1}}\|u(t)\|_{H^{s}}<\infty.
\]
Therefore, $u(T_{1})$ lies in $H^{s}$ and we can construct $H^{s}$
solution at time $T_{1}$ both forward and backward in time. This
should extend the solution $\tilde{u}$, which yields a contradiction.
\end{proof}
\begin{proof}[Proof of Theorem \ref{thm:subcritical LWP, GWP}]
(Subcritical LWP) Let $s>s_{k}$. For any $u_{0}\in H^{s}$, by Corollary
\ref{cor:subcritical lifespan of s_k solution}, there exists $T=T(\|u_{0}\|_{H^{s}})>0$
and $\dot{H}^{s_{k}}$ solution $u\in\dot{Z}_{T}^{s_{k}}$ with initial
data $u_{0}$. By persistence of regularity (Proposition \ref{prop:persistence of regularity}),
the solution $u$ indeed lies in $Z_{T}^{s}$. This proves existence
part. The uniqueness part and continuity of the solution map is standard.

(GWP and conservation laws) Note that persistence of regularity and
local well-posedness guarantees that we can approximate any $H^{\frac{1}{2}}$
solutions by smooth solutions. Since the mass and energy functionals
are $H^{\frac{1}{2}}$ continuous, we obtain mass and energy conservation
laws. Combining with the subcritical well-posedness, the global well-posedness
follows for $H^{\frac{1}{2}}$ solutions. If $s\geq\frac{1}{2}$,
then we can use persistence of regularity to transfer the results
for $H^{\frac{1}{2}}$ solutions to $H^{s}$ solutions.
\end{proof}
\begin{rem}
Even in case of the focusing (gBO) with $k\geq4$, subcritical well-posedness
part of Theorem \ref{thm:subcritical LWP, GWP} still holds. Therefore,
one can have global well-posedness of focusing (gBO) with initial
data satisfying certain mass, energy, and kinetic energy assumptions.
See \cite[Theorem 1.1]{FarahLinaresPastor2014}.
\end{rem}

\section{\label{sec:existence of critical elt}Existence of Critical Element}

From this section, we start proving Theorem \ref{thm:scattering}.
The scheme of the proof is a compactness-contradiction argument, which
originates from the pioneering work of Kenig and Merle in the setting
of energy-critical nonlinear Schr\"odinger and wave equation \cite{Kenig-Merle2006Invent,KenigMerle2008Acta}.
The argument then successfully implemented in the scattering problem
of various semilinear equations, such as nonlinear Schr\"odinger
or wave equations. In case of mass-critical and mass-supercritical
\eqref{eq:gKdV}, see \cite{Dodson2017AnnPDE} and \cite{FarahLinaresPastorVisciglia2017}.

Having the global well-posedness, we first derive a criterion that
determines whether a solution scatters or not. It will be written
in terms of finiteness of the spacetime norms $\dot{X}_{t}^{s_{k}}$
and $L_{x}^{k}L_{t}^{\infty}$. Having established the criterion,
suppose that Theorem \ref{thm:scattering} fails. Then, there exists
a critical element that does not scatter and attains minimal mass/energy.
A genuine property of this critical element is that it should stay
in a compact set modulo symmetries of the equation (modulo spatial
translations in our setting.) This section is devoted to obtain existence
of such a critical element and its compactness property. To this end,
as in other contexts, we use the profile decomposition and perturbation
theory. Theorem \ref{thm:scattering} will be proved in Section \ref{sec:precllusion of a.p. sol}
once we show that such a critical element indeed cannot exist.

In other literatures, the linear profile decomposition and long-time
perturbation theory are used to obtain the nonlinear profile decomposition,
that is, the sum of nonlinear profiles becomes an approximate solution.
In our setting, however, we encounter a technical difficulty to deduce
long-time perturbation theory from local well-posedness. It is because
the iteration norms contain $L^{\infty}$-type norms $L_{x}^{k}L_{t}^{\infty}$
and $\dot{X}_{t}^{s_{k}}$, in which subdivision of time interval
does not guarantee the smallness of $L_{x}^{k}L_{I_{j}}^{\infty}$
for a short time interval $I_{j}$. Instead of obtaining a general
form of long-time perturbation, we directly prove that nonlinear profile
decomposition holds.

In Section \ref{subsec:scattering criterion}, we derive the aforementioned
criterion. In Section \ref{subsec:profile decomposition}, we derive
and discuss more about the linear profile decomposition. In Section
\ref{subsec:perturbation theory}, we prove the nonlinear profile
decomposition. In Section \ref{subsec:existence of minimal blowup solution},
we finally construct a critical element. The last subsection (Section
\ref{subsec:proof of error estimation}) is devoted to estimate the
error term appeared in Section \ref{subsec:perturbation theory}.

\subsection{\label{subsec:scattering criterion}Scattering Criterion}

In this subsection, we derive the scattering criterion in terms of
spacetime norms. This is nothing but the local theory at time $\pm\infty$.
In contrast to the usual local theory, we cannot make small $L_{x}^{\infty}L_{T+}^{2}$
norm with choosing $T$ large.\footnote{In view of the sharp local smoothing estimate $\|D_{x}^{1/2}V(t)u_{0}\|_{L_{x}^{\infty}L_{t}^{2}}\sim\|u_{0}\|_{L_{x}^{2}}$,
we do not expect that $L_{x}^{\infty}L_{[T,+\infty)}^{2}$ norm becomes
small.} Fortunately, we can have small $L_{x}^{k}L_{T+}^{\infty}$ norm and
this allows us to achieve local well-posedness at time $+\infty$.
This will be done for the original formulation of \eqref{eq:gBO},
not the distorted equation \eqref{eq:distorted eq}.
\begin{lem}[Vanishing of $L_{x}^{k}L_{T+}^{\infty}$]
\label{lem:vanishing of (k,infty) for large time}Let $k\geq4$.
For any $\eta>0$ and $\varphi\in\dot{H}^{s_{k}}$, there exists $T=T(\eta,\varphi)<\infty$
such that 
\[
\|V(t)\varphi\|_{L_{x}^{k}L_{T+}^{\infty}}<\eta.
\]
\end{lem}
\begin{proof}
From the linear estimates, we have $\|V(t)\varphi\|_{L_{x}^{k}L_{t}^{\infty}}\lesssim\|\varphi\|_{\dot{H}^{s_{k}}}$.
By density argument, we may assume that $\varphi$ is Schwartz. By
DCT, it suffices to show that for each $x\in\mathbb{R}$, $\|[V(t)\varphi](x)\|_{L_{T+}^{\infty}}$
goes to zero as $T\to+\infty$. This follows by the usual dispersive
estimate.
\end{proof}
As we shall use the formulation \eqref{eq:gBO} instead of \eqref{eq:distorted eq},
we need to estimate the whole nonlinearity $F(u)$ of \eqref{eq:gBO}.
The following nonlinear estimates are direct consequences of Lemma
\ref{lem:nonlinearity estimate}.
\begin{cor}[More estimates]
\label{cor:more nonlinearity estimate}Let $k\geq4$ and $I$ be
an interval. We have
\begin{align*}
\|F(u)\|_{\dot{\mathcal{N}}_{I}^{s_{k}}} & \lesssim\|u\|_{L_{x}^{k}L_{I}^{\infty}}^{k-1}\|u\|_{\dot{Z}_{I}^{s_{k}}}^{2},\\
\|F(u)-F(v)\|_{\dot{\mathcal{N}}_{I}^{s_{k}}} & \lesssim\|u-v\|_{\dot{Z}_{I}^{s_{k}}}\big(\|u\|_{L_{x}^{k}L_{I}^{\infty}}^{k-2}+\|v\|_{L_{x}^{k}L_{I}^{\infty}}^{k-2}\big)\big(\|u\|_{\dot{Z}_{I}^{s_{k}}}^{2}+\|v\|_{\dot{Z}_{I}^{s_{k}}}^{2}\big).
\end{align*}
\end{cor}
\begin{proof}
Apply Lemma \ref{lem:nonlinearity estimate} for the following expressions.
\begin{align*}
F(u) & =-\pi(u,u)+g(u)\\
F(u)-F(v) & =-\pi(u,u-v)-[\pi(u,v)-\pi(v,v)]+[g(u)-g(v)].
\end{align*}
\end{proof}
\begin{prop}[Existence of wave operator]
\label{prop:LWP at infty}Let $k\geq4$.\\
1. For any $\tilde{u}_{+}\in\dot{H}^{s_{k}}$, there exists $r>0$
and $T=T(\tilde{u}_{+})<+\infty$ such that any $u_{+}\in B_{\dot{H}^{s_{k}}}(\tilde{u}_{+};r)$
admits a unique $\dot{H}^{s_{k}}$ solution $u$ in $C_{T+}\dot{H}^{s_{k}}\cap L_{x}^{k}L_{T+}^{\infty}\cap\dot{X}_{T+}^{s_{k}}$
which scatters to $V(t)u_{+}$ in $\dot{H}^{s_{k}}$ forward in time.
Moreover, the solution $u$ indeed lies in $\dot{Z}_{T+}^{s_{k}}$
and the solution map $u_{+}\mapsto u\in\dot{Z}_{T+}^{s_{k}}$ is locally
Lipschitz.\\
2. Let $s\geq s_{k}$. Then, the $H^{s}$-version of the result holds.
That is, the same result holds for $C_{T+}H^{s}\cap L_{x}^{k}L_{T+}^{\infty}\cap X_{T+}^{s}$,
$H^{s}$, and $Z_{T+}^{s}$.
\end{prop}
\begin{proof}
We only prove the first statement, as the second one can be proven
in a similar manner. Given $u_{+}\in\dot{H}^{s_{k}}$, consider the
operator 
\[
[\Phi u](t)=V(t)u_{+}-\int_{t}^{\infty}V(t-s)F(u(s))ds.
\]
Let us consider the norm
\[
\|u\|_{\Lambda}\coloneqq\|u\|_{L_{x}^{k}L_{T+}^{\infty}}+\delta\|u\|_{\dot{X}_{T+}^{s_{k}}\cap L_{T+}^{\infty}\dot{H}_{x}^{s_{k}}},
\]
where $0<\delta\ll1$ to be chosen later. We shall apply the contraction
mapping principle on $B_{\Lambda}(0;\delta^{\frac{3}{4}})$.

We present linear and nonlinear estimates. For the linear evolution,
we estimate
\[
\|V(t)u_{+}\|_{\Lambda}\lesssim\|V(t)\tilde{u}_{+}\|_{L_{x}^{k}L_{T+}^{\infty}}+\|\tilde{u}_{+}-u_{+}\|_{\dot{H}^{s_{k}}}+\delta\|u_{+}\|_{\dot{H}^{s_{k}}}.
\]
For the nonlinear estimate, we apply Corollary \ref{cor:more nonlinearity estimate}
to obtain
\begin{align*}
\|F(u)-F(v)\|_{\dot{\mathcal{N}}_{T+}^{s_{k}}} & \lesssim\|u-v\|_{\Lambda}(\|u\|_{L_{x}^{k}L_{T+}^{\infty}}^{k-2}+\|v\|_{L_{x}^{k}L_{T+}^{\infty}}^{k-2})(\|u\|_{\Lambda}^{2}+\|v\|_{\Lambda}^{2})\\
 & \lesssim\delta^{\frac{3k}{4}-2}\|u-v\|_{\dot{Z}_{T+}^{s_{k}}}.
\end{align*}
If in particular $v=0$, then 
\[
\|F(u)\|_{\dot{\mathcal{N}}_{T+}^{s_{k}}}\lesssim\delta^{\frac{3k}{4}-2}\|u\|_{\Lambda}.
\]
Therefore,
\begin{align*}
\|\Phi u\|_{\Lambda} & \lesssim\|V(t)u_{+}\|_{L_{x}^{k}L_{T+}^{\infty}}+r+\delta\|u_{+}\|_{\dot{H}^{s_{k}}}+\delta^{\frac{3k}{4}-2}\|u\|_{\Lambda},\\
\|\Phi u-\Phi v\|_{\Lambda} & \lesssim\delta^{\frac{3k}{4}-2}\|u-v\|_{\Lambda}.
\end{align*}
If we choose $\delta$ sufficiently small, $r$ small, and choosing
$T=T(u_{+})$ sufficiently large, then $\Phi$ becomes a contraction.

We now show that any solution $u\in C_{T+}\dot{H}^{s_{k}}\cap L_{x}^{k}L_{T+}^{\infty}\cap\dot{X}_{T+}^{s_{k}}$
indeed lies in $\dot{Z}_{T+}^{s_{k}}$. We only show \eqref{eq:continuity assumption 2 for (k,infty)},
that is, $\|u\|_{L_{x}^{k}L_{T+}^{\infty}}\to0$ as $T\to+\infty$.
To show this, by the Duhamel formula,
\[
\|u\|_{L_{x}^{k}L_{T+}^{\infty}}\leq\|V(t)u_{+}\|_{L_{x}^{k}L_{T+}^{\infty}}+\|F(u)\|_{\dot{\mathcal{N}}_{T+}^{s_{k}}}.
\]
As $T\to+\infty$, we use Lemma \ref{lem:vanishing of (k,infty) for large time}
for the linear evolution and DCT for the nonlinear evolution to obtain
$u\in\dot{Z}_{T+}^{s_{k}}$.
\end{proof}
We now state the scattering criterion. In contrast to the case of
mass-critical \eqref{eq:gKdV}, we should keep track of two norms
$L_{x}^{k}L_{T+}^{\infty}$ and $\dot{X}_{T+}^{s_{k}}$. As we discussed
in Section \ref{sec:WP theory}, we have one derivative in the nonlinearity,
but the local smoothing estimate recovers at most half derivative.
Thus we are forced to use $L^{\infty}$ type spacetime norms. Consult
Section \ref{eq:double local smoothing}. 
\begin{prop}[Scattering Criterion]
\label{prop:scattering criterion}Let $u$ be a global $H^{\frac{1}{2}}$
solution satisfying $\|u\|_{L_{x}^{k}L_{T+}^{\infty}}+\|u\|_{\dot{X}_{T+}^{s_{k}}}<\infty$
for some $T<+\infty$. Then, $u$ scatters in $H^{\frac{1}{2}}$ forward
in time. In particular, $u$ belongs to $Z_{T+}^{\frac{1}{2}}$. The
analogous statement holds for backward in time.
\end{prop}
\begin{proof}
We first show that $u$ scatters in $\dot{H}^{s_{k}}$. To this end,
it suffices to show that 
\[
t\mapsto\int_{0}^{t}V(-s)F(u(s))ds
\]
converges in $\dot{H}^{s_{k}}$ as $t\to+\infty$. By linear estimates,
observe that 
\[
\Big\|\int_{t_{0}}^{t_{1}}V(-s)F(u(s))ds\Big\|_{\dot{H}^{s_{k}}}\lesssim\|F(u)\|_{\dot{\mathcal{N}}_{[t_{0},t_{1}]}^{s_{k}}}.
\]
Since $\|F(u)\|_{\dot{\mathcal{N}}_{T+}^{s_{k}}}<\infty$, Lebesgue's
DCT implies $\dot{H}^{s_{k}}$ scattering. Denote the scattering state
by $u_{+}\in\dot{H}^{s_{k}}$.

To conclude that $u$ scatters in $H^{\frac{1}{2}}$, observe from
the mass/energy conservation that $u_{+}$ indeed belongs to $H^{\frac{1}{2}}$.
Applying Proposition \ref{prop:LWP at infty}, $u$ indeed lies in
$Z_{T+}^{\frac{1}{2}}$ and scatters to $u_{+}$.
\end{proof}
For small initial data, we have both global well-posedness and scattering.
\begin{thm}[Small Data GWP and Scattering]
\label{thm:small data theory}Let $k\geq4$.\\
1. There exists $\eta>0$ such that any initial data $u_{0}\in\dot{H}^{s_{k}}$
with $\|u_{0}\|_{\dot{H}^{s_{k}}}<\eta$ admits a unique  $\dot{H}^{s_{k}}$-solution
$u$ to (gBO) in $\dot{Z}_{t}^{s_{k}}$. Moreover, the solution map
from $B_{\dot{H}^{s_{k}}}(0;\eta)$ to $\dot{Z}_{t}^{s_{k}}$ is Lipschitz.\\
2. The solution $u$ scatters in $\dot{H}^{s_{k}}$ both forward and
backward in time.\\
3. For $s\geq s_{k}$, the $H^{s}$-version of the result holds.
\end{thm}
\begin{proof}
The proof is standard once one obtains Corollary \ref{cor:more nonlinearity estimate}
and exploit smallness of the initial data. Indeed, given $u_{0}\in\dot{H}^{s_{k}}$
with $\|u_{0}\|_{\dot{H}^{s_{k}}}\leq\eta$, we consider the operator
$\Phi$ defined by 
\[
[\Phi u](t)\coloneqq V(t)u_{0}+\int_{0}^{t}V(t-s)F(s)ds.
\]
For some $r$ chosen later, we iterate on the space $B_{\dot{Z}_{t}^{s_{k}}}(0;r)$.
By the linear estimates and Corollary \ref{cor:more nonlinearity estimate},
we have
\begin{align*}
\|\Phi u\|_{\dot{Z}_{t}^{s_{k}}} & \lesssim\|u_{0}\|_{\dot{H}^{s_{k}}}+\|u\|_{\dot{Z}_{t}^{s_{k}}}^{k+1}\\
 & \lesssim\eta+r^{k+1}\\
\|\Phi u-\Phi v\|_{\dot{Z}_{t}^{s_{k}}} & \lesssim\|u-v\|_{\dot{Z}_{t}^{s_{k}}}\big(\|u\|_{\dot{Z}_{t}^{s_{k}}}^{k}+\|v\|_{\dot{Z}_{t}^{s_{k}}}^{k}).\\
 & \lesssim r^{k}\|u-v\|_{\dot{Z}_{t}^{s_{k}}}.
\end{align*}
By choosing $r$ small and then $\eta$ small, $\Phi$ becomes a contraction.
The remaining parts of the proof are fairly standard.
\end{proof}
\begin{rem}
In fact, the small data global well-posedness and scattering can be
obtained using the usual local smoothing norms instead of Besov type
norms. See \cite[Appendix A.1]{MolinetRibaud2004}. There, one avoids
paraproduct decomposition of the nonlinearity, but should use fractional
Leibniz rules. In the following, we only need local theory in terms
of $\dot{Z}_{t}^{s_{k}}$ norm.
\end{rem}

\subsection{\label{subsec:profile decomposition}Linear Profile Decomposition}

One of the main ingredients of compactness-contradiction argument
is the profile decomposition. This type of results was intensively
exploited in the study of critical dispersive equations for last decades.
We start with the linear profile decomposition. Consider the linear
estimate (for $k\geq4$)
\[
\|V(t)u_{0}\|_{L_{x}^{k}L_{t}^{\infty}}\lesssim\|u_{0}\|_{H^{\frac{1}{2}}}.
\]
This embedding is not compact due to two \emph{noncompact} \emph{symmetries}:\\
1. time translations $u_{0}\mapsto V(t_{0})u_{0}$ for any $t_{0}\in\mathbb{R}$,
and\\
2. spatial translations $u_{0}\mapsto u_{0}(\cdot-x_{0})$ for any
$x_{0}\in\mathbb{R}$.\\

\begin{rem}
If we replace $\|u_{0}\|_{H^{1/2}}$ by $\|u_{0}\|_{\dot{H}^{s_{k}}}$,
then we have one more additional symmetry: the scaling symmetry $u_{0}\mapsto\lambda^{1/k}u_{0}(\lambda\cdot)$
for any $\lambda>0$.
\end{rem}
Roughly speaking, the linear profile decomposition says that these
symmetries are essentially all the sources of lack of compactness
for the linear estimate $H^{\frac{1}{2}}\to L_{x}^{k}L_{t}^{\infty}$.
Let us state the linear profile decomposition.
\begin{prop}[Linear Profile Decomposition]
\label{prop:linear profile decomposition}Let $k>4$ and $\{u_{n}\}_{n\in\mathbb{N}}$
be bounded in $H^{\frac{1}{2}}$. After passing to a subsequence in
$n$ if necessary, there exist profiles $\{\phi^{j}\}_{j\in\mathbb{N}}\subset H^{\frac{1}{2}}$,
spatial parameters $\{x_{n}^{j}\}_{n,j\in\mathbb{N}}\subset\mathbb{R}$,
and time parameters $\{t_{n}^{j}\}_{n,j\in\mathbb{N}}\subset\mathbb{R}$,
and defining $w_{n}^{J}$ for each $J\in\mathbb{N}$ by 
\[
u_{n}=\sum_{j=1}^{J}V(t_{n}^{j})\phi^{j}(\cdot-x_{n}^{j})+w_{n}^{J},
\]
satisfying the following properties.\\
1. (Asymptotic orthogonality in $\dot{H}^{s}$) For any $s\in[0,\frac{1}{2}]$
and $J\in\mathbb{N}$, we have
\begin{equation}
\lim_{n\to\infty}\Big[\|u_{n}\|_{\dot{H}^{s}}^{2}-\sum_{j=1}^{J}\|\phi^{j}\|_{\dot{H}^{s}}^{2}-\|w_{n}^{J}\|_{\dot{H}^{s}}^{2}\Big]=0.\label{eq:asymp ortho H^s}
\end{equation}
2. (Asymptotic vanishing of the remainder) We have
\begin{equation}
\lim_{J\to\infty}\limsup_{n\to\infty}\|V(t)w_{n}^{J}\|_{L_{x}^{k}L_{t}^{\infty}}=0.\label{eq:asymp vanishing (k,infty)}
\end{equation}
3. (Asymptotic vanishing of weak limit) For any $j\leq J<\infty$,
we have
\begin{equation}
V(-t_{n}^{j})w_{n}^{J}(\cdot+x_{n}^{j})\rightharpoonup0\quad\text{weakly in }H^{\frac{1}{2}}.\label{eq:vanishing of weak limit}
\end{equation}
4. (Asymptotic separation of parameters) For each $j\neq j'$, we
have
\begin{equation}
|x_{n}^{j}-x_{n}^{j'}|+|t_{n}^{j}-t_{n}^{j'}|\to\infty.\label{eq:separation of parameters}
\end{equation}
\end{prop}
\begin{proof}
The proof is standard, especially very similar to that in \cite[Lemma 2.1]{HolmerRoudenko2008}.
For sake of completeness, we include it in Appendix \ref{sec:proof of linear profile decomposition}.
\end{proof}
\begin{rem}
Although the embedding $\|V(t)u_{0}\|_{L_{x}^{k}L_{t}^{\infty}}\lesssim\|u_{0}\|_{H^{\frac{1}{2}}}$
is true for $k\geq4$, we do not know whether \eqref{eq:asymp vanishing (k,infty)}
is true for $k=4$. The $L_{x}^{4}L_{t}^{\infty}$ estimate is at
the genuine endpoint in local smoothing estimates. As one can see
in the proof of Proposition \ref{prop:linear profile decomposition},
one has to use interpolation. If \eqref{eq:asymp vanishing (k,infty)}
were true for $k=4$, Theorem \ref{thm:scattering} holds for $k=4$.
\end{rem}
\begin{rem}
\label{rem:no asym vanish of X}One may ask whether \eqref{eq:asymp vanishing (k,infty)}
is true if we replace $L_{x}^{k}L_{t}^{\infty}$-norm by $\dot{X}_{t}^{s_{k}}$-norm.
However, this seems to be impossible. Indeed, the linear estimate
$\|D_{x}^{1/2}V(t)u_{0}\|_{L_{x}^{\infty}L_{t}^{2}}\sim\|u_{0}\|_{L_{x}^{2}}$
says that $\|V(t)w_{n}^{J}\|_{\dot{X}_{t}^{s_{k}}}\sim\|w_{n}^{J}\|_{\dot{H}^{s_{k}}}$.
Hence, asymptotic vanishing of $\dot{X}_{t}^{s_{k}}$-norm is equivalent
to asymptotic vanishing of $\dot{H}^{s_{k}}$-norm, which seems to
be impossible.
\end{rem}
Let us discuss the statements of Proposition \ref{prop:linear profile decomposition}.
There can be infinitely many profiles $\phi^{j}$. However, the time
and spatial parameters associated to each profiles should be far enough,
so that each (time and spatial translated) profile contributes to
$u_{n}$ almost orthogonally. After deleting the contributions of
profiles, the remainder term should asymptotically vanish in $L_{x}^{k}L_{t}^{\infty}$
norm. This is the heart of compactness modulo symmetries.

\subsection{\label{subsec:perturbation theory}Nonlinear Profile Decomposition}

As we are dealing with the nonlinear equation, we will obtain a nonlinear
version of the linear profile decomposition. We refer to as \emph{the}
\emph{nonlinear profile decomposition} (Theorem \ref{thm:nonlinear profile decomposition})
\emph{holds} if the sum of nonlinear profiles approximate nonlinear
solutions to the original equation. In other contexts, this is achieved
by the long-time perturbation theory. The long-time perturbation theory
is basically a consequence of concatenating short-time perturbation
theory, in which one exploits smallness of solution norms on each
small time interval. In our case, however, the solution norm contains
$L_{x}^{k}L_{I}^{\infty}$ norm that may not be small even for a short
time interval $I$. Thus obtaining the long-time perturbation and
hence nonlinear profile decomposition here becomes more delicate.

We now introduce the nonlinear profile decomposition in more concrete
terms. Let $\{u_{n}(0)\}_{n\in\mathbb{N}}$ be a bounded sequence
in $H^{\frac{1}{2}}$. We apply the linear profile decomposition (Proposition
\ref{prop:linear profile decomposition}) to $\{u_{n}(0)\}_{n\in\mathbb{N}}$
and follow the notations made in that proposition.

Possibly taking a further subsequence, we may assume that $t_{n}^{j}$
converges in $[-\infty,+\infty]$. If $t_{n}^{j}\to t^{j}\in(-\infty,+\infty)$,
then we may replace $\phi^{j}$ by $V(t^{j})\phi^{j}$ and $t_{n}^{j}$
by $t_{n}^{j}-t^{j}$ to assume that $t_{n}^{j}\to0$. Moreover, replacing
$w_{n}^{J}$ by $w_{n}^{J}+\sum_{j=1}^{J}[V(t_{n}^{j})\phi^{j}-\phi^{j}]$,
we may assume that $t_{n}^{j}\equiv0$.

By global well-posedness of \eqref{eq:gBO} (Theorem \ref{thm:subcritical LWP, GWP}),
there exist global  $H^{\frac{1}{2}}$ solutions $v^{j}$ for each
$j\in\mathbb{N}$ such that $v^{j}(0)=\phi^{j}$ if $t_{n}^{j}\equiv0$,
$v^{j}$ scatters forward to $V(t)\phi^{j}$ if $t_{n}^{j}\to+\infty$,
and $v^{j}$ scatters to forward to $V(t)\phi^{j}$ if $t_{n}^{j}\to-\infty$.
These $v^{j}$'s are referred to as \emph{nonlinear profiles} associated
to $\{u_{n}(0)\}_{n\in\mathbb{N}}$. We define
\begin{align*}
v_{n}^{j}(t) & \coloneqq v^{j}(t+t_{n}^{j},\cdot-x_{n}^{j}),\\
u_{n}^{J}(t) & \coloneqq\sum_{j=1}^{J}v_{n}^{j}(t),\\
\tilde{u}_{n}^{J}(t) & \coloneqq u_{n}^{J}(t)+V(t)w_{n}^{J}.
\end{align*}
Note that $\tilde{u}_{n}^{J}$ and $u_{n}^{J}$ are globally defined.
The approximate solutions $\{\tilde{u}_{n}^{J}\}_{n,J\in\mathbb{N}}$
will be referred to as a \emph{nonlinear profile decomposition} associated
to $\{u_{n}(0)\}_{n\in\mathbb{N}}$.\footnote{Here, we use the terminology ``nonlinear profile decomposition'' in
two different ways. On the one hand, we call $\{\tilde{u}_{n}^{J}\}_{n,J\in\mathbb{N}}$
as a nonlinear profile decomposition. On the other hand, the statement
of Theorem \ref{thm:nonlinear profile decomposition} is said that
nonlinear profile decomposition holds.} We now state the main theorem of this subsection.
\begin{thm}[Nonlinear Profile Decomposition]
\label{thm:nonlinear profile decomposition}Consider a nonlinear
profile decomposition $\{\tilde{u}_{n}^{J}\}_{n,J\in\mathbb{N}}$
associated to a bounded sequence $\{u_{n}(0)\}_{n\in\mathbb{N}}\subset H^{\frac{1}{2}}$.
Let $I\subseteq\mathbb{R}$ be a fixed interval (possibly $I=\mathbb{R}$)
containing $0$. If $\|v^{j}\|_{\dot{Z}_{I}^{s_{k}}}<\infty$ for
all $j\in\mathbb{N}$, then 
\[
\lim_{J\to\infty}\limsup_{n\to\infty}\|\tilde{u}_{n}^{J}-u_{n}\|_{\dot{Z}_{I}^{s_{k}}}=0.
\]
In particular, combining this with Lemma \ref{lem:uniform boundedness of approximate solutions}
below, we have
\[
\limsup_{n\to\infty}\|u_{n}\|_{\dot{Z}_{I}^{s_{k}}}<\infty.
\]
\end{thm}
Theorem \ref{thm:nonlinear profile decomposition} will be proved
after we establish an appropriate perturbation theory. We now list
some properties of the nonlinear profile decomposition associated
to $\{u_{n}(0)\}_{n\in\mathbb{N}}$. Namings of the lemmas presented
below follow \cite{Killip-Visan2013Clay}.
\begin{lem}[Asymptotic agreement of initial data]
\label{lem:asymptotic agreement of initial data}We have
\[
\sup_{J}\limsup_{n\to\infty}\|\tilde{u}_{n}^{J}(0)-u_{n}(0)\|_{H^{\frac{1}{2}}}=0.
\]
\end{lem}
\begin{proof}
Observe that 
\[
\|\tilde{u}_{n}^{J}(0)-u_{n}(0)\|_{H^{\frac{1}{2}}}=\sum_{j\leq J:|t_{n}^{j}|\to\infty}\|v^{j}(t_{n}^{j})-V(t_{n}^{j})\phi^{1}\|_{H^{\frac{1}{2}}}.
\]
The conclusion follows by the construction of $v^{j}$.
\end{proof}
\begin{lem}[Uniform boundedness of approximate solutions]
\label{lem:uniform boundedness of approximate solutions}Let $I\subseteq\mathbb{R}$
be a fixed interval (possibly $I=\mathbb{R}$). If $\|v^{j}\|_{\dot{Z}_{I}^{s_{k}}}<\infty$
for all $j\in\mathbb{N}$, we have
\[
\sup_{J}\limsup_{n\to\infty}\big(\|\tilde{u}_{n}^{J}\|_{\dot{Z}_{I}^{s_{k}}}+\|\tilde{u}_{n}^{J}\|_{L_{I}^{\infty}H_{x}^{\frac{1}{2}}}\big)<\infty.
\]
\end{lem}
\begin{proof}
For notational simplicity, we only prove the case $I=\mathbb{R}$.
Fix $0<\eta_{0}\ll1$ such that we can apply the small data theory
for $\phi^{j}$ having $H^{\frac{1}{2}}$-norm less than $\eta_{0}$.
Choose $J_{1}\in\mathbb{N}$ such that 
\[
\Big(\sum_{j>J_{1}}\|\phi^{j}\|_{H^{\frac{1}{2}}}^{2}\Big)^{\frac{1}{2}}\leq\eta_{0}.
\]
We then see by Duhamel's principle that 
\begin{align*}
 & \sup_{J}\limsup_{n\to\infty}\|\sum_{j>J_{1}}^{J}v_{n}^{j}\|_{\dot{Z}_{t}^{s_{k}}}\\
 & \lesssim\sup_{J}\limsup_{n\to\infty}\|\sum_{j>J_{1}}^{J}v_{n}^{j}(0)\|_{\dot{H}^{s_{k}}}+\sum_{j>J_{1}}^{J}\|\pi(v^{j},v^{j})\|_{\dot{\mathcal{N}}_{t}^{s_{k}}}+\|g(v^{j})\|_{\dot{\mathcal{N}}_{t}^{s_{k}}}\\
 & \lesssim\sup_{J}\Big(\sum_{j>J_{1}}^{J}\|\phi^{j}\|_{\dot{H}^{s_{k}}}^{2}\Big)^{\frac{1}{2}}+\sum_{j>J_{1}}^{J}\|\phi^{j}\|_{\dot{H}^{s_{k}}}^{k+1}\\
 & \lesssim\eta_{0}.
\end{align*}
Applying exactly the same method with replacing $\dot{H}^{s_{k}}$
by $H^{\frac{1}{2}}$, we have
\[
\sup_{J}\limsup_{n\to\infty}\|\sum_{j>J_{1}}^{J}v_{n}^{j}\|_{L_{t}^{\infty}H_{x}^{\frac{1}{2}}}\lesssim\eta_{0}.
\]
Since only finitely many profiles are left, (applying mass/energy
conservation when we bound $L_{t}^{\infty}H_{x}^{\frac{1}{2}}$) the
conclusion follows.
\end{proof}
\begin{lem}[Error estimation]
\label{lem:error estimation}Let $I\subseteq\mathbb{R}$ be a fixed
interval (possibly $I=\mathbb{R}$). If $\|v^{j}\|_{\dot{Z}_{I}^{s_{k}}}<\infty$
for all $j\in\mathbb{N}$, we have
\[
\lim_{J\to\infty}\limsup_{n\to\infty}\|F(\tilde{u}_{n}^{J})-\sum_{j=1}^{J}F(v_{n}^{j})\|_{\dot{\mathcal{N}}_{I}^{s_{k}}}=0.
\]
\end{lem}
\begin{proof}
We postpone the proof in Section \ref{subsec:proof of error estimation}.
\end{proof}
We now go back to our discussion of technical difficulties arising
from long-time perturbation theory. It is instructive to compare with
the mass-critical \eqref{eq:gKdV} long-time perturbation theory \cite{KillipVisanKwonShao2012},
which we shall state as follows. Suppose that $\tilde{u}$ solves
\[
(\partial_{t}+\partial_{xxx})\tilde{u}=\partial_{x}(\tilde{u}^{5})+e
\]
with initial data $\tilde{u}_{0}$ and satisfies 
\[
\|\tilde{u}_{0}\|_{L_{I}^{\infty}L_{x}^{2}}\leq M\quad\text{and}\quad\|\tilde{u}\|_{L_{x}^{5}L_{I}^{10}}\leq L
\]
for some finite $M$ and $L$. Then, there exists $\epsilon_{0}=\epsilon_{0}(M,L)>0$
such that whenever 
\[
\|u_{0}-\tilde{u}_{0}\|_{L_{x}^{2}}+\|D_{x}^{-1}e\|_{L_{x}^{1}L_{I}^{2}}\leq\epsilon
\]
for some $0<\epsilon<\epsilon_{0}$, we have 
\[
\|u-\tilde{u}\|_{L_{x}^{5}L_{I}^{10}}\lesssim_{M,L}\epsilon.
\]
The case when $L$ is chosen sufficiently small constant is called
the short-time perturbation theory.

The main step in proving the above perturbation theory is to subdivide
the interval $I$ into $L$-dependently many subintervals (say $j=1,\dots,N(J)$)
on which $\tilde{u}$ has sufficiently small $L_{x}^{5}L_{I_{j}}^{10}$
norm. Once $\tilde{u}$ has small $L_{x}^{5}L_{I_{j}}^{10}$ norm,
by a short-time perturbation, one can conclude that $\tilde{u}$ is
close to $u$. One then inductively applies the short-time perturbation,
but only $L$-dependently many times to have the desired long-time
perturbation.

In our case, $\dot{X}_{I}^{s_{k}}$ and $L_{x}^{k}L_{I}^{\infty}$
are the iteration spaces in place of $L_{x}^{5}L_{I}^{10}$. The problem
is that, as $\dot{X}_{I}^{s_{k}}$ and $L_{x}^{k}L_{I}^{\infty}$
contains $L^{\infty}$ norm (either in $x$ or $t$), subdivision
of the interval $I$ into small subintervals does not guarantee small
$\dot{X}_{I_{j}}^{s_{k}}$ or $L_{x}^{k}L_{I_{j}}^{\infty}$ norm.
Thus, we do not expect the long-time perturbation analogous to that
of \eqref{eq:gKdV}, which is written in terms of $M$ and $L$. Therefore,
we shall not attempt to obtain such a general long-time perturbation
theory. Instead, we directly prove the nonlinear profile decomposition
(Theorem \ref{thm:nonlinear profile decomposition}).

The first observation in proving the nonlinear profile decomposition
is that, as in many other literatures, short-time perturbation is
merely a slight generalization of the local theory (Section \ref{sec:WP theory}).
This still holds in our setting, so we can obtain short-time perturbation
similar to that of \eqref{eq:gKdV}. In other words, the short-time
perturbation can be obtained by exploiting the smallness of $\|v^{j}-v^{j}(0)\|_{L_{x}^{k}L_{T}^{\infty}}$
and $\|v^{j}\|_{\dot{X}_{T}^{s_{k}}}$ to the distorted equation \eqref{eq:distorted eq}
when $T$ is sufficiently small. This justifies approximating two
close solutions on compact time intervals. However, this cannot take
care of the case when we compare two solutions on unbounded intervals.

In order to take care of unbounded intervals, our second observation
is to utilize Proposition \ref{prop:LWP at infty}. In Proposition
\ref{prop:LWP at infty}, we saw that smallness of $L_{x}^{k}L_{T+}^{\infty}$
is obtained for scattering solutions, so a perturbation theory on
$[T,+\infty)$ can be obtained.

Therefore, we must obtain two types of perturbation lemmas. The first
one is associated to the usual local theory, which takes care of compact
intervals. The second one is associated to the local theory at time
$\pm\infty$, which takes care of remaining unbounded intervals.
\begin{lem}[Perturbation lemma, I]
\label{lem:short-time 1}Suppose that $\tilde{u}$ and $u$ solve
\begin{align*}
(\partial_{t}+\mathcal{H}\partial_{xx})\tilde{u} & =F(\tilde{u})+e\\
(\partial_{t}+\mathcal{H}\partial_{xx})u & =F(u)
\end{align*}
and satisfy
\[
\|\tilde{u}\|_{\dot{Z}_{I}^{s_{k}}}+\|\tilde{u}\|_{L_{I}^{\infty}H_{x}^{\frac{1}{2}}}\leq M.
\]
Then, there exists $\epsilon_{0}=\epsilon_{0}(M)>0$ such that whenever
$I$ has length shorter than $\epsilon_{0}$ and 
\begin{align*}
\|\tilde{u}-\tilde{u}(t_{0})\|_{L_{x}^{k}L_{I}^{\infty}}+\|\tilde{u}\|_{\dot{X}_{I}^{s_{k}}} & \leq\epsilon_{0},\\
\|\tilde{u}(t_{0})-u(t_{0})\|_{\dot{H}^{s_{k}}} & \leq\epsilon<\epsilon_{0},\\
\|e\|_{\dot{\mathcal{N}}_{I}^{s_{k}}} & \leq\epsilon<\epsilon_{0}
\end{align*}
for some $t_{0}\in I$ and $\epsilon\in(0,\epsilon_{0})$, we have
\[
\|\tilde{u}-u\|_{\dot{Z}_{I}}\lesssim\epsilon.
\]
\end{lem}
\begin{proof}
Let $w=\tilde{u}-u$. As we do not have smallness of $\|\tilde{u}\|_{L_{x}^{k}L_{I}^{\infty}}$,
we use
\begin{align*}
(\partial_{t}+\mathcal{H}\partial_{xx})\tilde{u}+\pi(\tilde{u}(t_{0}),\tilde{u}) & =[\pi(\tilde{u}(t_{0}),\cdot)-\pi(\tilde{u},\cdot)](\tilde{u})+g(\tilde{u})+e,\\
(\partial_{t}+\mathcal{H}\partial_{xx})u+\pi(\tilde{u}(t_{0}),u) & =[\pi(\tilde{u}(t_{0}),\cdot)-\pi(u,\cdot)](u)+g(u).
\end{align*}
Therefore, we have
\begin{align*}
 & (\partial_{t}+\mathcal{H}\partial_{xx})w+\pi(\tilde{u}(t_{0}),w)\\
 & =[\pi(\tilde{u}(t_{0}),\cdot)-\pi(\tilde{u},\cdot)](w)+[\pi(\tilde{u},\cdot)-\pi(u,\cdot)](u)+g(\tilde{u})-g(u).
\end{align*}
As we are assuming that $I$ is short enough, we apply Corollary \ref{cor:distorted linear estimates}
and Lemma \ref{lem:nonlinearity estimate} to obtain
\[
\|w\|_{\dot{Z}_{I}^{s_{k}}}\lesssim p_{k}(M)\big[\|w(t_{0})\|_{\dot{H}^{s_{k}}}+\epsilon_{0}M^{k-1}\|w\|_{\dot{Z}_{I}^{s_{k}}}+\epsilon_{0}^{2}M^{k-2}\|w\|_{\dot{Z}_{I}^{s_{k}}}+p_{k,M}(\|w\|_{\dot{Z}_{I}^{s_{k}}})\big]
\]
for some higher order polynomial $p_{k,M}$. By a continuity argument,
we have 
\[
\|w\|_{\dot{Z}_{I}^{s_{k}}}\lesssim\|w(t_{0})\|_{\dot{H}^{s_{k}}}\lesssim\epsilon.
\]
\end{proof}
\begin{lem}[Perturbation lemma, II]
\label{lem:short-time 2}Suppose that $\tilde{u}$ and $u$ solve
\begin{align*}
(\partial_{t}+\mathcal{H}\partial_{xx})\tilde{u} & =F(\tilde{u})+e\\
(\partial_{t}+\mathcal{H}\partial_{xx})u & =F(u)
\end{align*}
and satisfy
\[
\|\tilde{u}\|_{\dot{Z}_{I}^{s_{k}}}\leq M.
\]
Then there exists $\epsilon_{0}=\epsilon_{0}(M)>0$ such that whenever
\begin{align*}
\|\tilde{u}\|_{L_{x}^{k}L_{I}^{\infty}} & \leq\epsilon_{0},\\
\|\tilde{u}(t_{0})-u(t_{0})\|_{\dot{H}^{s_{k}}} & \leq\epsilon<\epsilon_{0},\\
\|e\|_{\dot{\mathcal{N}}_{I}^{s_{k}}} & \leq\epsilon<\epsilon_{0}
\end{align*}
for some $t_{0}\in I$ and $\epsilon\in(0,\epsilon_{0})$, we have
\[
\|\tilde{u}-u\|_{\dot{Z}_{I}^{s_{k}}}\lesssim\epsilon.
\]
\end{lem}
\begin{proof}
Let $w=\tilde{u}-u$. By Duhamel's formula, 
\[
w(t)=V(t-t_{0})w(t_{0})+\int_{t_{0}}^{t}V(t-s)[F(\tilde{u})-F(u)+e](s)ds.
\]
By Corollary \ref{cor:more nonlinearity estimate} and our hypothesis,
we have 
\begin{align*}
\|w\|_{\dot{Z}_{I}^{s_{k}}} & \lesssim\|w(t_{0})\|_{\dot{H}^{s_{k}}}+\|F(\tilde{u})-F(u)\|_{\dot{\mathcal{N}}_{I}^{s_{k}}}+\|e\|_{\dot{\mathcal{N}}_{I}^{s_{k}}}\\
 & \lesssim\|w(t_{0})\|_{\dot{H}^{s_{k}}}+\epsilon_{0}^{k-2}M^{2}\|w\|_{\dot{Z}_{I}^{s_{k}}}+\|e\|_{\dot{\mathcal{N}}_{I}^{s_{k}}}+p_{k,M}(\|w\|_{\dot{Z}_{t}^{s_{k}}}),
\end{align*}
for some higher order polynomial $p_{k,M}$. By a continuity argument,
we have 
\[
\|w\|_{\dot{Z}_{I}^{s_{k}}}\lesssim\|w(t_{0})\|_{\dot{H}^{s_{k}}}\lesssim\epsilon.
\]
\end{proof}
With the above perturbation lemmas in hand, we explain how we handle
the long-time control of $\tilde{u}_{n}^{J}$ and $u_{n}$. In view
of uniform boundedness of initial data in $H^{\frac{1}{2}}$, there
are only finitely many large profiles. Note that there can be infinitely
many small profiles, but they will not cause problems due to almost
orthogonality and small data global theory. We immediately estimate
the sum of small profiles by a small $Z$ norm.

It remains to treat large profiles, which are finitely many! Because
each nonlinear profile scatters by our hypothesis in Theorem \ref{thm:nonlinear profile decomposition},
we can truncate intervals $(-\infty,-T]$ and $[T,+\infty)$ on which
local theory at time $\pm\infty$ (Lemma \ref{lem:short-time 2})
holds. Then, the remaining intermediate interval is compact. We then
subdivide that compact interval into finitely many short subintervals
so that we can apply Lemma \ref{lem:short-time 1}. This is the main
idea of Lemma \ref{lem:subdivision of R} how we can chop the time
interval into subintervals on which perturbation lemmas are applicable,
and hence, prove Theorem \ref{thm:nonlinear profile decomposition}.
\begin{proof}[Proof of Theorem \ref{thm:nonlinear profile decomposition}]
From Lemma \ref{lem:uniform boundedness of approximate solutions},
we fix $M<\infty$ satisfying 
\[
\sup_{J}\limsup_{n\to\infty}\big(\|\tilde{u}_{n}^{J}\|_{\dot{Z}_{I}^{s_{k}}}+\|\tilde{u}_{n}^{J}\|_{L_{I}^{\infty}H_{x}^{\frac{1}{2}}}\big)<M.
\]
Choose $\eta=\epsilon_{0}(2M)>0$ in Lemmas \ref{lem:short-time 2}
and \ref{lem:short-time 1}. We then apply the following lemma.
\begin{lem}[Subdivision of $\mathbb{R}$]
\label{lem:subdivision of R}Let $\{\tilde{u}_{n}^{J}\}_{n,J\in\mathbb{N}}$
and $M$ as above. For any $\eta>0$, there exists $N\in\mathbb{N}$
such that for all large $J$ and $n>n(J)$, there exist intervals
$I_{n,1}^{J}$, $I_{n,2}^{J}$, $\dots$, $I_{n,N}^{J}$ partitioning
$I$ such that every $I_{n,\ell}^{J}$ satisfies one of the following
properties:\\
1. (type-I) we have $\|\tilde{u}_{n}^{J}-\tilde{u}_{n}^{J}(\inf I_{n,\ell})\|_{L_{x}^{k}L_{I_{n,\ell}^{J}}^{\infty}}+\|\tilde{u}_{n}^{J}\|_{\dot{X}_{I_{n,\ell}^{J}}^{s_{k}}}\lesssim\eta$
and the length of $I_{n,\ell}^{J}$ is sufficiently short so that
we can apply Corollary \ref{cor:distorted linear estimates}.\\
2. (type-II) we have $\|\tilde{u}_{n}^{J}\|_{L_{x}^{k}L_{I_{n,\ell}^{J}}^{\infty}}\lesssim\eta$.
\begin{proof}
For notational simplicity, we only prove for the case of $I=\mathbb{R}$.
Let us first consider small data part. Choose $J_{1}\in\mathbb{N}$
such that 
\[
\Big(\sum_{j>J_{1}}\|\phi^{j}\|_{H^{\frac{1}{2}}}^{2}\Big)^{\frac{1}{2}}\lesssim\eta
\]
and small data theory is applicable for each $v^{j}$, $j>J_{1}$.
As in the proof of Lemma \ref{lem:uniform boundedness of approximate solutions},
we have
\[
\lim_{J\to\infty}\limsup_{n\to\infty}\|\sum_{j>J_{1}}^{J}v_{n}^{j}\|_{\dot{Z}_{t}^{s_{k}}}\lesssim\eta.
\]

We now consider large profiles $v_{n}^{1},\dots,v_{n}^{J_{1}}$. For
each $v_{n}^{j}$, by Proposition \ref{prop:scattering criterion}
and definition of the space $\dot{Z}_{t}^{s_{k}}$, we can partition
$\mathbb{R}$ into 
\[
\mathbb{R}=(-\infty,s_{n,1}^{j}]\cup[s_{n,1}^{j},s_{n,2}^{j}]\cup\cdots\cup[s_{n,N_{j}-1}^{j},s_{n,N_{j}}^{j}]\cup[s_{n,N_{j}}^{j},+\infty)
\]
so that $(-\infty,s_{n,1}^{j}]$ and $[s_{n,N_{j}}^{j},\infty)$ are
type-II intervals with $v_{n}^{j}$, $\frac{\eta}{J_{1}}$ in place
of $\tilde{u}_{n}^{J}$, $\eta$, respectively and $[s_{n,1}^{j},s_{n,2}^{j}]$,
$[s_{n,2}^{j},s_{n,3}^{j}]$, $\dots$, and $[s_{n,N_{j}-1}^{j},s_{n,N_{j}}^{j}]$
are type-I intervals with $v_{n}^{j}$, $\frac{\eta}{J_{1}}$ in place
of $\tilde{u}_{n}^{J}$, $\eta$, respectively. We may further assume
that $|s_{n,\ell}^{j}-s_{n,\ell-1}^{j}|<\epsilon$ for all $\ell=2,\dots,N_{j}$,
where $\epsilon>0$ satisfies the following property:
\[
\sup_{1\leq j\leq J_{1}}\Big(\|v^{j}-v^{j}(t_{0})\|_{L_{x}^{k}L_{\tilde{I}}^{\infty}}+\|v^{j}\|_{\dot{X}_{\tilde{I}}^{s_{k}}}\Big)\lesssim\frac{\eta}{J_{1}}
\]
whenever $\tilde{I}$ is a time interval whose length is shorter than
$\epsilon$ and $t_{0}\in\tilde{I}$. Note that such $\epsilon$ exists
because each $v^{j}$ lies in $Z_{t}^{\frac{1}{2}}$ (by Proposition
\ref{prop:scattering criterion}) and we can use regularity. We then
consider the refinement of the all intervals found above and use the
triangle inequality.

For the remaining term, note that 
\[
\lim_{J\to\infty}\limsup_{n\to\infty}\|V(t)w_{n}^{J}\|_{L_{x}^{k}L_{t}^{\infty}}=0,
\]
so it suffices to consider $\dot{X}_{t}^{s_{k}}$-norm on type-I intervals
(say $I$). By the Sobolev embedding and H\"older's inequality, we
have
\[
\|V(t)w_{n}^{J}\|_{\dot{X}_{I}^{s_{k}}}\lesssim\epsilon^{0+}\|w_{n}^{J}\|_{H^{\frac{1}{2}}}\lesssim\epsilon^{0+}M.
\]
Therefore, it suffices to choose $\epsilon=\epsilon(M,\eta)>0$ small
so that $\epsilon^{0+}M\lesssim\eta$ and Corollary \ref{cor:distorted linear estimates}
is applicable on intervals having length less than $\epsilon$.
\end{proof}
\end{lem}
From Lemma \ref{lem:subdivision of R}, we have finitely many (profile
decomposition dependently, not $n$ or $J$ dependently) intervals
partitioning $\mathbb{R}$ on which Lemma \ref{lem:short-time 2}
and \ref{lem:short-time 1} are applicable. Because we have asymptotic
agreement of initial data (Lemma \ref{lem:asymptotic agreement of initial data})
and asymptotic vanishing of error (Lemma \ref{lem:error estimation}),
we have the conclusion.
\end{proof}
Local well-posedness theory guarantees that if two initial data are
close, then corresponding solutions are close on a compact time interval
measured in $Z$ norm. Due to Theorem \ref{thm:nonlinear profile decomposition},
if $u$ is a scattering solution, we have a perturbation result in
a neighborhood of $u$, even on an unbounded time interval (say $[T,+\infty)$).
More precisely, we have the following corollary.
\begin{cor}[Closeness of solutions on an unbounded interval]
\label{cor:closeness of two solutions on unbounded interval}Let
$u$ and $u_{n}$ be $H^{\frac{1}{2}}$ solutions to (gBO) such that
$\|u\|_{\dot{Z}_{I}^{s_{k}}}<\infty$ and $u_{n}(t_{0})\to u(t_{0})$
in $H^{\frac{1}{2}}$ for some $t_{0}\in I$. Then, $\|u_{n}-u\|_{\dot{Z}_{I}^{s_{k}}}\to0$.
\end{cor}
\begin{proof}
Since $u_{n}(t_{0})$ converges to $u(t_{0})$ in $H^{\frac{1}{2}}$,
\[
\tilde{u}_{n}(t)\coloneqq u(t)+V(t-t_{0})(u_{n}(t_{0})-u(t_{0}))
\]
becomes a nonlinear profile decomposition associated to $\{u_{n}(t_{0})\}_{n\in\mathbb{N}}$
(with $J=1$). As $\|u\|_{\dot{Z}_{I}^{s_{k}}}<\infty$, we can use
the longtime perturbation theory to obtain 
\[
\lim_{n\to\infty}\|\tilde{u}_{n}-u_{n}\|_{\dot{Z}_{I}^{s_{k}}}=0.
\]
Because $\|\tilde{u}_{n}-u\|_{\dot{Z}_{I}^{s_{k}}}\lesssim\|u_{n}(t_{0})-u(t_{0})\|_{\dot{H}^{s_{k}}}$
goes to zero, we have
\[
\lim_{n\to\infty}\|\tilde{u}_{n}-u\|_{\dot{Z}_{I}^{s_{k}}}=0.
\]
This completes the proof.
\end{proof}

\subsection{\label{subsec:existence of minimal blowup solution}Existence of
the Critical Element}

Having established the nonlinear profile decomposition, we now construct
a critical element. Define functionals 
\begin{align*}
A(u) & \coloneqq M(u)+E(u),\\
S_{\mathbb{R}}(u) & \coloneqq\|u\|_{\dot{X}_{t}^{s_{k}}}+\|u\|_{L_{x}^{k}L_{t}^{\infty}},\\
L(A) & \coloneqq\sup\{S_{\mathbb{R}}(u):A(u)\leq A\}.
\end{align*}
Suppose that Theorem \ref{thm:scattering} fails. Then, there exists
some $A_{0}<\infty$ satisfying $L(A_{0})=+\infty$. We define a critical
value
\[
A_{c}\coloneqq\inf\{A:L(A)=+\infty\}.
\]
By the small data theory (Theorem \ref{thm:small data theory}), we
have $0<A_{c}\leq A_{0}<\infty$. In other words, $A_{c}$ is the
threshold in a sense that every solution having $A$ less than $A_{c}$
should scatter but a solution having $A$ greater than $A_{c}$ may
not scatter.
\begin{lem}
\label{lem:continuity of L}The function $L$ is continuous at $A=A_{c}$
and $L(A_{c})=+\infty$.
\end{lem}
\begin{proof}
We first show that $L(A_{c})=+\infty$. Suppose not; we assume $L(A_{c})<+\infty$.
Then, there exists a sequence $\{u_{n}(0)\}_{n\in\mathbb{N}}$ in
$H^{\frac{1}{2}}$ such that $A(u_{n})\downarrow A_{c}$ but $S_{\mathbb{R}}(u_{n})\uparrow+\infty$.
Consider a nonlinear profile decomposition associated to $u_{n}$.
Note that all profiles have $A$ less than or equal to $A_{c}$. Therefore,
we can apply longtime perturbation (Theorem \ref{thm:nonlinear profile decomposition})
and conclude that $\limsup_{n\to\infty}S_{\mathbb{R}}(u_{n})<\infty$.
This is absurd.

We now show that $L(A)\uparrow+\infty$ as $A\uparrow A_{c}$. For
any sufficiently large $M$, we can choose $u$ such that $A(u)=A_{c}$
and $S_{\mathbb{R}}(u)>M$. Then, we can choose a compact interval
$I\subset\mathbb{R}$ such that $S_{I}(u)>M$. Consider an initial
data $u_{0,\epsilon}=(1-\epsilon)u_{0}$. By local well-posedness,
we have $S_{I}(u_{0,\epsilon})>M$ for a sufficiently small $\epsilon$.
This completes the proof.
\end{proof}
The next proposition establishes the compactness property of critical
elements.
\begin{prop}[Palais-Smale Condition]
\label{prop:Palais-Smale condition}Suppose that $\{u_{n}\}_{n\in\mathbb{N}}$
is a sequence of  $H^{\frac{1}{2}}$-solutions satisfying $A(u_{n})\leq A_{c}$
and 
\[
\lim_{n\to\infty}S_{\geq t_{n}}(u_{n})=\lim_{n\to\infty}S_{\leq t_{n}}(u_{n})=+\infty
\]
for some $t_{n}$. Then, possibly taking a subsequence of $u_{n}$,
there exist spatial parameters $\{x_{n}\}_{n\in\mathbb{N}}$ such
that $u_{n}(t_{n},\cdot+x_{n})$ converges in $H^{\frac{1}{2}}$.
\end{prop}
\begin{proof}
By time translation, we may assume that $t_{n}\equiv0$. Consider
a nonlinear profile decomposition associated to $u_{n}$. We divide
into two cases.

We first consider the case when $\sup_{j}A(\phi^{j})<A_{c}$. By definition
of $A_{c}$, every $v^{j}$ has finite scattering norm. We apply Theorem
\ref{thm:nonlinear profile decomposition} to obtain $\limsup_{n\to\infty}S_{\mathbb{R}}(u_{n})<\infty$.
This is absurd.

The next case is when $\sup_{j}A(\phi^{j})=A_{c}$. By asymptotic
orthogonality in $H^{\frac{1}{2}}$, we have only one profile $u_{n}=v_{n}^{1}+V(t)w_{n}^{1}$
with $w_{n}^{1}\to0$ strongly in $H^{\frac{1}{2}}$. If $t_{n}^{1}\to+\infty$,
then $v^{1}$ scatters in forward by construction of $v^{1}$. We
combine
\[
S_{\geq0}(u_{n})\leq S_{\geq t_{n}^{1}}(v^{1})+\|w_{n}^{1}\|_{\dot{H}^{s_{k}}}
\]
with Proposition \ref{prop:LWP at infty} to obtain
\[
\limsup_{n\to\infty}S_{\geq0}(u_{n})\leq\limsup_{n\to\infty}S_{\geq t_{n}^{1}}(v^{1})<\infty,
\]
which is absurd. Similarly, we can exclude the case $t_{n}^{1}\to-\infty$.
Therefore, we conclude $t_{n}^{j}\equiv0$. In this case, $u_{n}(0,\cdot+x_{n})=\phi^{1}+w_{n}^{1}(\cdot+x_{n})$
converges to $\phi^{1}$ strongly in $H^{\frac{1}{2}}$.
\end{proof}
A global solution $u\in C_{t}H_{x}^{\frac{1}{2}}$ to \eqref{eq:gBO}
is \emph{almost periodic modulo spatial translations} (in short, \emph{almost
periodic}) if the set of cosets represented by $\{u(t):t\in\mathbb{R}\}$
is precompact in the quotient topology of $H^{\frac{1}{2}}$ modulo
spatial translation. Equivalently, there is $x(t)\in\mathbb{R}$ for
each $t$ such that the set
\[
\{u(t,\cdot+x(t)):t\in\mathbb{R}\}
\]
is precompact in $H^{\frac{1}{2}}$. 
\begin{thm}[Existence of the critical element]
\label{thm:existence of a.p. solutions}Suppose that Theorem \ref{thm:scattering}
fails. Then, there exists an almost periodic solution $u\in C_{t}H_{x}^{\frac{1}{2}}$
that does not scatter either forward nor backward in time.
\end{thm}
\begin{proof}
From continuity of $L$ at $A=A_{c}$ (Lemma \ref{lem:continuity of L}),
we can choose $u_{n}$ such that $A(u_{n})<A_{c}$, $A(u_{n})\uparrow A_{c}$,
and $S_{\mathbb{R}}(u_{n})\uparrow+\infty$. We claim that there exist
$t_{n}\in\mathbb{R}$ such that 
\[
\lim_{n\to\infty}S_{\geq t_{n}}(u_{n})=\lim_{n\to\infty}S_{\leq t_{n}}(u_{n})=+\infty.
\]
Because $A(u_{n})<A_{c}$, the solutions $u_{n}$ scatter in $H^{\frac{1}{2}}$
both forward and backward in time. Therefore, we can choose $\tilde{t}_{n}\in\mathbb{R}$
large such that $S_{\geq\tilde{t}_{n}}(u_{n})\lesssim A_{c}$. Hence,
using continuity of $\dot{Z}_{t}^{s_{k}}$ norm, we can choose $t_{n}\in\mathbb{R}$
for all large $n\in\mathbb{N}$ such that $S_{\geq t_{n}}(u_{n})=\frac{1}{2}S_{\mathbb{R}}(u_{n})$.
We then have $S_{\leq t_{n}}(u_{n})\geq S_{\mathbb{R}}(u_{n})-S_{\geq t_{n}}(u_{n})=\frac{1}{2}S_{\mathbb{R}}(u_{n})$
for all large $n\in\mathbb{N}$. The claim is now proved.

By Proposition \ref{prop:Palais-Smale condition}, (possibly passing
to a subsequence) there exist $x_{n}\in\mathbb{R}$ such that $u_{n}(t_{n},\cdot+x_{n})$
converges strongly in $H^{\frac{1}{2}}$, say $u(0)$. Let $u$ be
a global $H^{\frac{1}{2}}$ solution with initial data $u(0)$. We
show that $u$ does not scatter forward in time. Suppose that $S_{\geq0}(u)<\infty$.
Then, we have $\lim_{n\to\infty}S_{\geq t_{n}}(u_{n})=S_{\geq0}(u)<\infty$
by Corollary \ref{cor:closeness of two solutions on unbounded interval},
which makes a contradiction. Similarly, $u$ does not scatter backward
in time.

We now show that $u$ is almost periodic modulo spatial translations.
As $S_{\geq0}(u)=S_{\leq0}(u)=\infty$, we have $S_{\geq t}(u)=S_{\leq t}(u)=\infty$
for all $t\in\mathbb{R}$. Applying Proposition \ref{prop:Palais-Smale condition}
to $\{u(t)\}_{t\in\mathbb{R}}$, the conclusion follows.
\end{proof}
Let us conclude this subsection by noting a quantitative formulation
of almost periodic solutions. Using the Sobolev embedding $H^{\frac{1}{2}}\hookrightarrow L^{p}$
for any $2\leq p<\infty$ and the Arzela-Ascoli theorem in $L^{p}$
space, we have the following.
\begin{prop}[Arzela-Ascoli for a.p. solutions]
\label{prop:arzela-ascoli for a.p. solutions}Suppose that $u\in C_{t}H_{x}^{\frac{1}{2}}$
is almost periodic modulo spatial translations. Then, there exist
spatial parameters $x(t)\in\mathbb{R}$ for each $t\in\mathbb{R}$
such that for any $p\in[2,\infty)$ and $\eta>0$, there exists a
modulus $R=R(p,\eta,u)>0$ such that 
\[
\sup_{t\in\mathbb{R}}\Big(\int_{|x-x(t)|>R}|u(t,x)|^{p}dx\Big)<\eta.
\]
\end{prop}
\begin{rem}
\label{rem:4.23}In other contexts of critical setting, $\{u(t)\}_{t\in\mathbb{R}}$
is almost periodic modulo symmetries other than spatial translations,
such as scalings and frequency modulations. Thus, one introduces a
frequency scale function $N(t)$. In terms of behavior of $N(t)$,
it is decomposed into several scenarios. But in our case, as we work
on inhomogeneous setting $C_{t}H_{x}^{\frac{1}{2}}$, we only have
the case $N(t)\equiv1$.
\end{rem}

\subsection{\label{subsec:proof of error estimation}Proof of Lemma \ref{lem:error estimation}}

In this subsection, we prove Lemma \ref{lem:error estimation}. We
first introduce several lemmas, which account for decoupling of nonlinear
profiles.
\begin{lem}
\label{lem:good term another bound}For any $s>0$, $\ell=1,2,\dots,k-2$,
and $\varphi_{1},\varphi_{2},\psi_{1},\psi_{2}\in\dot{Z}_{t}^{s_{k}}$,
we have
\begin{align*}
 & \Big\|2^{js}\|\sum_{r\gtrsim j}(Q_{\sim r}\varphi_{1})(Q_{\sim r}\varphi_{2})(Q_{\lesssim r}\psi_{1})^{\ell}(Q_{\lesssim r}\psi_{2})^{k-1-\ell}\|_{L_{x}^{1}L_{t}^{2}}\Big\|_{\ell_{j}^{2}}\\
 & \lesssim_{s}\|\varphi_{2}\|_{L_{x}^{k}L_{t}^{\infty}}\|\psi_{1}\|_{L_{x}^{k}L_{t}^{\infty}}^{\ell-1}\|\psi_{2}\|_{L_{x}^{k}L_{t}^{\infty}}^{k-1-\ell}\Big\|2^{js}\|(Q_{\sim j}\varphi_{1})(Q_{\lesssim j}\psi_{1})\|_{L_{x}^{k}L_{t}^{2}}\Big\|_{\ell_{j}^{2}}.
\end{align*}
\end{lem}
\begin{proof}
Observe that
\begin{align*}
 & \Big\|2^{js}\|\sum_{r\gtrsim j}(Q_{\sim r}\varphi_{1})(Q_{\sim r}\varphi_{2})(Q_{\lesssim r}\psi_{1})^{\ell}(Q_{\lesssim r}\psi_{2})^{k-1-\ell}\|_{L_{x}^{1}L_{t}^{2}}\Big\|_{\ell_{j}^{2}}\\
 & \leq\Big\|\sum_{r\gtrsim j}2^{(j-r)s}\|(2^{rs}Q_{\sim r}\varphi_{1})(Q_{\sim r}\varphi_{2})(Q_{\lesssim r}\psi_{1})^{\ell}(Q_{\lesssim r}\psi_{2})^{k-1-\ell}\|_{L_{x}^{1}L_{t}^{2}}\Big\|_{\ell_{j}^{2}}\\
 & \lesssim\|\varphi_{2}\|_{L_{x}^{k}L_{t}^{\infty}}\|\psi_{1}\|_{L_{x}^{k}L_{t}^{\infty}}^{\ell-1}\|\psi_{2}\|_{L_{x}^{k}L_{t}^{\infty}}^{k-1-\ell}\Big\|\sum_{r\gtrsim j}2^{(j-r)s}\|(2^{rs}Q_{\sim r}\varphi_{1})(Q_{\lesssim r}\psi_{1})\|_{L_{x}^{1}L_{t}^{2}}\Big\|_{\ell_{j}^{2}}.
\end{align*}
We then use Young's inequality to obtain the conclusion.
\end{proof}
\begin{lem}
\label{lem:vanishing of weak limit for nonlinear solution}For any
$j\in\mathbb{N}$ and $t\in\mathbb{R}$, we have
\[
v^{j}(t+t_{n},\cdot+x_{n})\rightharpoonup0\quad\text{weakly in }H^{\frac{1}{2}},
\]
whenever $|t_{n}|+|x_{n}|\to\infty$.
\end{lem}
\begin{proof}
It suffices to show that $\langle\phi,v^{j}(t+t_{n},\cdot+x_{n})\rangle_{H^{\frac{1}{2}}}\to0$
for any $\phi\in C_{c}^{\infty}(\mathbb{R})$ with $\|\phi\|_{H^{\frac{1}{2}}}=1$.
We only consider the following two cases: 1. $|t_{n}|\to\infty$ and
2. $t_{n}\to t_{0}\in\mathbb{R}$ and $|x_{n}|\to\infty$. The remaining
part of the proof is an easy exercise.

Let us consider the first case $|t_{n}|\to\infty$. Because $\|v^{j}\|_{\dot{Z}_{t}^{s_{k}}}<\infty$,
$v^{j}$ scatters in $H^{\frac{1}{2}}$ both forward and backward
in time. Hence, we can approximate $v^{j}(t+t_{n},\cdot+x_{n})$ by
$V(t+t_{n})\psi(\cdot+x_{n})$ for some $\psi\in H^{\frac{1}{2}}$
for all large $n$. We then approximate $\psi$ by some $\psi\in C_{c}^{\infty}$
in $H^{\frac{1}{2}}$ sense. If we consider
\begin{align*}
 & |\langle\phi,v^{j}(t+t_{n},\cdot+x_{n})\rangle_{H^{\frac{1}{2}}}|\\
 & \leq\|v^{j}(t+t_{n})-V(t+t_{n})\psi\|_{H^{\frac{1}{2}}}+\|\psi-\tilde{\psi}\|_{H^{\frac{1}{2}}}+\|\langle\nabla\rangle\phi\|_{L^{1}}\|V(t+t_{n})\tilde{\psi}\|_{L^{\infty}}
\end{align*}
and use the dispersive decay $\|V(t+t_{n})\tilde{\psi}(\cdot+x_{n})\|_{L^{\infty}}\lesssim\langle t+t_{n}\rangle^{-\frac{1}{2}}\|\tilde{\psi}\|_{L^{1}}$,
we get the conclusion.

We now consider the remaining case: $t_{n}\to t_{0}\in\mathbb{R}$
and $|x_{n}|\to\infty$. Because $v^{j}\in C_{t}H_{x}^{\frac{1}{2}}$,
we can approximate $v^{j}(t+t_{n})$ by $v^{j}(t+t_{0})$. We then
approximate $v^{j}(t+t_{0})$ by some $\psi\in C_{c}^{\infty}(\mathbb{R})$
in $H^{\frac{1}{2}}$ sense. More precisely, we consider
\begin{align*}
 & |\langle\phi,v^{j}(t+t_{n},\cdot+x_{n})\rangle_{H^{\frac{1}{2}}}|\\
 & \leq\|v^{j}(t+t_{n})-v^{j}(t+t_{0})\|_{H^{\frac{1}{2}}}+\|v^{j}(t+t_{0})-\psi\|_{H^{\frac{1}{2}}}+|\langle\phi,\psi(\cdot+x_{n})\rangle_{H^{\frac{1}{2}}}|.
\end{align*}
We then apply Riemann-Lebesgue lemma to obtain $|\langle\phi,\psi(\cdot+x_{n})\rangle_{H^{\frac{1}{2}}}|\to0$.
This completes the proof.
\end{proof}
\begin{lem}[Asymptotic decoupling of $L_{x}^{k}L_{t}^{2}$ norm, I]
\label{lem:(k,2) decoupling 1}For each $m\in\mathbb{Z}$ and $j\neq j'$,
we have
\[
\lim_{n\to\infty}\|(Q_{\sim m}v_{n}^{j})(Q_{\lesssim m}v_{n}^{j'})\|_{L_{x}^{k}L_{t}^{2}}=0.
\]
\end{lem}
\begin{proof}
We may assume that $m=0$. Using change of variables and the property
\eqref{eq:continuity assumption 2 for (k,infty)}, it suffices to
show that
\[
\lim_{n\to\infty}\|(Q_{\lesssim0}v^{j'})(Q_{\sim0}v^{j}(\cdot+t_{n})(\cdot+x_{n}))\|_{L_{x}^{k}L_{T}^{2}}=0
\]
whenever $T<\infty$ and $|t_{n}|+|x_{n}|\to\infty$. Here, the variables
for $\cdot+t_{n}$ and $\cdot+x_{n}$ are $t$ and $x$, respectively.
We shall show the assertion by DCT. Observe for any $t$ and $x$
that we have a pointwise bound
\[
\sup_{n\in\mathbb{N}}|Q_{\sim0}v^{j}(t+t_{n})(x+x_{n})|\lesssim\|v^{j}\|_{L_{t}^{\infty}L_{x}^{2}}<\infty
\]
and hence 
\[
\Big\||Q_{\lesssim0}v^{j'}|\|v^{j}\|_{L_{t}^{\infty}L_{x}^{2}}\Big\|_{L_{x}^{k}L_{T}^{2}}\lesssim\sqrt{T}\|v^{j}\|_{L_{t}^{\infty}L_{x}^{2}}\|v^{j'}\|_{L_{x}^{k}L_{t}^{\infty}}<\infty.
\]
This allows us to use DCT; we are now reduced to show that
\[
Q_{\sim0}v^{j}(t+t_{n})(x+x_{n})\to0
\]
for each fixed $t$ and $x$. Because $Q_{\sim0}$ has a Schwartz
convolution kernel, it suffices to show that 
\[
v^{j}(t+t_{n},\cdot+x_{n})\rightharpoonup0
\]
weakly in $H^{\frac{1}{2}}$ for each $t$. This follows from Lemma
\ref{lem:vanishing of weak limit for nonlinear solution}.
\end{proof}
\begin{lem}[Asymptotic decoupling of $L_{x}^{k}L_{t}^{2}$ norm, II]
\label{lem:(k,2) decoupling 2}For each $j\in\mathbb{Z}$ and $\ell'\leq J$,
we have
\[
\lim_{n\to\infty}\|(Q_{\sim j}V(t)w_{n}^{J})(Q_{\lesssim j}v_{n}^{\ell'})\|_{L_{x}^{k}L_{t}^{2}}=0.
\]
\end{lem}
\begin{proof}
We proceed similarly as in the proof of Lemma \ref{lem:(k,2) decoupling 1}.
We may assume that $j=0$. Using change of variables and $\lim_{T\to\infty}\|v^{\ell'}\|_{L_{x}^{k}L_{T+}^{\infty}}=0$,
it suffices to show that
\[
\lim_{n\to\infty}\|(Q_{\lesssim0}v^{\ell'})(Q_{\sim0}V(\cdot-t_{n}^{\ell'})w_{n}^{J}(\cdot+x_{n}^{\ell'}))\|_{L_{x}^{k}L_{T}^{2}}=0
\]
whenever $T<\infty$. We shall show this assertion by DCT. Observe
for any $t$ and $x$ that we have a pointwise bound
\[
\sup_{n\in\mathbb{N}}|Q_{\sim0}V(t-t_{n}^{\ell'})w_{n}^{J}(x+x_{n}^{\ell'})|\lesssim\sup_{n\in\mathbb{N}}\|w_{n}^{J}\|_{L_{x}^{2}}<\infty
\]
and hence 
\[
\Big\||Q_{\lesssim0}v^{\ell'}|\sup_{n\in\mathbb{N}}\|w_{n}^{J}\|_{L_{x}^{2}}\Big\|_{L_{x}^{k}L_{T}^{2}}\lesssim\sqrt{T}\Big(\sup_{n\in\mathbb{N}}\|w_{n}^{J}\|_{L_{x}^{2}}\Big)\|v^{j}\|_{L_{x}^{k}L_{t}^{\infty}}<\infty.
\]
This allows us to use DCT; we are now reduced to show that
\[
[Q_{\sim0}V(t-t_{n}^{\ell'})w_{n}^{J}](x+x_{n}^{\ell'})\to0
\]
for each fixed $t$ and $x$. Because $Q_{\sim0}$ has a Schwartz
convolution kernel and $t$ is fixed, it suffices to show that 
\[
[V(-t_{n}^{\ell'})w_{n}^{J}](\cdot+x_{n}^{\ell'})\rightharpoonup0
\]
weakly in $H^{\frac{1}{2}}$. This follows from Lemma \ref{lem:vanishing of weak limit for nonlinear solution}.
\end{proof}
\begin{proof}[Proof of Lemma \ref{lem:error estimation}]
We must show
\[
\lim_{J\to\infty}\limsup_{n\to\infty}\|F(\tilde{u}_{n}^{J})-\sum_{j=1}^{J}F(v_{n}^{j})\|_{\dot{\mathcal{N}}_{t}^{s_{k}}}=0.
\]
We decompose 
\begin{align*}
 & F(\tilde{u}_{n}^{J})-\sum_{j=1}^{J}F(v_{n}^{j})\\
 & =\Big[F(\tilde{u}_{n}^{J})-\big(F(u_{n}^{J})+F(V(t)w_{n}^{J})\big)\Big]+\Big[F(u_{n}^{J})-\sum_{j=1}^{J}F(v_{n}^{j})\Big]+F(V(t)w_{n}^{J})
\end{align*}
Observe first that $F(\tilde{u}_{n}^{J})-[F(u_{n}^{J})+F(V(t)w_{n}^{J})]$
is a linear combination of 
\[
\partial_{x}[(u_{n}^{J})^{\ell}(V(t)w_{n}^{J})^{k+1-\ell}]
\]
where $1\leq\ell\leq k$. We then observe that $F(u_{n}^{J})-\sum_{j=1}^{J}F(v_{n}^{j})$
is a linear combination of 
\[
\partial_{x}[(v_{n}^{j})^{\ell}(v_{n}^{j'})^{k+1-\ell}]
\]
where $1\leq\ell\leq k$ and $j\neq j'$ with $j,j'\leq J$. Therefore,
in order to show our lemma, it suffices to show the following: 
\begin{align}
\lim_{J\to\infty}\limsup_{n\to\infty}\|\partial_{x}[(u_{n}^{J})^{\ell}(V(t)w_{n}^{J})^{k+1-\ell}]\|_{\dot{\mathcal{N}}_{t}^{s_{k}}} & =0,\label{eq:error-1}\\
\lim_{n\to\infty}\|\partial_{x}[(v_{n}^{j})^{\ell}(v_{n}^{j'})^{k+1-\ell}]\|_{\dot{\mathcal{N}}_{t}^{s_{k}}} & =0,\label{eq:error-2}\\
\lim_{J\to\infty}\limsup_{n\to\infty}\|F(V(t)w_{n}^{J})\|_{\dot{\mathcal{N}}_{t}^{s_{k}}} & =0,\label{eq:error-3}
\end{align}
whenever $1\leq\ell\leq k$ (and $j\neq j'$ for \eqref{eq:error-2}).

For \eqref{eq:error-3}, combine \eqref{eq:asymp vanishing (k,infty)}
with the following estimate
\[
\eqref{eq:error-3}\lesssim\lim_{J\to\infty}\limsup_{n\to\infty}\|w_{n}^{J}\|_{\dot{H}^{s_{k}}}\|V(t)w_{n}^{J}\|_{L_{x}^{k}L_{t}^{\infty}}^{k}=0.
\]
It suffices to show the remaining estimates \eqref{eq:error-1} and
\eqref{eq:error-2}.

We now show \eqref{eq:error-1}. Fix $1\leq\ell\leq k$ and observe
using paraproduct decomposition that
\begin{align*}
 & \|\partial_{x}[(u_{n}^{J})^{\ell}(V(t)w_{n}^{J})^{k+1-\ell}]\|_{\dot{\mathcal{N}}_{t}^{s_{k}}}\\
 & \lesssim\eqref{eq:error-1-1}+\eqref{eq:error-1-2}+\eqref{eq:error-1-3}+\eqref{eq:error-1-4}+\eqref{eq:error-1-5}
\end{align*}
where 
\begin{gather}
\Big\|2^{j(s_{k}+\frac{1}{2})}\|(Q_{\sim j}V(t)w_{n}^{J})(Q_{\ll j}u_{n}^{J})^{\ell}(Q_{\ll j}V(t)w_{n}^{J})^{k-\ell}\|_{L_{x}^{1}L_{t}^{2}}\Big\|_{\ell_{j}^{2}}\label{eq:error-1-1}\\
\Big\|2^{j(s_{k}+\frac{1}{2})}\|(Q_{\sim j}u_{n}^{J})(Q_{\ll j}u_{n}^{J})^{\ell-1}(Q_{\ll j}V(t)w_{n}^{J})^{k+1-\ell}\|_{L_{x}^{1}L_{t}^{2}}\Big\|_{\ell_{j}^{2}}\label{eq:error-1-2}\\
\Big\|2^{j(s_{k}+\frac{1}{2})}\|\sum_{r\gtrsim j}(Q_{\sim r}V(t)w_{n}^{J})^{2}(Q_{\lesssim r}u_{n}^{J})^{\ell}(Q_{\lesssim r}V(t)w_{n}^{J})^{k-1-\ell}\|_{L_{x}^{1}L_{t}^{2}}\Big\|_{\ell_{j}^{2}}\label{eq:error-1-3}\\
\Big\|2^{j(s_{k}+\frac{1}{2})}\|\sum_{r\gtrsim j}(Q_{\sim r}V(t)w_{n}^{J})(Q_{\sim r}u_{n}^{J})(Q_{\lesssim r}u_{n}^{J})^{\ell-1}(Q_{\lesssim r}V(t)w_{n}^{J})^{k-\ell}\|_{L_{x}^{1}L_{t}^{2}}\Big\|_{\ell_{j}^{2}}\label{eq:error-1-4}\\
\Big\|2^{j(s_{k}+\frac{1}{2})}\|\sum_{r\gtrsim j}(Q_{\sim r}u_{n}^{J})^{2}(Q_{\lesssim r}u_{n}^{J})^{\ell-2}(Q_{\lesssim r}V(t)w_{n}^{J})^{k+1-\ell}\|_{L_{x}^{1}L_{t}^{2}}\Big\|_{\ell_{j}^{2}}.\label{eq:error-1-5}
\end{gather}
Of course, the terms \eqref{eq:error-1-3} and \eqref{eq:error-1-5}
should be ignored if $\ell=k$ and $\ell=1$, respectively. In what
follows, we separately estimate each term.

For the terms \eqref{eq:error-1-2}-\eqref{eq:error-1-5}, we again
use asymptotic vanishing of the remainder \eqref{eq:asymp vanishing (k,infty)}.
We apply Lemma \ref{lem:good term another bound} to obtain 
\begin{align*}
\eqref{eq:error-1-2} & \lesssim\|u_{n}^{J}\|_{\dot{X}_{t}^{s_{k}}}\|u_{n}^{J}\|_{L_{x}^{k}L_{t}^{\infty}}^{\ell-1}\|V(t)w_{n}^{J}\|_{L_{x}^{k}L_{t}^{\infty}}^{k+1-\ell},\\
\eqref{eq:error-1-3} & \lesssim\|w_{n}^{J}\|_{\dot{H}^{s_{k}}}\|u_{n}^{J}\|_{L_{x}^{k}L_{t}^{\infty}}^{\ell}\|V(t)w_{n}^{J}\|_{L_{x}^{k}L_{t}^{\infty}}^{k-\ell},\\
\eqref{eq:error-1-4} & \lesssim\|u_{n}^{J}\|_{\dot{X}_{t}^{s_{k}}}\|u_{n}^{J}\|_{L_{x}^{k}L_{t}^{\infty}}^{\ell-1}\|V(t)w_{n}^{J}\|_{L_{x}^{k}L_{t}^{\infty}}^{k+1-\ell},\\
\eqref{eq:error-1-5} & \lesssim\|u_{n}^{J}\|_{\dot{X}_{t}^{s_{k}}}\|u_{n}^{J}\|_{L_{x}^{k}L_{t}^{\infty}}^{\ell-1}\|V(t)w_{n}^{J}\|_{L_{x}^{k}L_{t}^{\infty}}^{k+1-\ell}.
\end{align*}
Recall that we assume $1\leq\ell\leq k-1$ for \eqref{eq:error-1-3}.
We then apply \eqref{eq:asymp vanishing (k,infty)} to conclude.

The term \eqref{eq:error-1-1} is trickier. When $\ell\neq k$, we
can proceed as before and use asymptotic vanishing of the remainder
\eqref{eq:asymp vanishing (k,infty)}. However, in case of $\ell=k$,
Lemma \ref{lem:good term another bound} only yields the bound
\[
\|u_{n}^{J}\|_{L_{x}^{k}L_{t}^{\infty}}^{k}\|V(t)w_{n}^{J}\|_{\dot{X}_{t}^{s_{k}}},
\]
where we do not have asymptotic vanishing of $\|V(t)w_{n}^{J}\|_{\dot{X}_{t}^{s_{k}}}$
(c.f. Remark \ref{rem:no asym vanish of X}). To resolve this difficulty,
by the linear estimate, we use
\begin{align*}
\eqref{eq:error-1-1} & \lesssim\|w_{n}^{J}\|_{\dot{H}^{s_{k}}}^{k-\ell}\|u_{n}^{J}\|_{L_{x}^{k}L_{t}^{\infty}}^{\ell-1}\Big\|2^{j(s_{k}+\frac{1}{2})}\|(Q_{\sim j}V(t)w_{n}^{J})(Q_{\ll j}u_{n}^{J})\|_{L_{x}^{k}L_{t}^{2}}\Big\|_{\ell_{j}^{2}}
\end{align*}
instead. Hence, it suffices to show for each $J\in\mathbb{N}$ that
\[
\lim_{n\to\infty}\Big\|2^{j(s_{k}+\frac{1}{2})}\|(Q_{\sim j}V(t)w_{n}^{J})(Q_{\ll j}u_{n}^{J})\|_{L_{x}^{k}L_{t}^{2}}\Big\|_{\ell_{j}^{2}}=0.
\]
By DCT in $j\in\mathbb{Z}$ and definition of $u_{n}^{J}$, it suffices
to show that 
\[
\lim_{n\to\infty}\|(Q_{\sim j}V(t)w_{n}^{J})(Q_{\ll j}v_{n}^{\ell'})\|_{L_{x}^{k}L_{t}^{2}}=0
\]
for each $j\in\mathbb{Z}$ and $1\leq\ell'\leq J$. This follows from
Lemma \ref{lem:(k,2) decoupling 2}. This completes the proof of \eqref{eq:error-1}.

We now show \eqref{eq:error-2}. Using paraproduct decomposition and
symmetric arguments in $j$ and $j'$, it suffices to show that 
\[
\lim_{n\to\infty}\eqref{eq:error-2-1}+\eqref{eq:error-2-2}+\eqref{eq:error-2-3}=0
\]
for all $\ell=1,2,\dots,k$, where 
\begin{align}
 & \Big\|2^{m(s_{k}+\frac{1}{2})}\|(Q_{\sim m}v_{n}^{j'})(Q_{\ll m}v_{n}^{j})^{\ell}(Q_{\ll m}v_{n}^{j'})^{k-\ell}\|_{L_{x}^{1}L_{t}^{2}}\Big\|_{\ell_{m}^{2}}\label{eq:error-2-1}\\
 & \Big\|2^{m(s_{k}+\frac{1}{2})}\|\sum_{r\gtrsim m}(Q_{\sim r}v_{n}^{j'})^{2}(Q_{\lesssim r}v_{n}^{j})^{\ell}(Q_{\lesssim r}v_{n}^{j'})^{k-1-\ell}\|_{L_{x}^{1}L_{t}^{2}}\Big\|_{\ell_{m}^{2}}\label{eq:error-2-2}\\
 & \Big\|2^{m(s_{k}+\frac{1}{2})}\|\sum_{r\gtrsim m}(Q_{\sim r}v_{n}^{j})(Q_{\sim r}v_{n}^{j'})(Q_{\lesssim r}v_{n}^{j})^{\ell-1}(Q_{\lesssim r}v_{n}^{j'})^{k-\ell}\|_{L_{x}^{1}L_{t}^{2}}\Big\|_{\ell_{m}^{2}}\label{eq:error-2-3}
\end{align}
Of course, one should ignore \eqref{eq:error-2-2} if $k=\ell$.

For \eqref{eq:error-2-1}, by H\"older's inequality and Lebesgue's
DCT in $m\in\mathbb{Z}$, it suffices to show that
\[
\lim_{n\to\infty}\|(Q_{\sim m}v_{n}^{j'})(Q_{\ll m}v_{n}^{j})\|_{L_{x}^{k}L_{t}^{2}}=0
\]
for each $m\in\mathbb{Z}$. This follows from Lemma \ref{lem:(k,2) decoupling 1}.

For \eqref{eq:error-2-2}, recall that $1\leq\ell\leq k-1$. By Lemma
\ref{lem:good term another bound}, it suffices to show that
\[
\lim_{n\to\infty}\Big\|2^{m(s_{k}+\frac{1}{2})}\|(Q_{\sim m}v_{n}^{j'})(Q_{\lesssim m}v_{n}^{j})\|_{L_{x}^{k}L_{t}^{2}}\Big\|_{\ell_{m}^{2}}=0.
\]
By Lebesgue's DCT in $m\in\mathbb{Z}$, it suffices to show that 
\[
\lim_{n\to\infty}\|(Q_{\sim m}v_{n}^{j})(Q_{\lesssim m}v_{n}^{j'})\|_{L_{x}^{k}L_{t}^{2}}=0
\]
for each $m\in\mathbb{Z}$. This follows from Lemma \ref{lem:(k,2) decoupling 1}.
For \eqref{eq:error-2-3}, one can argue similarly. This completes
the proof of \eqref{eq:error-2}.
\end{proof}

\section{\label{sec:precllusion of a.p. sol}Preclusion of Minimal Blowup
Solutions}

We now fix a critical element $u$ found in Theorem \ref{thm:existence of a.p. solutions}.
In this section, we assert that such $u$ cannot exist. This would
yield a contradiction, so we conclude Theorem \ref{thm:scattering}.

In order to preclude critical elements, we use the monotonicity formula
(Proposition \ref{prop:monotonicity formula}). Here, since $u\in C_{t}H_{x}^{\frac{1}{2}}$,
we are forced to truncate the weight $(x-y)$ in the formula. This
truncation was exploited by Dodson \cite{Dodson2017AnnPDE} in the
defocusing \eqref{eq:gKdV} and in many earlier works, for instance,
\cite{OgawaTsutsumi1991JDE}. As we work with the Hilbert transform,
derivative operators, and their commutators, we are involved in several
technical difficulties and tons of error terms.

In Section \ref{subsec:localized interaction functional}, we discuss
how we localize the monotonicity formula and use it to prove Theorem
\ref{thm:scattering}. Estimates for the error terms appeared in Section
\ref{subsec:localized interaction functional} will be postponed to
Section \ref{subsec:various error estimates}.

\subsection{\label{subsec:localized interaction functional}Localized Interaction
Functional}

Set large parameters $R\gg1$ and $R_{1}\ll R$ (e.g. $R_{1}\coloneqq R^{1-\epsilon}$
for some $0<\epsilon\ll1$). Define a smooth even cutoff function
$\chi_{R}$ on $\mathbb{R}$ satisfying $0\leq\chi\leq1$, $\chi(x)=0$
if $|x|\geq R+R_{1}$, and $\chi(x)=1$ if $|x|\leq R$. We can further
assume that $\|\partial_{x}^{j}\chi_{R}\|_{L^{\infty}}\lesssim_{j}R_{1}^{-j}$
and the support of $\partial_{x}^{j}\chi_{R}$ has measure $\leq2R_{1}$
whenever $j\geq1$. Define $\Phi(x)\coloneqq\int_{0}^{x}(\chi_{R}^{2}\ast\chi_{R}^{2})(s)\frac{ds}{R}$.
Then $\Phi$ is an odd function such that $\Phi(x)\approx x$ for
$|x|\lesssim R$ and $|\Phi(x)|\approx R$ for $|x|\gg R$. Moreover,
the convolution in $\Phi$ is well-suited for the monotonicity formula
of interaction form. Note that 
\[
\frac{1}{R}(\chi_{R}^{2}\ast\chi_{R}^{2})(y-x)=\int_{\mathbb{R}}\chi_{R}^{2}(y-s)\chi_{R}^{2}(x-s)\frac{ds}{R}.
\]
Similarly to the work of Dodson \cite{Dodson2017AnnPDE}, we define
the \emph{localized interaction functional} by
\begin{equation}
M(t)\coloneqq\int_{\mathbb{R}\times\mathbb{R}}\Phi(y-x)\rho(t,x)e(t,y)dxdy.\label{eq:localized functional}
\end{equation}

If $u$ lies merely in $C_{t}H_{x}^{\frac{1}{2}}$, the integral \eqref{eq:localized functional}
should not be interpreted in absolute sense as we now explain. As
in other defocusing equations, one may use a crude estimate
\[
|M(t)|\lesssim\int_{\mathbb{R}\times\mathbb{R}}R|\rho(t,x)||e(t,y)|dxdy\lesssim R\|u\|_{L_{t}^{\infty}L_{x}^{2}}\int_{\mathbb{R}}|e(t,y)|dy.
\]
If $e(t,y)$ were pointwise nonnegative, then $\int_{\mathbb{R}}|e(t,y)|dy=E(u)$
so $M(t)$ can be well-defined for $u\in C_{t}H_{x}^{\frac{1}{2}}$.
In our case, however, 
\[
e(t,y)=[\frac{1}{2}u\mathcal{H}u_{x}+\frac{1}{k+2}u^{k+2}](t,y)
\]
is not a pointwise positive function, c.f. we only have $E(u)=\int_{\mathbb{R}}e(t,y)dy>0.$
The quantity $\int_{\mathbb{R}}|e(t,y)|dy$ does not seem to be estimated
by $\|u\|_{L_{t}^{\infty}H^{\frac{1}{2}}}$. The following lemma says
how we can interpret \eqref{eq:localized functional} appropriately.
\begin{lem}[$H^{\frac{1}{2}}$ bound for localization]
\label{lem:H(1/2) bound for localization}Let $f\in H^{1}$ and $g\in C^{\infty}$
with $\|g\|_{W^{1,\infty}}\coloneqq\|g\|_{L^{\infty}}+\|\partial_{x}g\|_{L^{\infty}}\lesssim H$.
Then,
\[
\Big|\int gf\mathcal{H}f_{x}\Big|+\Big|\int gf^{k+2}\Big|\lesssim H(\|f\|_{H^{\frac{1}{2}}}^{2}+\|f\|_{H^{\frac{1}{2}}}^{k+2}).
\]
In particular, if $u\in C_{t}H_{x}^{\frac{1}{2}}$, then 
\begin{equation}
\sup_{t\in\mathbb{R}}|M(t)|\lesssim_{A(u)}R.\label{eq:upper bound monotonicity}
\end{equation}
\end{lem}
\begin{proof}
By H\"older's inequality and Sobolev embedding $H^{\frac{1}{2}}\hookrightarrow L^{k+2}$,
we have
\[
\Big|\int gf^{k+2}\Big|\lesssim H\|f\|_{H^{\frac{1}{2}}}^{k+2}.
\]
In order to estimate $\int gf\mathcal{H}f_{x}$, it suffices to establish
the inequality
\[
\|gf\|_{H^{\frac{1}{2}}}\lesssim H\|f\|_{H^{\frac{1}{2}}}
\]
because 
\[
\Big|\int gf\mathcal{H}f_{x}\Big|=\Big|\int D_{x}^{\frac{1}{2}}(gf)D_{x}^{\frac{1}{2}}f\Big|\lesssim\|gf\|_{H^{\frac{1}{2}}}\|f\|_{H^{\frac{1}{2}}}.
\]
Applying the Stein-Weiss interpolation between
\[
\|gf\|_{L^{2}}\lesssim\|g\|_{L^{\infty}}\|f\|_{L^{2}}\lesssim H\|f\|_{L^{2}}
\]
and 
\[
\|gf\|_{H^{1}}\lesssim(\|g\|_{L^{\infty}}+\|\partial_{x}g\|_{L^{\infty}})\|f\|_{L^{2}}+\|g\|_{L^{\infty}}\|\partial_{x}f\|_{L^{2}}\lesssim H\|f\|_{H^{1}},
\]
we get $\|gf\|_{H^{\frac{1}{2}}}\lesssim H\|f\|_{H^{\frac{1}{2}}}$.

To show the remaining assertion, pretend $u(t)\in H^{1}$ for each
$t\in\mathbb{R}$. For any fixed $x$, note that the function $y\mapsto\Phi(y-x)$
has $W^{1,\infty}$ norm $\lesssim R$. We then see that
\[
\sup_{t\in\mathbb{R}}|M(t)|\leq\sup_{t\in\mathbb{R}}\|\rho(t,x)\|_{L_{x}^{1}}\|\int_{y}\Phi(y-x)e(t,y)\|_{L_{x}^{\infty}}\lesssim_{A(u)}R.
\]
The argument allows us to use density argument for $u(t)\in H^{\frac{1}{2}}$,
so the assertion is true for $u\in C_{t}H_{x}^{\frac{1}{2}}$.
\end{proof}
In the following, a number of error terms appears and computations
are involved, so we abbreviate the notations to present as neatly
as possible. We write $\chi_{R,s}\coloneqq\chi_{R}(\cdot-s)$. For
operators $A$ and $B$, we let $[A,B]=AB-BA$ be the commutator.
We then abbreviate various integrals as follows.
\begin{gather*}
\int_{x}\coloneqq\int_{\mathbb{R}}dx,\quad\int_{y}\coloneqq\int_{\mathbb{R}}dy,\quad\int_{z}\coloneqq\int_{\mathbb{R}}dz,\\
\int_{s}\coloneqq\int_{\mathbb{R}}\frac{ds}{R},\quad\int_{r}\coloneqq\int_{0}^{1}dr.
\end{gather*}
Notice that we integrate over $\frac{ds}{R}$ in $s$-variable, and
only on $[0,1]$ for $r$-variable. Finally, we use $o_{A(u),R}(1)$
as $R\to\infty$ for a function decaying to zero. If we use $R_{1}=R^{1-\epsilon}$,
then one can observe $o_{A(u),R}(1)\lesssim C_{A(u)}R^{-\epsilon'}$
for some $\epsilon'>0$.

As a consequence of the monotonicity formula, we state the main proposition
of this section.
\begin{prop}
\label{prop:localized interaction functional}Let $u\in C_{t}H_{x}^{\frac{1}{2}}$
be an almost periodic global solution to \eqref{eq:gBO}. If $R$
is sufficiently large, then we have a lower bound
\begin{equation}
|M(t_{1})-M(t_{2})|\gtrsim_{u}|t_{1}-t_{2}|\label{eq:lower bound monotonicity}
\end{equation}
for any $t_{1},t_{2}\in\mathbb{R}$.

From a contradiction of \eqref{eq:upper bound monotonicity} and \eqref{eq:lower bound monotonicity}
for the critical element constructed in Theorem \ref{thm:existence of a.p. solutions},
we deduce that Theorem \ref{thm:scattering} holds.
\end{prop}
\begin{proof}
We rely on density argument and assume $u(t,x)$ lies in $C_{t,loc}^{\infty}H_{x}^{\infty}$.
We assume that $u$ is almost periodic modulo spatial translations,
but we are not sure that $u$ can be approximated by smooth almost
periodic solutions $u_{n}$. However, we shall use density argument
to \eqref{eq:final result}, which is true for arbitrary smooth solutions,
and hence, for $H^{\frac{1}{2}}$ solutions. We then apply almost
periodicity of $u$ to deduce \eqref{eq:lower bound monotonicity}.
From now on, we assume that $u$ is a global smooth solution.

In light of fundamental theorem of calculus, we shall calculate $\partial_{t}M(t)$.
When we calculate $\partial_{t}M(t)$, as $t$ is fixed, we drop the
variable $t$ for convenience. We start with
\[
\partial_{t}M(t)=\eqref{eq:deriv on energy}+\eqref{eq:deriv on mass}.
\]
where 
\begin{align}
 & \int_{x,y}\Phi(y-x)\rho(x)\partial_{t}e(y),\label{eq:deriv on energy}\\
 & \int_{x,y}\Phi(y-x)e(y)\partial_{t}\rho(x).\label{eq:deriv on mass}
\end{align}

We first estimate \eqref{eq:deriv on energy}. Note that 
\[
\eqref{eq:deriv on energy}=\int_{x,y}\Phi(y-x)\rho(x)[\frac{1}{2}u_{t}\mathcal{H}\partial_{y}u+\frac{1}{2}u\mathcal{H}\partial_{y}u_{t}+u^{k+1}u_{t}](y).
\]
We rearrange the terms by their degree, namely, $2$, $k+2$, and
$2k+2$ in $u(y)$. More precisely, we decompose 
\[
\eqref{eq:deriv on energy}=\eqref{eq:energy derivative order 2}+\eqref{eq:energy derivative order k+2}+\eqref{eq:energy derivative order 2k+2}
\]
where 
\begin{gather}
\frac{1}{2}\int_{x,y}\Phi(y-x)\rho(x)[-(\mathcal{H}\partial_{yy}u)(\mathcal{H}\partial_{y}u)+u\partial_{yyy}u](y),\label{eq:energy derivative order 2}\\
\int_{x,y}\Phi(y-x)\rho(x)[-\frac{1}{2}\partial_{y}(u^{k+1})(\mathcal{H}\partial_{y}u)-\frac{1}{2}u\mathcal{H}\partial_{yy}(u^{k+1})-u^{k+1}\mathcal{H}\partial_{yy}u](y),\label{eq:energy derivative order k+2}\\
\int_{x,y}\Phi(y-x)\rho(x)[-u^{k+1}\partial_{y}(u^{k+1})](y).\label{eq:energy derivative order 2k+2}
\end{gather}
For the term, \eqref{eq:energy derivative order 2}, we use integration
by parts in $y$-variable to have
\begin{align*}
\eqref{eq:energy derivative order 2} & =\int_{s}\int_{x}\rho[\chi_{R,s}u](x)\int_{y}[\frac{1}{4}(\chi_{R,s}\mathcal{H}u_{y})^{2}+\frac{3}{4}(\chi_{R,s}u_{y})^{2}-\frac{1}{2}u^{2}\partial_{yy}(\chi_{R,s}^{2})](y)\\
 & =\int_{s}\int_{x}\rho[\chi_{R,s}u](x)\int_{y}[\frac{1}{4}(\chi_{R,s}\mathcal{H}u_{y})^{2}+\frac{3}{4}(\chi_{R,s}u_{y})^{2}](y)-\frac{1}{2}\int_{x,y}\rho(x)E_{1}(y),
\end{align*}
where the error term $E_{1}$ is defined by
\begin{equation}
E_{1}(y)\coloneqq\Phi'''(y-x)u^{2}(y).\label{eq:energy derivative error 1}
\end{equation}
Similarly, we observe that 
\begin{align*}
\eqref{eq:energy derivative order 2k+2} & =\frac{1}{2}\int_{s}\int_{x}\rho[\chi_{R,s}u](x)\int_{y}\chi_{R,s}^{2}u^{2k+2}(y)\\
 & \geq\int_{s}\int_{x}\rho[\chi_{R,s}u](x)\int_{y}\frac{1}{2}(\chi_{R,s}u)^{2k+2}(y).
\end{align*}
We now consider \eqref{eq:energy derivative order k+2}. Let us temporarily
write $\Phi(y-x)$ and $u^{k+1}(y)$ by $Q(y)$ and $v(y)$, respectively.
Integration by parts in $y$-variable and self-adjointness of $\mathcal{H}\partial_{y}$
yield
\begin{align*}
\eqref{eq:energy derivative order k+2} & =\int_{x,y}Q(y)\rho(x)[-\frac{1}{2}v_{y}\mathcal{H}\partial_{y}u-\frac{1}{2}u\mathcal{H}\partial_{y}v_{y}-v\mathcal{H}\partial_{yy}u](y)\\
 & =\int_{x,y}Q'(y)\rho(x)[v\mathcal{H}\partial_{y}u](y)+\frac{1}{2}\int_{x,y}Q(y)\rho(x)[v_{y}\mathcal{H}\partial_{y}u-u\mathcal{H}\partial_{y}v_{y}](y)\\
 & =\int_{x,y}Q'(y)\rho(x)[v\mathcal{H}\partial_{y}u](y)+\frac{1}{2}\int_{x,y}\rho(x)\big(u[\mathcal{H}\partial_{y},Q]v_{y}\big)(y).
\end{align*}
We then use 
\begin{align*}
[\mathcal{H}\partial_{y},Q]v_{y} & =\mathcal{H}(Q'v_{y})+[\mathcal{H},Q]\partial_{yy}v\\
 & =\mathcal{H}\partial_{y}(Q'v)-\mathcal{H}(Q''v)+[\mathcal{H},Q]\partial_{yy}v
\end{align*}
and self-adjointness of $\mathcal{H}\partial_{y}$ again to obtain
\[
\eqref{eq:energy derivative order k+2}=\frac{3}{2}\int_{x,y}\Phi'(y-x)\rho(x)[v\mathcal{H}\partial_{y}u](y)+\frac{1}{2}\int_{x,y}\rho(x)[E_{2}+E_{3}](y)
\]
where error terms $E_{2}$ and $E_{3}$ are defined by
\begin{align}
E_{2} & \coloneqq u[\mathcal{H},\Phi(\cdot-x)]v_{yy},\label{eq:energy derivative error 2}\\
E_{3} & \coloneqq\Phi''(\cdot-x)u^{k+1}\mathcal{H}u.\label{eq:energy derivative error 3}
\end{align}
Therefore, collecting altogether,
\begin{align}
 & \int_{x,y}\Phi(y-x)\rho(x)\partial_{t}e(y)\nonumber \\
 & =\int_{s}\int_{x}\rho[\chi_{R,s}u](x)\int_{y}\chi_{R,s}^{2}(y)[\frac{1}{4}(\mathcal{H}u_{y})^{2}+\frac{3}{4}u_{y}^{2}+\frac{3}{2}u^{k+1}\mathcal{H}u_{y}+\frac{1}{2}u^{2k+2}](y)\label{eq:unrefined main term for energy derivative}\\
 & \quad+\frac{1}{2}\int_{x,y}\rho(x)[-E_{1}+E_{2}+E_{3}](y)\label{eq:energy derivative error}
\end{align}
where the error terms $E_{1},E_{2},E_{3}$ are defined in \eqref{eq:energy derivative error 1},
\eqref{eq:energy derivative error 2}, and \eqref{eq:energy derivative error 3},
respectively. In the next subsection (Lemmas \ref{lem:estimate of main term of energy derivative}
and \ref{lem:estimate of energy derivative error}), we prove that
\begin{align*}
\eqref{eq:unrefined main term for energy derivative} & \geq\int_{s}\int_{x}\rho[\chi_{R,s}u](x)\int_{y}k[\chi_{R,s}u](y)-o_{A(u),R}(1),\\
\eqref{eq:energy derivative error} & =o_{A(u),R}(1).
\end{align*}
This takes care of \eqref{eq:deriv on energy}.

We now compute \eqref{eq:deriv on mass}.
\begin{align}
\eqref{eq:deriv on mass} & =2\int_{x,y}\Phi(y-x)e(y)u(x)[-\mathcal{H}\partial_{xx}u-\partial_{x}(u^{k+1})](x)\nonumber \\
 & =-\int_{s}\int_{x}[\chi_{R,s}^{2}j](x)\int_{y}[\chi_{R,s}^{2}e](y)\label{eq:unrefined main term for mass derivative}\\
 & \quad+2\int_{x,y}\Phi(y-x)[u_{x}\mathcal{H}u_{x}](x)e(y).\label{eq:mass derivative error}
\end{align}
We regard \eqref{eq:unrefined main term for mass derivative} as the
main term and \eqref{eq:mass derivative error} as an error term.
In the next subsection (Lemmas \ref{lem:estimate of main term of mass derivative}
and \ref{lem:estimate of mass derivative error}), we prove that 
\begin{align*}
\eqref{eq:unrefined main term for mass derivative} & =-\int_{s}\int_{x}j[\chi_{R,s}u](x)\int_{y}e[\chi_{R,s}u](y)-o_{A(u),R}(1).\\
\eqref{eq:mass derivative error} & =o_{A(u),R}(1).
\end{align*}

In conclusion, we have 
\begin{align*}
\partial_{t}M(t) & =\eqref{eq:deriv on mass}+\eqref{eq:deriv on energy}\\
 & \geq\int_{s}\int_{x}\rho[\chi_{R,s}u](x)\int_{y}k[\chi_{R,s}u](y)-\int_{x}j[\chi_{R,s}u](x)\int_{y}e[\chi_{R,s}u](y)-o_{A(u),R}(1).
\end{align*}
Combining this with the monotonicity formula (Proposition \ref{prop:monotonicity formula}),
we obtain 
\[
\partial_{t}M(t)\geq\frac{k^{2}}{2(k+2)^{2}}\int_{s}\Big(\int_{x}(\chi_{R,s}u)^{2}(x)\Big)^{2}\Big(\int_{y}(\chi_{R,s}u)^{k+2}(y)\Big)^{2}-o_{A(u),R}(1).
\]
By the fundamental theorem of calculus, we finally obtain 
\begin{align}
 & |M(t_{1})-M(t_{2})|\nonumber \\
 & \gtrsim\int_{t_{2}}^{t_{1}}\int_{s}\Big(\int_{x}(\chi_{R,s}u)^{2}(t,x)\Big)^{2}\Big(\int_{y}(\chi_{R,s}u)^{k+2}(t,y)\Big)^{2}dt-|t_{1}-t_{2}|o_{A(u),R}(1)\label{eq:final result}
\end{align}
whenever $t_{1}>t_{2}$. By density argument, \eqref{eq:final result}
indeed holds whenever $u$ is an arbitrary $H^{\frac{1}{2}}$ solutions.

We now assume that $u$ is almost periodic modulo spatial translations.
In light of Proposition \ref{prop:arzela-ascoli for a.p. solutions},
there exist spatial center $\{x(t)\}_{t\in\mathbb{R}}$ and $R_{0}$
depending on $u$ such that whenever $R>R_{0}$ and $|s-x(t)|\leq\frac{R}{2}$,
we have
\[
\int_{x}(\chi_{R,s}u)^{2}(t,x)\geq\int_{|x-x(t)|\leq\frac{R}{2}}u^{2}(t,x)dx\geq\frac{M(u)}{2}
\]
and 
\[
\int_{y}(\chi_{R,s}u)^{k+2}(t,y)\geq\int_{|y-x(t)|\leq\frac{R}{2}}u^{k+2}(t,y)dy\gtrsim_{u}1.
\]
Therefore, 
\begin{align*}
\eqref{eq:final result} & \gtrsim_{u}|t_{1}-t_{2}|-|t_{1}-t_{2}|o_{A(u),R}(1)
\end{align*}
for all $R>R_{0}$. Choosing $R$ even larger, we finally have
\[
|M(t_{1})-M(t_{2})|\gtrsim_{u}|t_{1}-t_{2}|.
\]
This completes the proof.
\end{proof}

\subsection{\label{subsec:various error estimates}Various Error Estimates in
proof of Proposition \ref{prop:localized interaction functional}}

In this subsection, we estimate various error terms appeared in the
proof of Proposition \ref{prop:localized interaction functional}.
We have seen that many errors terms involve commutators of the Hilbert
transform and a smooth function, say $[\mathcal{H},Q]$ for some $Q\in C^{\infty}$.
Using the formula of the Hilbert transform, we obtain
\[
([\mathcal{H},Q]u)(y)=\frac{1}{\pi}\int_{z}\frac{Q(z)-Q(y)}{y-z}u(z)=-\frac{1}{\pi}\int_{z}\int_{r}Q'(rz+(1-r)y)u(z).
\]
Thus, by Fubini's theorem,
\[
\int_{y}v[\mathcal{H},Q]u=-\frac{1}{\pi}\int_{r}\int_{y,z}Q'(rz+(1-r)y)v(y)u(z).
\]
If $u$ or $v$ contain derivatives, we can use integration by parts
to move derivatives to $Q$. For instance,
\[
\int_{y}v_{y}[\mathcal{H},Q]u_{y}=-\frac{1}{\pi}\int_{r}\int_{y,z}r(1-r)Q'''(rz+(1-r)y)v(y)u(z).
\]
In practice, we will use $\Phi$, $\chi_{R,s}$, and $\chi_{R,s}'$
in place of $Q$. Derivatives of $Q$ decay in $R$ as $R$ grows.
This says that error terms involving commutator should be well-estimated
as $R$ grows. Our main tool is Schur's test, we state the exact formulation
we use here, for reader's convenience.
\begin{lem}[Schur's test]
\label{lem:Schur lemma}Let $K(y,z)$ be a kernel satisfying the
following three bounds:\\
1. (height) $\|K\|_{L^{\infty}}\lesssim H$,\\
2. (size of $y$-support) $\sup_{z}|\{y:K(y,z)\neq0\}|\lesssim R_{1}$,
and\\
3. (size of $z$-support) $\sup_{y}|\{z:K(y,z)\neq0\}|\lesssim R_{2}$.\\
Then, 
\[
\Big|\int_{y,z}K(y,z)u(y)v(z)\Big|\lesssim H\sqrt{R_{1}R_{2}}\|u\|_{L^{2}}\|v\|_{L^{2}}.
\]
\end{lem}
\begin{proof}
The proof is fairly standard.
\end{proof}
\begin{lem}[Estimate of \eqref{eq:energy derivative error}]
\label{lem:estimate of energy derivative error}We have
\[
\eqref{eq:energy derivative error}=o_{A(u),R}(1).
\]
\end{lem}
\begin{proof}
From H\"older's inequality 
\[
\Big|\int_{x,y}\rho(x)E_{i}(x,y)\Big|\lesssim\|\rho(x)\|_{L_{x}^{1}}\|\int_{y}E_{i}(x,y)\|_{L_{x}^{\infty}},
\]
it suffices to show 
\begin{align*}
\|\int_{y}E_{i}(x,y)\|_{L_{x}^{\infty}} & =o_{A(u),R}(1)
\end{align*}
for each $i=1,2,3$. When $i=1,3$, this follows easily by H\"older's
inequality in $y$-variable. When $i=2$, observe that 
\[
\int_{y}E_{2}(x,y)=-\frac{1}{\pi}\int_{r}\int_{y,z}r^{2}\Phi'''(rz+(1-r)y-x)u(y)u^{k+1}(z).
\]
Whenever $r$ is fixed, the kernel $(y,z)\mapsto\Phi'''(rz+(1-r)y-x)$
has height $\lesssim\frac{1}{R_{1}^{2}}$, $y$-support size $\lesssim\frac{R}{1-r}$,
and $z$-support size $\lesssim\frac{R}{r}$. Thus by Minkowski's
inequality in $r$ and Schur's test, we obtain
\[
\Big|\int_{y}E_{2}(x,y)\Big|\lesssim_{A(u)}\int_{r}r^{\frac{3}{2}}(1-r)^{-\frac{1}{2}}\frac{R}{R_{1}^{2}}=o_{A(u),R}(1).
\]
This completes the proof.
\end{proof}
\begin{lem}[Estimate of \eqref{eq:unrefined main term for energy derivative}]
\label{lem:estimate of main term of energy derivative}We have
\[
\eqref{eq:unrefined main term for energy derivative}\geq\int_{s}\int_{x}\rho[\chi_{R,s}u](x)\int_{y}k[\chi_{R,s}u](y)-o_{A(u),R}(1).
\]
\end{lem}
\begin{proof}
Recall the expression \eqref{eq:unrefined main term for energy derivative}
\[
\int_{s}\int_{x}\rho[\chi_{R,s}u](x)\int_{y}\chi_{R,s}^{2}(y)[\frac{1}{4}(\mathcal{H}u_{y})^{2}+\frac{3}{4}u_{y}^{2}+\frac{3}{2}u^{k+1}\mathcal{H}u_{y}+\frac{1}{2}u^{2k+2}](y).
\]
We first approximate $\int_{y}(\chi_{R,s}\mathcal{H}u_{y})^{2}$ and
$\int_{y}(\chi_{R,s}u_{y})^{2}$ by $\int_{y}(\chi_{R,s}u)_{y}^{2}$.
Invoking Lemma \ref{lem:positivity of s}, $\int_{y}\chi_{R,s}^{2}u^{k+1}\mathcal{H}u_{y}$
is essentially greater than $\int_{y}\chi_{R,s}^{k+2}u^{k+1}\mathcal{H}u_{y}$
with error $\lesssim_{A(u)}\|(\chi_{R,s}^{2})''\|_{L^{\infty}}$.
We then approximate $\int_{y}\chi_{R,s}^{k+2}u^{k+1}\mathcal{H}u_{y}$
by $\int_{y}(\chi_{R,s}u)^{k+1}\mathcal{H}(\chi_{R,s}u)_{y}$. Note
that $\int_{y}\chi_{R,s}^{2}u^{2k+2}$ is always greater than or equal
to $\int_{y}(\chi_{R,s}u)^{2k+2}$.

Moreover, from H\"older's inequality and Fubini's theorem, we have
\begin{align*}
\Big|\int_{s}\int_{x}\rho[\chi_{R,s}u](x)\int_{y}g(y,s)\Big| & \leq\|\int_{x}\rho[\chi_{R,s}u](x)\|_{L_{s}^{1}(\frac{ds}{R})}\|\int_{y}g(y,s)\|_{L_{s}^{\infty}(\frac{ds}{R})}\\
 & \lesssim\|u\|_{L_{t}^{\infty}L_{x}^{2}}\|\int_{y}g(y,s)\|_{L_{s}^{\infty}(\frac{ds}{R})}.
\end{align*}
Therefore, to show our lemma, it suffices to show for each $s$ that
\begin{align}
\int_{y}(\chi_{R,s}\mathcal{H}u_{y})^{2} & =\int_{y}(\chi_{R,s}u)_{y}^{2}+o_{A(u),R}(1)\label{eq:refining k error 1}\\
\int_{y}(\chi_{R,s}u_{y})^{2} & =\int_{y}(\chi_{R,s}u)_{y}^{2}+o_{A(u),R}(1)\label{eq:refining k error 2}\\
\int_{y}\chi_{R,s}^{k+2}u^{k+1}\mathcal{H}u_{y} & =\int_{y}(\chi_{R,s}u)^{k+1}\mathcal{H}(\chi_{R,s}u)_{y}+o_{A(u),R}(1),\label{eq:refining k error 3}
\end{align}
 and 
\begin{equation}
\sup_{s\in\mathbb{R}}\|(\chi_{R,s}^{2})''\|_{L^{\infty}}\|u\|_{L_{t}^{\infty}\dot{H}^{\alpha}}^{k+2}=o_{A(u),R}(1).\label{eq:refining k error 4}
\end{equation}
Since $\|\partial_{yy}(\chi_{R,s}^{2})\|_{L^{\infty}}\lesssim\frac{1}{R_{1}^{2}}$,
the \eqref{eq:refining k error 4} easily follows. We now establish
three remaining estimates \eqref{eq:refining k error 1}, \eqref{eq:refining k error 2},
and \eqref{eq:refining k error 3}.

For \eqref{eq:refining k error 1}, observe that 
\begin{align*}
\chi_{R,s}^{2}(\mathcal{H}u_{y})^{2} & =\chi_{R,s}\mathcal{H}u_{y}\mathcal{H}(\chi_{R,s}u)_{y}-(\chi_{R,s}\mathcal{H}u_{y})[\mathcal{H}\partial_{y},\chi_{R,s}]u\\
 & =[\mathcal{H}(\chi_{R,s}u)_{y}]^{2}-\big(\mathcal{H}(\chi_{R,s}u)_{y}+(\chi_{R,s}\mathcal{H}u_{y})\big)[\mathcal{H}\partial_{y},\chi_{R,s}]u.
\end{align*}
We then use two different expressions of $[\mathcal{H}\partial_{y},\chi_{R,s}]$
\begin{align*}
[\mathcal{H}\partial_{y},\chi_{R,s}]u & =[\mathcal{H},\chi_{R,s}]u_{y}+\mathcal{H}(\chi_{R,s}'u)\\
 & =[\mathcal{H},\chi_{R,s}]u_{y}+[\mathcal{H},\chi_{R,s}']u+\chi_{R,s}'\mathcal{H}u
\end{align*}
to obtain
\begin{align*}
 & [\mathcal{H}(\chi_{R,s}u)_{y}]^{2}-\chi_{R,s}^{2}(\mathcal{H}u_{y})^{2}\\
 & =\mathcal{H}(\chi_{R,s}u)_{y}\big([\mathcal{H},\chi_{R,s}]u_{y}+\mathcal{H}(\chi_{R,s}'u)\big)\\
 & \quad+(\chi_{R,s}\mathcal{H}u_{y})\big([\mathcal{H},\chi_{R,s}]u_{y}+[\mathcal{H},\chi_{R,s}']u+\chi_{R,s}'\mathcal{H}u\big).
\end{align*}
We estimate the above five terms separately. First, observe that
\begin{align*}
 & \Big|\int_{y}\mathcal{H}(\chi_{R,s}u)_{y}[\mathcal{H},\chi_{R,s}]u_{y}\Big|\\
 & \leq\frac{1}{\pi}\int_{r}r(1-r)\Big|\int_{y,z}\chi_{R,s}'''(rz+(1-r)y)\chi_{R,s}(y)u(y)u(z)\Big|.
\end{align*}
Note that the kernel $(y,z)\mapsto\chi_{R,s}'''(rz+(1-r)y)\chi_{R,s}(y)$
has height $\lesssim\frac{1}{R_{1}^{3}}$, $y$-support size $\lesssim R$,
and $z$-support size $\lesssim\frac{R}{r}$. Therefore, 
\[
\Big|\int_{y}\mathcal{H}(\chi_{R,s}u)_{y}[\mathcal{H},\chi_{R,s}]u_{y}\Big|\lesssim_{A(u)}\int_{r}\frac{R}{R_{1}^{3}}r^{\frac{1}{2}}(1-r)=o_{A(u),R}(1).
\]
 Next, we observe by integration by parts that 
\begin{align*}
\int_{y}\mathcal{H}(\chi_{R,s}u)_{y}\mathcal{H}(\chi_{R,s}'u) & =\int_{y}(\chi_{R,s}u)_{y}(\chi_{R,s}'u)\\
 & =\int_{y}(\chi_{R,s}')^{2}u^{2}-\frac{1}{4}\int_{y}(\chi_{R,s}^{2})''u^{2}
\end{align*}
and use H\"older's inequality on each term. Next, again by integration
by parts, we have
\begin{align*}
 & \Big|\int_{y}(\chi_{R,s}\mathcal{H}u_{y})[\mathcal{H},\chi_{R,s}]u_{y}\Big|\\
 & \leq\Big|\int_{y}(\chi_{R,s}\mathcal{H}u)_{y}[\mathcal{H},\chi_{R,s}]u_{y}\Big|+\Big|\int_{y}(\chi_{R,s}'\mathcal{H}u)[\mathcal{H},\chi_{R,s}]u_{y}\Big|.
\end{align*}
and
\begin{align*}
 & \Big|\int_{y}(\chi_{R,s}\mathcal{H}u_{y})[\mathcal{H},\chi_{R,s}']u\Big|\\
 & \leq\Big|\int_{y}(\chi_{R,s}\mathcal{H}u)_{y}[\mathcal{H},\chi_{R,s}']u\Big|+\Big|\int_{y}(\chi_{R,s}'\mathcal{H}u)[\mathcal{H},\chi_{R,s}']u\Big|.
\end{align*}
On each term in the right hand sides, one can proceed similarly as
before. Finally, we observe by integration by parts that 
\[
\int_{y}[\chi_{R,s}\chi_{R,s}'\mathcal{H}u_{y}\mathcal{H}u](y)=-\frac{1}{4}\int_{y}[(\chi_{R,s}^{2})''(\mathcal{H}u)^{2}](y)
\]
This can be treated using H\"older's inequality.

For \eqref{eq:refining k error 2}, observe that 
\[
\chi_{R,s}^{2}u_{y}^{2}=(\chi_{R,s}u)_{y}^{2}-2\chi_{R,s}\chi_{R,s}'u_{y}u-(\chi_{R,s}')^{2}u^{2}.
\]
Using integration by parts, we obtain
\[
\int_{y}\chi_{R,s}^{2}u_{y}^{2}=\int_{y}(\chi_{R,s}u)_{y}^{2}-\frac{1}{2}\int_{y}(\chi_{R,s}^{2})''u^{2}-\int_{y}(\chi_{R,s}')^{2}u^{2}.
\]
We then apply H\"older's inequality to the last two terms in the
right hand side.

For \eqref{eq:refining k error 3}, observe that
\[
\chi_{R,s}^{k+2}u^{k+1}\mathcal{H}u_{y}=(\chi_{R,s}u)^{k+1}\mathcal{H}(\chi_{R,s}u)_{y}-(\chi_{R,s}u)^{k+1}[\mathcal{H}\partial_{y},\chi_{R,s}]u.
\]
Use $[\mathcal{H}\partial_{y},\chi_{R,s}]u=[\mathcal{H},\chi_{R,s}]u_{y}+\mathcal{H}(\chi_{R,s}'u)$,
Schur's test for the commutator term, and H\"older's inequality for
the remaining term. This completes the proof.
\end{proof}
\begin{lem}[Estimate of \eqref{eq:mass derivative error}]
\label{lem:estimate of mass derivative error}We have
\[
\eqref{eq:mass derivative error}=o_{A(u),R}(1).
\]
\end{lem}
\begin{proof}
Observe that 
\[
\eqref{eq:mass derivative error}=-\int_{x,y}\big(u_{x}[\mathcal{H},\Phi(y-\cdot)]u_{x}\big)(x)e(y).
\]
By Lemma \ref{lem:H(1/2) bound for localization}, it suffices to
show that the function
\[
g(y)\coloneqq\int_{x}\big(u_{x}[\mathcal{H},\Phi(y-\cdot)]u_{x}\big)(x)
\]
has $W^{1,\infty}$ norm $o_{A(u),R}(1)$. Observe that
\[
g(y)=\frac{1}{\pi}\int_{r}\int_{x,z}r(1-r)\Phi'''(y-rz-(1-r)x)u(x)u(z).
\]
Whenever $r$ is fixed, the kernel $(x,z)\mapsto\Phi'''(y-rz-(1-r)x)$
has height $\lesssim\frac{1}{R_{1}^{2}}$, $x$-support size $\lesssim\frac{R_{1}}{1-r}$,
and $z$-support size $\lesssim\frac{R_{1}}{r}$. By Schur's test,
we have
\[
\|g\|_{L^{\infty}}\lesssim_{A(u)}\int_{r}r^{\frac{1}{2}}(1-r)^{\frac{1}{2}}R_{1}^{-1}\lesssim_{A(u)}\frac{1}{R_{1}}.
\]
Similarly, we have $\|g\|_{W^{1,\infty}}\lesssim_{A(u)}\frac{1}{R_{1}}$.
\end{proof}
\begin{lem}[Estimate of \eqref{eq:unrefined main term for mass derivative}]
\label{lem:estimate of main term of mass derivative}We have 
\[
\eqref{eq:unrefined main term for mass derivative}=-\int_{s}\int_{x}j[\chi_{R,s}u](x)\int_{y}e[\chi_{R,s}u](y)+o_{A(u),R}(1).
\]
\end{lem}
\begin{proof}
Observe that 
\begin{align*}
j[\chi_{R,s}u] & =\chi_{R,s}^{2}j[u]+2(\tilde{E}_{1}+\tilde{E}_{2}+\tilde{E}_{3})-\frac{2(k+1)}{k+2}\tilde{E}_{4},\\
e[\chi_{R,s}u] & =\chi_{R,s}^{2}e[u]+\frac{1}{2}(\tilde{E}_{1}+\tilde{E}_{2}+\tilde{E}_{3})-\frac{1}{k+2}\tilde{E}_{4}.
\end{align*}
where 
\begin{align*}
\tilde{E}_{1}(s,x) & =\big(\chi_{R,s}u[\mathcal{H},\chi_{R,s}']u\big)(x),\\
\tilde{E}_{2}(s,x) & =\big(\chi_{R,s}u[\mathcal{H},\chi_{R,s}]u_{x}\big)(x),\\
\tilde{E}_{3}(s,x) & =(\chi_{R,s}\chi_{R,s}'u\mathcal{H}u)(x),\\
\tilde{E}_{4}(s,x) & =\big((\chi_{R,s}^{2}-\chi_{R,s}^{k+2})u^{k+2}\big)(x).
\end{align*}
Therefore, it suffices to show that 
\begin{align}
\int_{s}\int_{x}[\chi_{R,s}^{2}u\mathcal{H}u_{x}](x)\int_{y}\tilde{E}_{i}(s,y) & =o_{A(u),R}(1),\label{eq:main times error 1}\\
\int_{s}\int_{x}[\chi_{R,s}^{2}u^{k+2}](x)\int_{y}\tilde{E}_{i}(s,y) & =o_{A(u),R}(1),\label{eq:main times errror 2}\\
\int_{s}\int_{x}\tilde{E}_{i}(s,x)\int_{y}\tilde{E}_{i'}(s,y) & =o_{A(u),R}(1),\label{eq:error times error}
\end{align}
whenever $i,i'=1,2,3,4$.

Let us give a remark here. In the proof of Lemma \ref{lem:estimate of main term of energy derivative},
we used H\"older's inequality in $s$ to reduce ourselves into the
situation when $s$ is fixed. Indeed, because $\rho(x)$ is nonnegative,
we could use Fubini's theorem to obtain
\[
\|\int_{x}\chi_{R,s}^{2}(x)\rho(x)\|_{L_{s}^{1}(\frac{ds}{R})}=\|\chi_{R}^{2}(s)\|_{L_{s}^{1}(\frac{ds}{R})}\int_{x}\rho(x)\lesssim_{A(u)}1.
\]
However, in the case where $\rho$ is replaced by $j$, we do not
have almost everywhere nonnegativity of $j$. If we take $L_{s}^{1}(\frac{ds}{R})$
norm, we only estimate
\[
\|\int_{x}\chi_{R,s}^{2}(x)j(x)\|_{L_{s}^{1}(\frac{ds}{R})}\leq\|\chi_{R}^{2}(s)\|_{L_{s}^{1}(\frac{ds}{R})}\int_{x}|j(x)|\lesssim\int_{x}|j(x)|,
\]
where the right hand side cannot be estimated in terms of the $H^{\frac{1}{2}}$
norm.

Back to the proof, we now estimate \eqref{eq:main times error 1}
and \eqref{eq:main times errror 2}. To avoid the aforementioned difficulty,
we use Lemma \ref{lem:H(1/2) bound for localization} instead. Hence,
it suffices to show that the function
\[
g_{i}(x)\coloneqq\int_{s}\int_{y}\chi_{R}^{2}(x-s)\tilde{E}_{i}(s,y)
\]
has $W^{1,\infty}$ norm bounded by $o_{A(u),R}(1)$ for each $i=1,2,3,4$.

When $i=1$, observe that 
\begin{align*}
g_{1}(x) & =-\frac{1}{\pi}\int_{r}\int_{y,z}K_{1}(x,r,y,z)u(y)u(z),\\
K_{1}(x,r,y,z) & \coloneqq\int_{s}\chi_{R}^{2}(x-s)\chi_{R}(y-s)\chi_{R}''(rz+(1-r)y-s).
\end{align*}
When $x$ and $r$ are fixed, the kernel $(y,z)\mapsto K_{1}(x,r,y,z)$
has height 
\[
|K_{1}(x,r,y,z)|\leq\int_{s}\chi_{R}''(rz+(1-r)y-s)\lesssim\frac{1}{R_{1}^{2}}\cdot\frac{R_{1}}{R}\lesssim\frac{1}{RR_{1}}.
\]
It has $y$-support $\lesssim R$ (by ignoring $\chi_{R}''$ part),
and $z$-support $\lesssim\frac{R}{r}$. By Schur's test and Minkowski's
inequality in $r$, we have
\[
\|g_{1}\|_{L^{\infty}}\lesssim_{A(u)}\int_{r}\frac{1}{R_{1}}r^{-\frac{1}{2}}=o_{A(u),R}(1).
\]
Note that the integral formula of $\partial_{x}g_{1}$ has kernel
$\partial_{x}K_{1}$. Because $\partial_{x}K_{1}$ obeys an even better
estimate, we conclude that $\|g_{1}\|_{W^{1,\infty}}=o_{A(u),R}(1)$.

When $i=2$, observe that 
\begin{align*}
g_{2}(x) & =-\frac{1}{\pi}\int_{r}\int_{y,z}K_{2}(x,r,y,z)u(y)u(z),\\
K_{2}(x,r,y,z) & \coloneqq\int_{s}r\chi_{R}^{2}(x-s)\chi_{R}(y-s)\chi_{R}''(rz+(1-r)y-s).
\end{align*}
One can proceed as exactly same way as when $i=1$.

When $i=3$, observe that 
\begin{align*}
g_{3}(x) & =\int_{y}\int_{s}\chi_{R}^{2}(x-s)[\chi_{R}\chi_{R}'](y-s)[u\mathcal{H}u](y).
\end{align*}
Using H\"older's inequality in $y$,
\[
|g_{3}(x)|\leq\Big\|\int_{s}\chi_{R}^{2}(x-s)[\chi_{R}\chi_{R}'](y-s)\Big\|_{L_{y}^{\infty}}\|u\mathcal{H}u\|_{L^{1}}\lesssim\frac{1}{R}\|u\|_{L^{2}}^{2}.
\]
Similarly, we have a better estimate for $\partial_{x}g_{3}$, and
so $\|g_{3}\|_{W^{1,\infty}}=o_{A(u),R}(1)$.

When $i=4$, we have
\[
g_{4}(x)=\int_{y}\int_{s}\chi_{R}^{2}(x-s)[\chi_{R}^{2}-\chi_{R}^{k+2}](y-s)u^{k+2}(y).
\]
Using H\"older's inequality in $y$ as before, 
\[
|g_{4}(x)|\leq\Big\|\int_{s}\chi_{R}^{2}(x-s)[\chi_{R}^{2}-\chi_{R}^{k+2}](y-s)\Big\|_{L_{y}^{\infty}}\|u\|_{L^{k+2}}^{k+2}\lesssim\frac{R_{1}}{R}\|u\|_{L^{k+2}}^{k+2}.
\]
For $\partial_{x}g_{4}$, we have a better estimate. Thus $\|g_{4}\|_{W^{1,\infty}}=o_{A(u),R}(1)$.

Next, we turn to the proof of \eqref{eq:error times error}. By H\"older's
inequality in $s$, it suffices to show that 
\[
\Big\|\int_{y}\tilde{E}_{i}(s,y)\Big\|_{L_{s}^{2}(\frac{ds}{R})}=o_{A(u),R}(1).
\]
When $i=1$, we observe that 
\begin{align*}
\int_{y}\tilde{E}_{1}(s,y) & =\int_{r}\int_{y,z}\tilde{K}_{1}(r,s,y,z)u(y)u(z),\\
\tilde{K}_{1}(r,s,y,z) & \coloneqq\chi_{R}(y-s)\chi_{R}''(rz+(1-r)y-s)
\end{align*}
Whenever $t$ and $s$ are fixed, the kernel $(y,z)\mapsto\tilde{K}_{1}(r,s,y,z)$
has height $\lesssim\frac{1}{R_{1}^{2}}$, $y$-support size $\lesssim R$
(ignoring $\chi_{R}''$), and $z$-support size $\lesssim\frac{R_{1}}{r}$.
Therefore, by Schur's test,
\[
\Big\|\int_{y,z}\tilde{K}_{1}(r,s,y,z)u(y)u(z)\Big\|_{L_{s}^{\infty}}\lesssim_{A(u)}R_{1}^{-\frac{3}{2}}R^{\frac{1}{2}}r^{-\frac{1}{2}}.
\]
On the other hand, whenever $r$ is fixed, the kernel $(y,z)\mapsto\|\tilde{K}_{1}(t,s,y,z)\|_{L_{s}^{1}(\frac{ds}{R})}$
has height $\lesssim\frac{1}{RR_{1}}$ and $z$-support size $\lesssim\frac{R}{r}$.
The $y$-support of the kernel
\[
\|\tilde{K}_{1}(t,s,y,z)\|_{L_{s}^{1}(\frac{ds}{R})}=\int_{s}\chi_{R}(y-s)|\chi_{R}''(rz+(y-s)-ry)|
\]
has $\lesssim\frac{R}{r}$. Therefore, 
\begin{align*}
\Big\|\int_{y,z}\tilde{K}_{1}(r,s,y,z)u(y)u(z)\Big\|_{L_{s}^{1}(\frac{ds}{R})} & \leq\int_{y,z}\|\tilde{K}_{1}(r,s,y,z)\|_{L_{s}^{1}(\frac{ds}{R})}|u(y)||u(z)|\\
 & \lesssim_{A(u)}(RR_{1})^{-1}r^{-1}.
\end{align*}
Interpolating above two estimates, we have 
\begin{align*}
\Big\|\int_{y}\tilde{E}_{1}(s,y)\Big\|_{L_{s}^{2}(\frac{ds}{R})} & \lesssim\int_{r}\Big\|\int_{y,z}\tilde{K}_{1}(r,s,y,z)u(y)u(z)\Big\|_{L_{s}^{2}(\frac{ds}{R})}\\
 & \lesssim_{A(u)}\int_{r}R^{-\frac{1}{4}}R_{1}^{-\frac{5}{4}}r^{-\frac{3}{4}}\\
 & \lesssim_{A(u)}R^{-\frac{1}{4}}R_{1}^{-\frac{5}{4}}.
\end{align*}
Note that the case of $i=2$ can be treated in a similar manner.

When $i=3$, observe that 
\begin{align*}
\Big\|\int_{y}\tilde{E}_{3}(s,y)\Big\|_{L_{s}^{1}(\frac{ds}{R})} & \lesssim\|\tilde{E}_{3}(s,y)\|_{L_{s,y}^{1}(\frac{dsdy}{R})}\lesssim_{A(u)}\frac{1}{R},\\
\Big\|\int_{y}\tilde{E}_{3}(s,y)\Big\|_{L_{s}^{\infty}(\frac{ds}{R})} & \lesssim_{A(u)}\frac{1}{R_{1}}.
\end{align*}
Interpolating the above two estimates, we have
\[
\Big\|\int_{y}\tilde{E}_{3}(s,y)\Big\|_{L_{s}^{2}(\frac{ds}{R})}\lesssim_{A(u)}R^{-\frac{1}{2}}R_{1}^{-\frac{1}{2}}.
\]
When $i=4$, observe that 
\begin{align*}
\Big\|\int_{y}\tilde{E}_{4}(s,y)\Big\|_{L_{s}^{1}(\frac{ds}{R})} & \lesssim\|\tilde{E}_{4}(s,y)\|_{L_{s,y}^{1}(\frac{dsdy}{R})}\lesssim_{A(u)}\frac{R_{1}}{R},\\
\Big\|\int_{y}\tilde{E}_{4}(s,y)\Big\|_{L_{s}^{\infty}(\frac{ds}{R})} & \lesssim_{A(u)}1.
\end{align*}
Interpolating above two estimates, we have
\[
\Big\|\int_{y}\tilde{E}_{4}(s,y)\Big\|_{L_{s}^{2}(\frac{ds}{R})}\lesssim_{A(u)}R_{1}^{\frac{1}{2}}R^{-\frac{1}{2}}.
\]
This completes the proof.
\end{proof}

\appendix

\section{\label{sec:proof of linear profile decomposition}Proof of Proposition
\ref{prop:linear profile decomposition}}

The linear profile decomposition (Proposition \ref{prop:linear profile decomposition})
is now fairly standard. In fact, it is very similar to that in Holmer-Roudenko
\cite[Lemma 2.1]{HolmerRoudenko2008}. We include it here for sake
of completeness. As a last remark, we are not sure whether Proposition
\ref{prop:linear profile decomposition} with $k=4$ holds or not.
The $L_{x}^{4}L_{t}^{\infty}$ estimate is a genuine endpoint estimate
among local smoothing estimates, so the usual interpolation argument
does not work.
\begin{proof}[Proof of Proposition \ref{prop:linear profile decomposition}]
As we can pass to a subsequence if necessary, combining with the
usual diagonal trick, we may assume that all the limits what follow
exist. Let $C$ be the implicit constant for linear local smoothing
estimates.

We proceed by induction. For notational convenience, let $w_{n}^{0}\coloneqq u_{n}$.
Suppose that we have constructed $\{\phi^{j},x_{n}^{j},t_{n}^{j}\}$
satisfying all the properties for $j=1,\dots,J-1$. We shall construct
$\{\phi^{J},x_{n}^{J},t_{n}^{J}\}$ as follows. Let 
\[
A_{J}\coloneqq\lim_{n\to\infty}\|w_{n}^{J-1}\|_{H^{\frac{1}{2}}}\quad\text{and}\quad c_{J}\coloneqq\lim_{n\to\infty}\|V(t)w_{n}^{J-1}\|_{L_{x}^{k}L_{t}^{\infty}}.
\]
If $c_{J}=0$, then we set $\phi^{j}=0$, $t_{n}^{j}=x_{n}^{j}=0$
for all $j\geq J$, and stop the procedure. Henceforth, we assume
that $c_{J}\neq0$. For some large $r>1$ chosen later, consider a
Schwartz function $\chi$ such that $|\hat{\chi}|\leq1$,  $\hat{\chi}(\xi)=1$
when $\frac{1}{r}\leq|\xi|\leq r$, and $\hat{\chi}(\xi)=0$ when
$|\xi|\leq\frac{1}{2r}$ or $|\xi|\geq2r$. Convoluting $w_{n}^{J-1}$
with $\chi$, we have
\begin{align*}
\|\chi\ast V(t)w_{n}^{J-1}\|_{L_{x}^{k}L_{t}^{\infty}} & \geq\|V(t)w_{n}^{J-1}\|_{L_{x}^{k}L_{t}^{\infty}}-\|V(t)w_{n}^{J-1}-\chi\ast V(t)w_{n}^{J-1}]\|_{L_{x}^{k}L_{t}^{\infty}}\\
 & \geq\|V(t)w_{n}^{J-1}\|_{L_{x}^{k}L_{t}^{\infty}}-C\|w_{n}^{J-1}-\chi\ast w_{n}^{J-1}\|_{\dot{H}^{s_{k}}}.
\end{align*}
As $n\to\infty$, we have
\[
\lim_{n\to\infty}\|\chi\ast V(t)w_{n}^{J-1}\|_{L_{x}^{k}L_{t}^{\infty}}\geq c_{J}-Cr^{-\frac{1}{k}}A_{J}.
\]
On the other hand, an easy interpolation yields (this is where the
assumption $k>4$ is required)
\begin{align*}
\lim_{n\to\infty}\|\chi\ast V(t)w_{n}^{J-1}\|_{L_{x}^{k}L_{t}^{\infty}} & \leq\lim_{n\to\infty}\big(\|\chi\ast V(t)w_{n}^{J-1}\|_{L_{x}^{4}L_{t}^{\infty}}^{\frac{4}{k}}\|\chi\ast V(t)w_{n}^{J-1}\|_{L_{x,t}^{\infty}}^{1-\frac{4}{k}}\big)\\
 & \leq(CA_{J})^{\frac{4}{k}}\lim_{n\to\infty}\|\chi\ast V(t)w_{n}^{J-1}\|_{L_{x,t}^{\infty}}^{1-\frac{4}{k}}.
\end{align*}
Therefore,
\[
\lim_{n\to\infty}\|\chi\ast V(t)w_{n}^{J-1}\|_{L_{x,t}^{\infty}}\geq c_{J}(CA_{J})^{-\frac{4}{k-4}}-(CA_{J})r^{-\frac{1}{k-4}}.
\]
If we choose $r>1$ satisfying $(CA_{J})r^{-\frac{1}{k-4}}\leq\frac{1}{3}c_{J}(CA_{J})^{-\frac{4}{k-4}}$,
i.e.
\[
r\leq(CA_{J})^{k}c_{J}^{-(k-4)},
\]
then we can choose $t_{n}^{J},x_{n}^{J}$ such that 
\[
|[\chi\ast V(-t_{n}^{J})w_{n}^{J-1}](x_{n}^{J})|\geq\frac{c_{J}}{2}(CA_{J})^{-\frac{4}{k-4}}
\]
for all large $n$. Let $\phi^{J}$ be the weak $H^{\frac{1}{2}}$-limit
of $[V(-t_{n}^{J})w_{n}^{J-1}](x_{n}^{J})$ and $w_{n}^{J}$ be $w_{n}^{J-1}-V(t_{n}^{J})\phi^{J}(\cdot-x_{n}^{J})$.
Note that $V(-t_{n}^{J})w_{n}^{J}(\cdot+x_{n}^{J})\rightharpoonup0$
in $H^{\frac{1}{2}}$. We can give a lower bound for $H^{\frac{1}{2}}$-norm
of the profile $\phi^{J}$. From
\[
c_{J}A_{J}^{-\frac{4}{k-4}}\lesssim\lim_{n\to\infty}|[\chi\ast V(-t_{n}^{J})u_{n}](x_{n}^{J})|=|[\chi\ast\phi^{J}](0)|\lesssim\|\chi\|_{H^{-s_{k}}}\|\phi^{J}\|_{H^{s_{k}}},
\]
we have 
\[
\|\phi^{J}\|_{H^{\frac{1}{2}}}\gtrsim c_{J}A_{J}^{-\frac{4}{k-4}}r^{-\frac{1}{k}}\gtrsim c_{J}^{\frac{4}{k}}A_{J}^{-\frac{k}{k-4}}.
\]

We now prove asymptotic orthogonality in $\dot{H}^{s}$. By the definition
of $\phi^{J}$ and $w_{n}^{J}$, 
\begin{align*}
\|w_{n}^{J-1}\|_{\dot{H}^{s}}^{2} & =\|w_{n}^{J}\|_{\dot{H}^{s}}^{2}+\|V(t_{n}^{J})\phi^{J}(\cdot-x_{n}^{J})\|_{\dot{H}^{s}}^{2}+\langle w_{n}^{J},V(t_{n}^{J})\phi^{J}(\cdot-x_{n}^{J})\rangle_{\dot{H}^{s}}\\
 & =\|w_{n}^{J}\|_{\dot{H}^{s}}^{2}+\|\phi^{J}\|_{\dot{H}^{s}}^{2}+\langle V(-t_{n}^{J})w_{n}^{J}(\cdot+x_{n}^{J}),\phi^{J}\rangle_{\dot{H}^{s}}.
\end{align*}
As $n\to\infty$, we have
\[
\lim_{n\to\infty}\Big[\|w_{n}^{J-1}\|_{\dot{H}^{s}}^{2}-\|w_{n}^{J}\|_{\dot{H}^{s}}^{2}-\|\phi^{J}\|_{\dot{H}^{s}}^{2}\Big]=0.
\]
As the induction hypothesis tells
\[
\lim_{n\to\infty}\Big[\|u_{n}\|_{\dot{H}^{s}}^{2}-\sum_{j=1}^{J-1}\|\phi^{j}\|_{\dot{H}^{s}}^{2}-\|w_{n}^{J-1}\|_{\dot{H}^{s}}^{2}\big]=0,
\]
we conclude that 
\[
\lim_{n\to\infty}\Big[\|u_{n}\|_{\dot{H}^{s}}^{2}-\sum_{j=1}^{J}\|\phi^{j}\|_{\dot{H}^{s}}^{2}-\|w_{n}^{J}\|_{\dot{H}^{s}}^{2}\big]=0.
\]
This completes the proof of asymptotic orthogonality in $\dot{H}^{s}$.

We now show asymptotic separation of parameters. Combining with the
induction hypothesis, it suffices to show that for each $j<J$, we
have $|t_{n}^{j}-t_{n}^{J}|+|x_{n}^{j}-x_{n}^{J}|\to\infty$. Suppose
not; choose the maximal $j<J$ such that $t_{n}^{j}-t_{n}^{J}$ and
$x_{n}^{j}-x_{n}^{J}$ converge for some subsequence. Then,
\[
V(-t_{n}^{j})w_{n}^{J}(\cdot+x_{n}^{j})=V(-t_{n}^{j})w_{n}^{j}(\cdot+x_{n}^{j})+\sum_{j'=j+1}^{J-1}V(t_{n}^{j'}-t_{n}^{j})\phi^{j}(\cdot+x_{n}^{j}-x_{n}^{j'})+\phi^{j}.
\]
By the induction hypothesis, we know that $|t_{n}^{j'}-t_{n}^{j}|+|x_{n}^{j}-x_{n}^{j'}|\to\infty$
for all $j'<J$ so the summation part of the above display weakly
converges to zero. By the construction, $V(-t_{n}^{j})w_{n}^{j}(\cdot+x_{n}^{j})$
converges weakly to zero. Moreover, since we assumed that $t_{n}^{j}-t_{n}^{J}$
and $x_{n}^{j}-x_{n}^{J}$ converge, $V(-t_{n}^{j})w_{n}^{J}(\cdot+x_{n}^{j})$
can be well approximated by $V(t_{0})V(-t_{n}^{J})w_{n}^{J}(\cdot+x_{n}^{J}-x_{0})$
for some fixed $x_{0}$ and $t_{0}$, so it converges weakly to zero.
Therefore, every term in the above display except $\phi^{j}$ weakly
converges to zero. This yields a contradiction.

We now show asymptotic vanishing of the weak limit. Combining with
the induction hypothesis, it suffices to show that for each $j=1,\dots,J$,
we have $V(-t_{n}^{j})w_{n}^{J}(\cdot+x_{n}^{j})\rightharpoonup0$.
We express
\[
V(-t_{n}^{j})w_{n}^{J}(\cdot+x_{n}^{j})=V(-t_{n}^{j})w_{n}^{j}(\cdot+x_{n}^{j})+\sum_{j'=j+1}^{J}V(t_{n}^{j'}-t_{n}^{j})\phi^{j}(\cdot+x_{n}^{j}-x_{n}^{j'}).
\]
The first term of the RHS vanishes by the construction. For the remaining
term, we use separation of parameters.

We finally show asymptotic vanishing of remainder. Because 
\[
A_{1}^{-\frac{2k}{k-4}}\sum_{j}c_{j}^{\frac{8}{k}}\leq\sum_{j}c_{j}^{\frac{8}{k}}A_{j}^{-\frac{2k}{k-4}}\lesssim\sum_{j}\|\phi^{j}\|_{H^{\frac{1}{2}}}^{2}<\infty,
\]
we should have $c_{j}\to0$.
\end{proof}
\bibliographystyle{plain}
\bibliography{References}

\end{document}